\documentclass[11pt]{amsart}
\usepackage{latexsym}
\usepackage{amsmath}
\usepackage{enumerate}
\makeatletter
\usepackage[notintoc]{nomencl}
\makenomenclature
\@namedef{subjclassname@2010}{

\textup{2010} Mathematics Subject Classification}
\usepackage{amssymb, amsmath,amsthm, amsfonts}
\usepackage{mathtools}

\usepackage[english]{babel}
\usepackage[utf8]{inputenc}
\usepackage{comment}
\usepackage{color}
\usepackage{mathrsfs}
\usepackage{enumitem}
\usepackage[pdftex,citecolor=green,linkcolor=red]{hyperref}

\usepackage{multicol}

\newtheorem{innercustomgeneric}{\customgenericname}
\providecommand{\customgenericname}{}
\newcommand{\newcustomtheorem}[2]{%
  \newenvironment{#1}[1]
  {%
   \renewcommand\customgenericname{#2}%
   \renewcommand\theinnercustomgeneric{##1}%
   \innercustomgeneric
  }
  {\endinnercustomgeneric}
}

\newcustomtheorem{customthm}{Theorem}
\newtheorem{thm}{Theorem}[section]
\newtheorem{prop}[thm]{Proposition}

\newtheorem{cor}[thm]{Corollary}
\newtheorem{lemma}[thm]{Lemma}
\newtheorem{lem}[thm]{Lemma}
\theoremstyle{definition}
\newtheorem*{note1}{Important note}
\newtheorem{remark}{Remark}[section]
\newtheorem{def1}{Definition}[section]

\newcommand{\ra}{\rightarrow}
\newcommand{\bk}{\backslash}
\newcommand{\mc}{\mathcal}
\newcommand{\mb}{\mathbb}
\newcommand{\sg}{\sigma}

\newcommand{\llf}{\left\lfloor}

\newcommand{\e}{\varepsilon}

\newcommand{\rrf}{\right\rfloor}

\newcommand{\asum}{\sideset{}{^{\ast}}\sum}

\renewcommand{\bar}{\overline}

\renewcommand{\Re}{\textnormal{Re}}
\renewcommand{\Im}{\textnormal{Im}}

\renewcommand{\phi}{\varphi}

\newcommand{\cond}{\textnormal{cond}}
\renewcommand{\pmod}[1]{\ (\mathrm{mod}\ #1)}
\newcommand{\legendre}[2]{\genfrac{(}{)}{}{}{#1}{#2}}

\begin{document}
\title[Multiplicative functions in short arithmetic progressions]{Multiplicative functions in short arithmetic progressions}
\author
{Oleksiy Klurman}
\address{School of Mathematics,
University of Bristol, Woodland Road, Bristol,  UK}
\email{lklurman@gmail.com}

\author{Alexander P. Mangerel}
\address{Department of Mathematical Sciences, Durham University, Upper Mountjoy Campus, Stockton Road, Durham, UK}
\email{smangerel@gmail.com}

\author{Joni Ter\"{a}v\"{a}inen}
\address{Department of Mathematics and Statistics, University of Turku, Turku, Finland}
\email{joni.p.teravainen@gmail.com}

\begin{abstract}
We study for bounded multiplicative functions $f$ sums of the form
\begin{align*}
\sum_{\substack{n\leq x\\n\equiv a\pmod q}}f(n),    
\end{align*}
establishing that their variance over residue classes $a \pmod q$ is small as soon as $q=o(x)$, for almost all moduli $q$, with a nearly power-saving exceptional set of $q$. This improves and generalizes  previous results of Hooley on Barban--Davenport--Halberstam-type theorems for such $f$, and moreover our exceptional set is essentially optimal unless one is able to make progress on certain well-known conjectures. We are nevertheless able to prove stronger bounds for the number of the exceptional moduli $q$ in the cases where $q$ is restricted to be either smooth or prime, and conditionally on GRH we show that our variance estimate is valid for every $q$.

These results are special cases of a ''hybrid result'' that we establish that works for sums of $f$ over almost all short intervals and arithmetic progressions simultaneously, thus generalizing the Matom\"aki--Radziwi\l{}\l{} theorem on multiplicative functions in short intervals.

We also consider the maximal deviation of $f$ over \emph{all} residue classes $a\pmod q$ in the square root range $q\leq x^{1/2-\varepsilon}$, and show that it is small for ''smooth-supported'' $f$, again apart from a nearly power-saving set of exceptional $q$, thus providing a smaller exceptional set than what follows from Bombieri--Vinogradov-type theorems. 

As an application of our methods, we consider Linnik-type problems for products of exactly three primes, and in particular prove a ternary approximation to a conjecture of Erd\H{o}s on representing every element of the multiplicative group $\mathbb{Z}_p^{\times}$ as the product of two primes less than $p$.
\end{abstract}

\maketitle

\begin{flushright}
\textit{To the memory of Christopher Hooley}
\end{flushright}

\setcounter{tocdepth}{1}

\tableofcontents

\section{Main theorems}

Let $\mb{U}\coloneqq \{z\in \mathbb{C}:\, |z|\leq 1\}$ denote the unit disc of the complex plane, and let $f \colon \mb{N} \ra \mb{U}$ be a 1-bounded multiplicative function. In this paper we study sums of the form
\begin{align}\label{eqq0}
\sum_{\substack{n\leq x\\n\equiv a\pmod{q}}}f(n)
\end{align}
with $(a,q)=1$ and with the modulus $1 \leq q \leq x$ being very large as a function of $x$. We call such arithmetic progressions \emph{short}, since the number of terms is $\sim x/q$, which is assumed to grow slowly with $x$. 

Our main results concern the \emph{deviation} of multiplicative functions $f\colon \mathbb{N}\to \mathbb{U}$ in residue classes in the square-root range $q\leq x^{1/2-\varepsilon}$, as well as their \emph{variance} in residue classes in the full range $q=o(x)$. Here by deviation we mean
\begin{align}\label{eqq1}
\max_{a\in \mathbb{Z}_{q}^{\times}}\Big|\sum_{\substack{n\leq x\\n\equiv a\pmod q}}f(n)-\frac{\chi_1(a)}{\phi(q)}\sum_{n\leq x}f(n)\bar{\chi_1}(n)\Big|, 
\end{align}
where $\mathbb{Z}_{q}^{\times}$ is the set of invertible residue classes $\hspace{-0.1cm}\pmod q$, and by variance we mean 
\begin{align}\label{eqq2}
\asum_{a \pmod{q}}\Big|\sum_{\substack{n\leq x\\n\equiv a\pmod q}}f(n)-\frac{\chi_1(a)}{\phi(q)}\sum_{n\leq x}f(n)\bar{\chi_1}(n)\Big|^2, 
\end{align}
where $\sum_{a(q)}^{*}$ denotes a sum over reduced residue classes $\hspace{-0.1cm}\pmod q$. 
The character $\chi_1\pmod q$ here is chosen\footnote{If there is more than one minimizing character, we may choose any of these.} such that the map $\chi \mapsto \inf_{|t| \leq \log x}\mb{D}_q(f,\chi(n) n^{it};x)$ is minimized, where, given $f,g\colon \mb{N} \ra \mb{U}$ we define
\begin{align}\label{eqq123}
\mathbb{D}_q(f,g;x)\coloneqq \Big(\sum_{\substack{p\leq x\\p\nmid q}}\frac{1-\textnormal{Re}(f(p)\overline{g(p)})}{p}\Big)^{1/2}
\end{align}
to be the pretentious distance function of Granville and Soundararajan (see, e.g.,~\cite[p. 3]{bgs}). As a consequence of a well-known theorem of Hal\'{a}sz, it can be shown that any other character $\chi \neq \chi_1\pmod q$ has small correlation sums $\sum_{n \leq x} f(n)\bar{\chi}(n)$, and so informally we may think of $\chi_1$ as a character that ``correlates the most'' with $f$ among all the characters $\hspace{-0.1cm}\pmod{q}$.

Comparing the sum~\eqref{eqq0} to the main term $\chi_1(a)/\varphi(q)\cdot \sum_{n\leq x}f(n)\overline{\chi}_1(n)$ is natural, since if, in fact, $f$ ``correlates" significantly with some Dirichlet character $\chi,$ then we expect
\begin{align*}
\sum_{\substack{n\leq x\\n\equiv a\pmod q}}f(n)\approx \frac{\chi(a)}{\phi(q)}\sum_{n\leq x}f(n)\bar{\chi}(n).    
\end{align*}

In this paper, we develop a systematic approach to estimating weighted character sums $\sum_{n\le x}f(n)\bar{\chi(n)}n^{it}$ for the wide range of parameters $t,q=O(x)$, and deduce numerous estimates for~\eqref{eqq1} and~\eqref{eqq2}.

\subsection{Results for prime moduli}

For many problems on well-distribution in arithmetic progressions one can obtain stronger results for prime moduli than for general moduli (see, for example,~\cite{green-bombieri}, \cite{fkm}); the same is true in our setting.   

Our first main result concerns the variance~\eqref{eqq2} in the range where $x/q$ tends to infinity very slowly. It is motivated by the groundbreaking work of Matom\"aki and Radziwi\l{}\l{}~\cite{mr-annals}, which produces a comparable result for multiplicative functions in short \emph{intervals}. 

All the constants in this paper implied by the $\ll$ notation will be absolute unless otherwise indicated.

\begin{cor}\label{MRAPTHM-prime} Let $1\leq Q\leq x/10$ and $(\log(x/Q))^{-1/200}\leq \varepsilon\leq 1$. Then there exists a set $[1,x^{\varepsilon^{200}}]\cap \mathbb{Z}\subset \mathcal{Q}_{x,\varepsilon}\subset [1,x]\cap \mathbb{Z}$ with $|[1,Q] \setminus \mathcal{Q}_{x,\varepsilon}|\ll (\log x)^{\varepsilon^{-200}}$ such that the following holds.

Let $p\in \mathcal{Q}_{x,\varepsilon}\cap [1,Q]$ be a prime. Let $f\colon \mb{N} \ra \mb{U}$ be a multiplicative function. Let $\chi_1$ be a character $\hspace{-0.1cm}\pmod p$ minimizing the distance $\inf_{|t|\leq \log x}\mathbb{D}_p(f,\chi(n)n^{it};x)$. Then  we have
\begin{align}\label{form23}
 \asum_{a \pmod{p}} \Big|\sum_{\substack{n \leq x \\ n \equiv a \pmod{p}}} f(n) - \frac{\chi_1(a)}{\phi(p)}\sum_{n \leq x} f(n)\bar{\chi_1}(n)\Big|^2 \ll \varepsilon \frac{x^2}{p}.
\end{align}
Moreover, assuming GRH,~\eqref{form23} holds for all $p\in [1,Q]$.
\end{cor}

\begin{remark}
Applying Hal\'asz's theorem (Lemma~\ref{le_hal_hyb}), we see that in Corollary~\ref{MRAPTHM-prime} (as well as in our other results to follow) the main term $(\chi_1(a)/\phi(q))\cdot\sum_{n\leq x}f(n)\overline{\chi_1}(n)$ can be deleted from the variance, unless 
\begin{align}\label{eqq124}
\inf_{|t|\leq \log x}\mathbb{D}_q(f,\chi_1(n)n^{it}; x)^2\leq 2\log \frac{1}{\varepsilon}.    
\end{align} 
In particular, if GRH holds, then by the pretentious triangle inequality we see that~\eqref{eqq124} can only hold if $\chi_1$ is induced by $\chi'$, where $\chi'$ is the primitive character of conductor $\leq Q$ that minimizes $\inf_{|t|\leq \log x}\mathbb{D}(f,\chi(n)n^{it};x)$ (without assuming GRH, the situation is somewhat more complicated; cf. Subsection~\ref{subsec: vino}). 
\end{remark}

We refer to Section~\ref{sec: optimal} for a discussion of the strength of this theorem as well as that of our other theorems.

\subsection{Smooth-supported functions in the square root range}

We are also able to obtain a result on the deviation~\eqref{eqq1} of multiplicative functions in all arithmetic progressions $n\equiv a\pmod q$ in the ``middle range'' $q\leq x^{1/2-o(1)}$. This supports the well-known analogy between results for all moduli in the middle range $q\leq x^{1/2-o(1)}$ and almost all moduli in the large range $x^{1-\varepsilon}\leq q\leq x^{1-o(1)}$ (an example of this analogy is provided by the theorems of Bombieri--Vinogradov and Barban--Davenport--Halberstam).

Transferring results from the almost all case to the case of all arithmetic progressions requires a bilinear structure in our sums. In our case, we introduce this bilinear structure by considering multiplicative functions $f$ supported on \emph{smooth} (otherwise known as \emph{friable}) numbers.

\begin{thm}\label{BVTHM} Let $\eta>0$ be fixed. Let $x\geq 10$, $(\log x)^{-1/200}\leq \varepsilon\leq 1$, and $Q\leq x^{1/2-100\eta}$. There is a set $[1,x^{\varepsilon^{200}}]\cap \mathbb{Z}\subset \mathcal{Q}_{x,\varepsilon}\subset [1,x]\cap \mathbb{Z}$ with $|[1,Q]\setminus \mathcal{Q}_{x,\varepsilon}|\ll Qx^{-\varepsilon^{200}}$ such that the following holds.

Let $q\in \mathcal{Q}_{x,\varepsilon}\cap [1,Q]$. Let $f \colon \mb{N} \ra \mb{U}$ be a multiplicative function supported on $x^{\eta}$-smooth numbers. Let $\chi_1$ be a character $\hspace{-0.1cm}\pmod q$ minimizing the distance $\inf_{|t|\leq \log x}\mathbb{D}_q(f,\chi(n)n^{it};x)$. Then we have
\begin{align}\label{eqq11}
\max_{a\in \mathbb{Z}_q^{\times}}\Big|\sum_{\substack{n\leq x\\n\equiv a\pmod q}}f(n)-\frac{\chi_1(a)}{\phi(q)}\sum_{n\leq x}f(n)\overline{\chi_1}(n)\Big|\ll \varepsilon \frac{x}{q}. \end{align}
Furthermore, if $\mc{Q}^{\prime}$ is any subset of $[1,Q]$ whose elements are pairwise coprime, then we have the bound $|\mc{Q}^{\prime}\setminus \mathcal{Q}_{x,\varepsilon}|\ll (\log x)^{\varepsilon^{-200}}$. Moreover, assuming GRH,~\eqref{eqq11} holds for all $q\in [1,Q]$.
\end{thm}

\subsection{Results for smooth moduli}

In addition to primality of moduli, we can also leverage their smoothness (see~\cite{zhang},~\cite{polymath8a} for some other level of distribution estimates leveraging the smoothness of moduli). For moduli $q$ that are $q^{\varepsilon'}$-smooth, we may prove a variant of Corollary~\ref{MRAPTHM-prime} without any exceptional moduli at all, but with the disadvantage that the upper bound for the variance is weaker (and possibly trivial) when $q$ has abnormally many small prime divisors. To this end, we make the following definition.

\begin{def1}\label{def-typical} We say that an integer $q\geq 1$ is \emph{$y$-typical} if
\begin{align*}
|\{p\leq z:\, p\mid q\}|\leq \frac{1}{100}\pi(z)\quad \textnormal{for all}\quad z\geq y.   
\end{align*}
\end{def1}

Theorem~\ref{MRAPTHM-smooth} below analogizes Corollary~\ref{MRAPTHM-prime} for smooth moduli that are, in addition, $(x/Q)^{\e^2}$-typical numbers.
A simple argument (see Lemma~\ref{le_delta}) shows that all $q\leq x$ are such numbers if $Q=o(x/(\log x)^{1/\varepsilon^2})$, and otherwise the number of $q\leq Q$ that are not $(x/Q)^{\varepsilon^2}$-typical is bounded by $\ll Q\exp(-10^{-4}(x/Q)^{\varepsilon^2})$.

\begin{thm}\label{MRAPTHM-smooth}
Let $1\leq Q\leq x/10$, $(\log(x/Q))^{-1/200}\leq \varepsilon\leq 1$, and $\varepsilon'=\exp(-\varepsilon^{-3})$. Let $q\leq Q$ be $q^{\varepsilon'}$-smooth and $(x/Q)^{\varepsilon^2}$-typical. Let $f\colon \mathbb{N}\to \mathbb{U}$ be a multiplicative function. Let $\chi_1\pmod q$ be a character minimizing the distance $\inf_{|t|\leq \log x}\mathbb{D}_q(f,\chi(n)n^{it};x)$. Then we have 
\begin{align*}
\asum_{a\pmod q}\Big|\sum_{\substack{n\leq x\\n\equiv a \pmod{q}}}f(n)-\frac{\chi_1(a)}{\varphi(q)}\sum_{n\leq x}f(n)\bar{\chi_1}(n)\Big|^2\ll \varepsilon \phi(q) \Big(\frac{x}{q}\Big)^2.    
\end{align*}
\end{thm}

We note that the need to restrict to typical moduli arises naturally in our proof and is present also in other works (formulated in slightly different terms), see e.g.~\cite{radziwill-mobius},~\cite{gs-uncertainty}.  See also Subsection~\ref{subsec: typical} for a discussion of the necessity of this assumption.

\subsection{General moduli}

We may now state a result for general moduli $q$ that are not required to be prime or smooth. In this case we obtain the desired bound for the variance~\eqref{eqq2} for all typical moduli outside a nearly power-saving exceptional set. 

\begin{thm} \label{MRAPTHM}
Let $1\leq Q\leq x/10$ and $(\log(x/Q))^{-1/200} \leq\varepsilon\leq 1$. Then there exists a set $[1,x^{\varepsilon^{200}}]\cap \mathbb{Z}\subset \mathcal{Q}_{x,\varepsilon}\subset [1,x]\cap \mathbb{Z}$ with $|[1,Q]\setminus \mathcal{Q}_{x,\varepsilon}|\ll Qx^{-\varepsilon^{200}}$ such that the following holds.

Let $q\in \mathcal{Q}_{x,\varepsilon}\cap [1,Q]$ be $(x/Q)^{\varepsilon^2}$-typical. Let $f\colon \mb{N} \ra \mb{U}$ be a multiplicative function. Let $\chi_1$ be a character $\hspace{-0.1cm}\pmod q$ minimizing the distance $\inf_{|t|\leq \log x}\mathbb{D}_q(f,\chi(n)n^{it};x)$. Then we have
\begin{align}\label{form22}
 \asum_{a \pmod{q}} \Big|\sum_{\substack{n \leq x \\ n \equiv a \pmod{q}}} f(n) - \frac{\chi_1(a)}{\phi(q)}\sum_{n \leq x} f(n)\bar{\chi_1}(n)\Big|^2 \ll \varepsilon\phi(q)\Big(\frac{x}{q}\Big)^2.
\end{align}
Moreover, assuming GRH,~\eqref{form22} holds for all $(x/Q)^{\varepsilon^2}$-typical $q\in [1,Q]$.
\end{thm}

\subsection{Hybrid results}\label{sub_hybrid}

As already mentioned, our results are motivated by the following theorem from~\cite{mr-annals}.

\begin{customthm}{A}[Matom\"aki--Radziwi\l{}\l{}]\label{theo_mr} Let $10\leq h\leq X$, and let $f\colon \mathbb{N}\to [-1,1]$ be multiplicative. Then we have
\begin{align*}
\int_{X}^{2X}\Big|\sum_{x< n\leq x+h}f(n)-\frac{h}{X}\sum_{X\leq n\leq 2X}f(n)\Big|^2\,dx \ll \Big(\Big(\frac{\log \log h}{\log h}\Big)^2+(\log X)^{-1/50}\Big)Xh^2.
\end{align*}
\end{customthm}

This was generalized to functions $f\colon \mathbb{N}\to \mathbb{U}$ that are not $n^{it}$-pretentious for any $|t|\leq X$ by Matom\"aki--Radziwi\l{}\l{}--Tao~\cite{mrt-chowla}. Our next theorem is a hybrid result that allows us to ''interpolate'' between Theorem~\ref{theo_mr} (in the complex-valued case) and our Theorem~\ref{MRAPTHM} on multiplicative functions in short arithmetic progressions, thus generalizing both results. This theorem applies to sums of the form
\begin{align*}
\sum_{\substack{x< n\leq x+h\\n\equiv a\pmod q}}f(n)    
\end{align*}
over short intervals and arithmetic progressions, with averaging over $x\in [X,2X]$ and $a\in \mathbb{Z}_q^{\times}$, as soon as $h/q\rightarrow \infty$.

\begin{thm}[A hybrid theorem]\label{MRAPHybrid}
Let $X \geq h \geq 10$ and $1 \leq Q \leq h/10$. Let $(\log(h/Q))^{-1/200} \leq \varepsilon \leq 1$. Then there is a set $[1,X^{\varepsilon^{200}}]\cap \mathbb{Z}\subset \mc{Q}_{X,\e}\subset [1,X]\cap \mathbb{Z}$ satisfying $|[1,Q] \bk \mc{Q}_{X,\e}| \ll QX^{-\e^{200}}$ such that the following holds.

Let $q \in \mc{Q}_{X,\e} \cap [1,Q]$ be $(h/Q)^{\varepsilon^2}$-typical. Let $f\colon \mathbb{N}\to \mathbb{U}$ be multiplicative.  Let $\chi_1$ be a character $\hspace{-0.1cm}\pmod q$ minimizing the distance $\inf_{|t| \leq X} \mb{D}_q(f,\chi(n) n^{it};X)$, and for each $\chi$ let $t_{\chi} \in [-X,X]$ be a point that minimizes\footnote{If there are several such $t_{\chi}$, pick any one of them.} $\mb{D}_q(f,\chi(n) n^{it};X)$. Then  we have
\begin{align}\label{eqq151}
\int_X^{2X} \asum_{a \pmod{q}} &\Big|\sum_{\substack{x < n \leq x + h \\ n \equiv a \pmod{q}}} f(n) - \frac{\chi_1(a)}{\phi(q)} \left(\int_x^{x+h} v^{it_{\chi_1}} dv\right) \frac{1}{3X}\sum_{n \leq 3X} f(n)\bar{\chi}_1(n) n^{-it_{\chi_1}} \Big|^2 dx\\
&\ll \e\phi(q)X\Big(\frac{h}{q}\Big)^2.\nonumber
\end{align}
Moreover, assuming GRH,~\eqref{eqq151} holds for all $(h/Q)^{\varepsilon^2}$-typical $q\in [1,Q]$.
\end{thm}

We remark that for $h\leq \varepsilon X$, by Taylor approximation we have
\begin{align*}
\int_{x}^{x+h}v^{it_{\chi_1}}\, dv=hx^{it_{\chi_1}}+O(\varepsilon h).
\end{align*}

Taking $Q=1$, $\varepsilon=(\log h)^{-1/200}$, and letting $h$ tend to infinity slowly with $X$, we recover Theorem~\ref{theo_mr}  (though with a smaller power of logarithm saving) in a form that applies to any $1$-bounded $f$, whether $n^{it}$-pretentious or not (cf.~\cite[Theorem 1.7]{mrII}). Taking in turn $Q=o(h)$ and $h=X$, we arrive at a slightly weaker  form of our variance result, Theorem~\ref{MRAPTHM}, where we now need to average over $x\in [X,2X]$.

In the case of real-valued multiplicative functions $f\colon  \mb{N} \ra [-1,1]$, we have a simpler formulation of the result as follows.
\begin{cor}\label{HybridReal}
Let the notation be as in Theorem~\ref{MRAPHybrid}, and assume additionally that $f$ is real-valued. Then for all $q\in \mathcal{Q}_{X,\varepsilon}\cap [1,Q]$ that are $(h/Q)^{\varepsilon^2}$-typical we have
\begin{align*}
\int_X^{2X} \asum_{a \pmod{q}}\Big|\sum_{\substack{x < n \leq x+h \\ n \equiv a \pmod{q}}}f(n) - 
\frac{\chi_1(a)}{\phi(q)} \frac{h}{3X}\sum_{n\leq 3X} f(n)\bar{\chi_1}(n)\Big|^2\, dx \ll \varepsilon\phi(q)X\Big(\frac{h}{q}\Big)^2.
\end{align*}
Moreover, the second sum inside the absolute values can be deleted unless $\chi_1\pmod q$ is real. 
\end{cor}

We can also specialize Corollary~\ref{HybridReal} to $f=\mu$ and to the smaller range $q\leq x^{\varepsilon^{200}}$ to obtain a clean statement, which has recently been used in~\cite{wei} to obtain applications to ergodic theory. 

\begin{cor}\label{cor_mobius}
Let $A\geq 1$ be fixed. Let $X\geq h\geq 10q\geq 10$,  $(\log(h/q))^{-1/200}\leq \varepsilon\leq 1$,  $q\leq X^{\varepsilon^{200}}$, and let $q$ be $(h/q)^{\varepsilon^2}$-typical. Then we have
\begin{align*}
\int_{X}^{2X}\asum_{a\pmod q}\Big|\sum_{\substack{x < n \leq x+h \\ n \equiv a \pmod{q}}}\mu(n)\Big|^2\, dx\ll \varepsilon \varphi(q)X\left(\frac{h}{q}\right)^2,   
\end{align*}
except possibly if $q$ is a multiple of a single number $q_0\geq (\log X)^{A}$ depending only on $A$ and $X$.
\end{cor}

The exclusion of the multiples of a single modulus is necessary if Siegel zeros exist, as they bias the distribution of $\mu$ in residue classes.

\section{Applications}\label{section_apps}

A celebrated theorem of Linnik states that the least prime $p\equiv a\pmod q$ is $\ll q^{L}$ for some absolute constant $L$ and uniformly for $a\in \mathbb{Z}_q^{\times}$ and $q\geq 1$. The record value to date is $L=5$, due to Xylouris~\cite{xylouris}. For $q^{\delta}-$smooth moduli (with $\delta=\delta(\varepsilon)$), a better bound of $\ll q^{12/5+\varepsilon}$ is available, this being a result of Chang~\cite[Corollary 11]{chang}. Under GRH, we would have $L=2+o(1)$ in place of $L=5$, and assuming a conjecture of Cram\'er-type, $L=1+o(1)$ would be the optimal exponent.

We apply the techniques used to prove our main results to make progress on the analogue of Linnik's theorem for $E_3$ numbers, that is, numbers that are the product of \emph{exactly} $3$ primes. We seek bounds on the quantity 
\begin{align*}
\mathscr{L}_{3}(q)\coloneqq \max_{a\in \mathbb{Z}_q^{\times}}\min\{n\in \mathbb{N}:\,\,n\equiv a\pmod q:\,\, n\in E_3\}.
\end{align*}
One can show that under GRH one has $\mathscr{L}_3(q)\ll q^{2+o(1)}$. The $E_3$ numbers, just like the primes, are subject to the parity problem, and hence one cannot use sieve methods to tackle the problem of bounding $\mathscr{L}_3(q)$ (in contrast, for products of \emph{at most} two primes $\leq x$, it is known that one can find them in every reduced residue class modulo $q$ for $q\leq x^{1/2+\delta}$ for some $\delta>0$ by a result of Heath-Brown~\cite{hb-almostprime} proved using sieve methods). In relation to this problem, Ramar\'{e} and Walker~\cite{RW} obtained the bound $\mathscr{L}_3(q) \ll q^{16}$ by constructing products of primes $p_1p_2p_3$ with each $p_j \leq q^{16/3}$.

We show unconditionally that $\mathscr{L}_3(q)\ll q^{2+o(1)}$ for all smooth moduli and for all but a few prime moduli; moreover, the products $p_1p_2p_3$ constructed satisfy $p_j < q$ for $j = 1,2,3$.

\begin{thm}\label{theo_linnik1}
Let $\varepsilon>0$, and let $\varepsilon'>0$ be small enough in terms of $\varepsilon$.
\begin{enumerate}[label=\upshape(\roman*)] 
    \item For any integer $q\geq 1$  that is $q^{\varepsilon'}$-smooth, for any $a\in \mathbb{Z}_q^{\times}$, there exists some $q$-smooth $n\in E_3$ such that $n\ll q^{2+\varepsilon}$ and $n\equiv a\pmod q$. Consequently, $\mathscr{L}_{3}(q)\ll q^{2+\varepsilon}$.
    
    \item Let $Q\geq 2$. Then for all but $\ll_{\varepsilon} 1$ primes $q\in [Q^{1/2}, Q]$, for any $a\in \mathbb{Z}_q^{\times}$, there exists some $q$-smooth $n\in E_3$ such that $n\ll q^{2+\varepsilon}$ and $n\equiv a\pmod q$. Consequently, $\mathscr{L}_{3}(q)\ll q^{2+\varepsilon}$.
\end{enumerate}
\end{thm}

This will be proved in Section~\ref{sec: linnik}. Since all the $E_3$ numbers we detect are $q$-smooth, our results are connected to the question of representing every element of the multiplicative group $\mathbb{Z}_q^{\times}$ by using only a bounded number of small primes. This problem was introduced by Erd\H{o}s, Odlyzko and S\'ark\"ozy in~\cite{eos}. In~\cite[Section 2]{eos} it is mentioned that Erd\H{o}s conjectured that every residue class in $\mathbb{Z}_{q}^{\times}$, with $q$ a large prime, has a representative of the form $p_1p_2$ with $p_1,p_2\leq q$ primes. As is noted in~\cite{walker}, this remains open, even under GRH. The weaker ``Schnirelmann-type'' question of representing every residue class in $\mathbb{Z}_q^{\times}$ as the product of \emph{at most} $k$ primes in $[1,q]$ was studied by Walker~\cite{walker}, who showed\footnote{Both in~\cite{walker} and~\cite{shparlinski} a stronger result was shown, namely that one can restrict to primes in $[1,q^{1-\eta}]$ for explicitly given values of $\eta>0$. An inspection of the proof of our Corollary~\ref{theo_schnirelmann} shows that there also we could restrict to primes bounded by $q^{1-\eta}$, with $\eta>0$ small enough.} that $k=6$ suffices for all large primes $q$, and moreover that $k=48$ suffices if we consider products of exactly $k$ primes. Shparlinski~\cite{shparlinski} then improved on the former by showing that at most $5$ primes suffice for every large integer $q$. See also the very recent works~\cite{brs},~\cite{szabo} for further results on this problem. From Theorem~\ref{theo_linnik1} we deduce the following.

\begin{cor}[Ternary version of Erd\H{o}s' conjecture with bounded exceptional set]\label{theo_schnirelmann} There exists an absolute constant $C>0$ such that the following holds. For all $Q\geq 2$ and all primes $q\in [Q^{1/2},Q]$, apart from $\leq C$ exceptions, every element of the multiplicative group $\mathbb{Z}_q^{\times}$ can be represented as the product of exactly three primes from $[1,q]$.
\end{cor}

Finally, we consider an analogue of Linnik's theorem concerning values of the M\"obius function. Since the theorems above give $\mathscr{L}_3(q)\ll q^{2+o(1)}$ for smooth $q$ and all but a few primes $q$ (and since the $E_3$ numbers we detect are typically squarefree), for such $q$ the least number $n$ with $\mu(n)=-1$ and $n\equiv a\pmod q$ also satisfies $n\ll q^{2+o(1)}$. Going further, we are able to obtain lower bounds of the correct order of magnitude for the number of $n \leq x$ with $\mu(n) = -1$ in any residue class $a \pmod{q}$ as soon as $x \geq q^{2+\e}$, as opposed to just showing their existence.

\begin{prop}\label{theo_linnik_mobius} Let $\varepsilon>0$ and $Q\geq 2$. Then, for all but $\ll_{\varepsilon} 1$ primes $q\in [Q^{1/2}, Q]$, we have
\begin{align*}
\min_{a\in \mathbb{Z}_q^{\times}} |\{n\leq x: \,\, n\equiv a \pmod q, \,\, \mu(n)=-1\}| \gg_{\varepsilon} \frac{x}{q}    
\end{align*}
for all $x\geq q^{2+\varepsilon}$. The same holds when the condition $\mu(n) = -1$ is replaced by $\mu(n) = +1$.
\end{prop}

We lastly remark that unconditionally proving the estimate $\mathscr{L}_{3}(q)\ll q^{2+o(1)}$ for {\it every} $q$ seems challenging, due to connections between this problem and Vinogradov's conjecture (see Subsection~\ref{subsec: vino}).

\section{Optimality of theorems and previous work}\label{sec: optimal}

\subsection{Previous results}\label{sub: 1.1}

The study of the deviations $\eqref{eqq1}$ and~\eqref{eqq2} of $f$ in arithmetic progressions can roughly speaking be divided into three different regimes: the \emph{small moduli} $q\leq x^{\varepsilon}$, the \emph{middle moduli} $x^{\varepsilon}\leq q\leq x^{1-\varepsilon}$, and the \emph{large moduli} $x^{1-\varepsilon}\leq q=o(x)$, for $\e > 0$ small.

\subsubsection{Small moduli}

In the regime of small moduli, Linnik's theorem, in its quantitative form~\cite[Theorem 18.6]{iw-kow}, gives the expected asymptotic formula for the average of $\mu$ (or $\Lambda$) over $a\pmod q$, valid for \emph{all} $a\in \mathbb{Z}_q^{\times}$ and $q\leq x^{\varepsilon}$, apart possibly from multiples $q$ of a single number $q_0$ (a Siegel modulus). A far-reaching generalization of this to arbitrary $1$-bounded multiplicative functions $f$ was achieved by Balog, Granville and Soundararajan~\cite{bgs}. See also the work~\cite{ghs} of Granville, Harper and Soundararajan for related results. One consequence of our Theorem~\ref{MRAPHybrid} (noting that the set $\mathcal{Q}_{X,\varepsilon}$ there contains $[1,X^{\varepsilon^{200}}]\cap \mathbb{Z}$) is a short interval version of the result of~\cite{bgs}, albeit with an average over $a$.

\subsubsection{Middle moduli}

The middle regime $q=x^{\theta}$ with $\varepsilon\leq \theta\leq 1-\varepsilon$ (and typically with $\theta$ near $1/2$) is arguably the most well-studied one. Results related to this range include the celebrated Bombieri--Vinogradov theorem, which for $f=\mu$ (or $f=\Lambda$) can be interpreted as providing cancellation in the deviation~\eqref{eqq1} for \emph{almost all} $q\leq x^{1/2-\varepsilon}$ and \emph{all} $a\in \mathbb{Z}_q^{\times}$.  A complete generalization of the Bombieri--Vinogradov theorem to arbitrary $1$-bounded multiplicative functions was recently achieved by Granville and Shao~\cite[Theorem 1.2]{granville-shao}.

The work of Granville and Shao in particular implies the following result applicable to almost all moduli: if $f\colon  \mb{N} \ra \mb{U}$ is multiplicative, then for all but $\leq Q/(\log x)^{1-\frac{1}{\sqrt{2}}-2\varepsilon}$ choices of $q\in [Q,2Q]\subset [1,x^{1/2-\varepsilon}]$, we have
\begin{align}\label{eqq3}
\max_{a\in \mathbb{Z}_q^{\times}}\Big|\sum_{\substack{n\leq x\\n\equiv a\pmod q}}f(n)-\frac{\chi_1(a)}{\phi(q)}\sum_{n\leq x}f(n)\bar{\chi_1}(n)\Big|=o\Big(\frac{x}{q}\Big).     
\end{align}
In~\cite[Corollary 1.7]{granville-shao}, Granville and Shao obtained a saving of an arbitrary power of $\log x$, assuming that $f$ is supported on $x^{\eta}$-smooth numbers and that $f$ satisfies the Siegel--Walfisz condition.

Our Theorem~\ref{BVTHM} improves on~\eqref{eqq3} in the aspect of the exceptional set, as our result implies that if $f$ is supported on $x^{\eta}$-smooth numbers, then the size of the exceptional set of $q\leq x^{1/2-\varepsilon}$ in~\eqref{eqq3} can be reduced to an almost power-saving bound, or even to a power of logarithm in the case of prime moduli. In this aspect, our result may also be compared with a recent result of Baker~\cite{baker}, who showed that the estimate~\eqref{eqq3} holds for $f=\Lambda$ (with $\chi_1\equiv 1$) for all but a power of logarithm number of primes $q\leq x^{9/40-\varepsilon}$. 

Though the theorems in this paper are not related to such developments, we note in passing that in the literature there are several results, valid in the middle range $q\leq x^{\theta}$, that go beyond $\theta=1/2$ for general multiplicative functions,  provided that one removes the maximum over the residue classes $a\pmod q$. See the works of Green~\cite{green-bombieri}, Granville--Shao~\cite{granville-shao}, Drappeau--Granville--Shao~\cite{dgs}, and Fouvry--Radziwi\l{}\l{}~\cite{fr}, among others.

\subsubsection{Large moduli}

In the range $x^{1-\varepsilon}\leq q=o(x)$, one aims for estimates valid for \emph{almost all} $q$ and for \emph{almost all} $a\in \mathbb{Z}_q^{\times}$; results of this shape arise from upper bounds for the variance~\eqref{eqq2}. The most classical theorem of this type is the Barban--Davenport--Halberstam theorem~\cite[Chapter 17]{iw-kow}, which states that
\begin{align}\label{eqq120}
\sum_{q\leq x/(\log x)^B}\sum_{a\in \mathbb{Z}_q^{\times}}\Big|\sum_{\substack{n\leq x\\n\equiv a \pmod q}}\mu(n)\Big|^2\ll_A \frac{x^2}{(\log x)^A},    
\end{align}
with $B=B(A)$ explicit (and there is an analogue in which $\mu$ is replaced with $\Lambda$). 

The Barban--Davenport--Halberstam theorem was extensively studied by Hooley in a seminal series of publications titled ``On the Barban--Davenport--Halberstam theorem'', spanning 19 papers. In this series, he significantly improved and generalized the Barban--Davenport--Halberstam bound, and among other things produced an asymptotic formula for the left-hand side of~\eqref{eqq120}, and also with $\mu$ replaced by any bounded sequence satisfying a {\it Siegel--Walfisz type assumption}. Of this series of papers, the ones related to the aims of the present paper are~\cite{H3}, \cite{H6}, \cite{H9}, \cite{H10}, \cite{H14}. 
In particular, from~\cite{H3} (where Hooley considers the variance summed over all moduli $q\leq Q$) one extracts the following result (see also the work of Vaughan~\cite{vaughanII} for a related result, proved using the circle method).

\begin{customthm}{B}[Hooley]\label{theo_hooley} Let $\varepsilon>0$ and $A\geq 1$ be fixed. Let $1\leq Q\leq x$, and let $f\colon \mathbb{N}\to \mathbb{U}$ be an arbitrary function satisfying the Siegel--Walfisz condition. Denote $H\coloneqq x/Q$. Then, for all $1\leq q\leq Q$ apart from $\ll Q((\log H)/H+(\log x)^{-A})$ exceptions we have
\begin{align*}
\asum_{a \pmod{q}}\Big|\sum_{\substack{n\leq x\\n\equiv a\pmod q}}f(n)-\frac{\chi_0(a)}{\phi(q)}\sum_{n\leq x}f(n)\chi_0(n)\Big|^2\leq \varepsilon \varphi(q)\left(\frac{x}{q}\right)^2,   
\end{align*}
where, for each $q \leq Q$, the character $\chi_0$ is principal modulo $q$.
\end{customthm}

By our Theorem~\ref{MRAPTHM}, and the fact that the number of moduli $q\leq Q$ that are not $H^{\varepsilon^2}$-typical is $\ll Q\exp(-(1/1000+o(1))H^{\varepsilon^2})$ (see Lemma~\ref{le_delta}), the size of the exceptional set here for multiplicative $f$ reduces to $\ll Q\exp(-c_0H^{\varepsilon^2})$, for $c_0 > 0$ constant. We can at the same time remove the Siegel--Walfisz assumption on $f$. If we restrict ourselves to $H^{\varepsilon^2}$-typical moduli only, then  Theorem~\ref{MRAPTHM} bounds the exceptional set by $\ll Qx^{-\varepsilon^{200}}$. This essentially power-saving bound was not, according to our knowledge, previously available even for $f=\mu$.

We now discuss some of the key features of our results when it comes to the strength and optimality.

\subsection{The description and size of the exceptional set}

The set $([1,x]\cap \mathbb{Z})\setminus \mathcal{Q}_{x,\varepsilon}$ of exceptional moduli present in our main theorems turns out to be completely independent of the function $f$ that we consider, a feature that does not arise from applying the Barban--Davenport--Halberstam theorem or Hooley's Theorem~\ref{theo_hooley}. We have the following explicit\footnote{Here, given a Dirichlet character $\chi$ we denote by $\text{cond}(\chi)$ the conductor of $\chi$.} description of $\mathcal{Q}_{x,\varepsilon}$ in terms of zeros of $L$-functions $\hspace{-0.1cm}\pmod q$: 
\begin{align*}
\mathcal{Q}_{x,\varepsilon}\coloneqq\Bigg\{q\leq x: \prod_{\substack{\chi\pmod q\\\cond(\chi)>x^{\varepsilon^{200}}}} L(s,\chi)\neq 0\quad \textnormal{for}\quad \Re(s)\geq 1-\frac{\varepsilon^{-100}(\log \log x)}{\log x},\quad |\Im(s)|\leq 3x\Bigg\};  \end{align*}
see Proposition~\ref{prop_Parseval1} and Lemma~\ref{le_Qset} for this. Hence, assuming GRH (or even a weaker version of it), $\mathcal{Q}_{x,\varepsilon} = [1,x]\cap \mathbb{Z}$. From the description of $\mathcal{Q}_{x,\varepsilon}$ above and zero density estimates, it is not difficult to see that we have a \emph{structural description} of the exceptional moduli as the set of multiples of a subset $\mathcal{E}_x\subset [x^{\varepsilon^{200}},x]$ of integers of size $O((\log x)^{\varepsilon^{-200}})$. This explains why the bound for the number of exceptional \emph{prime} moduli in Corollary~\ref{MRAPTHM-prime} is so strong, compared to the case of general moduli.

\subsection{Connection to Vinogradov's conjecture and character sums}\label{subsec: vino}

 For any fixed $\varepsilon>0$, the number of exceptional moduli in Theorem~\ref{MRAPTHM} is of the form $Qx^{-\e^{200}}$, saving a power of $x$ that tends to $0$ as $\varepsilon\to 0$. We show here that this is essentially the best possible, in the sense that replacing $Qx^{-\e^{200}}$ by $Qx^{-\eta_0}$ for $\eta_0 > 0$ fixed would lead to the proof of some form of Vinogradov's conjecture\footnote{Vinogradov's conjecture on the least quadratic nonresidue states that for every $\eta>0$ and for any prime $q>q_{\eta}$ there is a quadratic nonresidue $\hspace{-0.1cm}\pmod q$ in the interval $[1,q^{\eta}]$.} (which is known under GRH but not unconditionally).

Indeed, assume that Vinogradov's conjecture is false. Then there exists $\eta>0$ and infinitely many $x\geq 10$ such that for some prime $x^{\eta-o(1)}\leq q_0\leq x^{\eta}$ we have $\legendre{n}{q_0} = 1$ for all $n$ with the largest prime factor $P^+(n) \leq q_0^{\eta}$. 

Defining the multiplicative function $f_{\eta}(n)\coloneqq 1_{P^{+}(n)\leq q_0^{\eta}}$, by the classical asymptotic formula for smooth numbers (and the fact that $q_0$ is prime), we have
\begin{align}\label{eqq10}
\begin{split}
\sum_{n\leq x} f_{\eta}(n)\legendre{n}{q_0} &=\sum_{n\leq x} f_{\eta}(n)\chi_0(n) = (\rho(\eta^{-2})+o(1)) x,
\end{split}
\end{align}
with $\chi_0$ the principal character modulo $q_0$ and $\rho(\cdot)$ the Dickman function (see Section~\ref{sub: not} for its definition). It follows that, regardless of the choice of $\chi_1 \pmod{q_0}$,
\begin{align}\label{eqq9}
\frac{1}{\phi(q_0)}\sum_{\substack{\chi \pmod{q_0}\\\chi\neq \chi_1}}\Big|\sum_{n\leq x}f_{\eta}(n)\overline{\chi}(n)\Big|^2\gg_{\eta} \varphi(q_0)\frac{x^2}{q_0^2}.    
\end{align}
However, by Parseval's identity (i.e., just expanding the square and using orthogonality),~\eqref{eqq9} equals to the left-hand side of~\eqref{form22} (with $f=f_{\eta}$), and thus $q_0\not \in \mathcal{Q}_{x,\varepsilon}$ if $\varepsilon$ is small in terms of $\eta$. 

Note that if $Q=x/\log x$ and $r=q_0p$ with $p\in [\log x,Q/q_0]$ a prime, then the same argument as above (with $\chi_0(n)$ and $\legendre{n}{q_0}$ replaced by $\chi_0(n)1_{(n,r)=1}$ and $\legendre{n}{q_0}1_{(n,r)=1}$ in~\eqref{eqq10}) shows that also $r\not \in \mathcal{Q}_{x,\varepsilon}$, meaning that there are $\gg Qx^{-\eta+o(1)}$ exceptional $q\leq Q$ (again with $\varepsilon$ small enough in terms of $\eta$). Taking $\eta<\eta_0$, this shows that the number of exceptional moduli for~\eqref{form22} is in fact not bounded by $\ll Qx^{-\eta_0}$. Thus, one cannot generally improve on the exceptional set in Theorem~\ref{MRAPTHM} without settling Vinogradov's conjecture at the same time.

One could also adapt the argument above to show more strongly that improving  the exceptional set for~\eqref{form22} implies cancellation in smooth-supported character sums. Using arguments from~\cite{gs-burgess}, it should further be possible to say that this implies bounds for zeros of $L$-functions near $1$ (which is out of reach of current knowledge).

Similar conclusions apply to the size of the exceptional set in our other main theorems.

\subsection{The restriction to typical moduli}\label{subsec: typical}

We now discuss the importance of working with typical moduli in Theorems~\ref{MRAPTHM} and~\ref{MRAPHybrid}. In our proofs, as in the work~\cite{mr-annals}, it is important for us to be able to discard those $n\leq x, n\equiv a\pmod q$ from the sum~\eqref{form22} that have no prime factors from certain long intervals $[P_i,Q_i]$ (with $Q_i\leq h/Q$). However, if $q$ is divisible by all (or most) primes in $[P_i,Q_i]$, then the contribution of such integers is not negligible. This would then prevent us from factorizing our character sums in a desirable way, which is crucial to our method.  

While Theorem~\ref{MRAPTHM} may remain valid for all moduli $q\leq Q$ (under GRH, say), there seem to be serious obstacles to proving this. Indeed, Granville and Soundararajan~\cite{gs-uncertainty} proved a very general uncertainty principle for arithmetic sequences, which roughly speaking says that ``multiplicatively interesting'' sequences cannot be perfectly distributed in all arithmetic progressions. For example, if $f(n)=1_{(n,r)=1}$ with $r$ having very many small prime factors in the sense that $\sum_{p\mid r, p\leq \log x}(\log p)/p\gg \log \log x$, then for large constant $C>0$ there exists $y\in (x/4,x)$ and a progression $a\pmod q$ with $(a,q)=1$ and $q\leq x/(\log x)^{C}$ and $P^{-}(q)\gg \log \log x$ such that the mean value of $f$ over $n\leq y, n\equiv a\pmod q$ does not obey the anticipated asymptotic formula. Note that this is not due to ``trivial'' reasons such as $f$ having sparse support, as it is possible for $f$, constructed in this fashion, to have $f(n) =1$ for a positive proportion of $n\leq x$, e.g., if $r=\prod_{(\log x)^{1-\eta}\leq p\leq \log x}p$.

Similarly, if for example $f$ is the indicator of sums of two squares, then the results of~\cite{gs-uncertainty} imply that $f$ is poorly distributed in some residue classes $a\pmod q$ with $q\leq x/(\log x)^{C}$.

\subsection{Remarks on improvements}

We finally list a few small improvements to our main theorems that could be obtained with only slight modifications to the proofs.

\begin{itemize}
    \item In Theorem~\ref{MRAPHybrid}, we obtain an upper bound for~\eqref{eqq151} of the form $(\log(h/Q))^{-c}\phi(q)(x/q)^2$  for small $c>0$ by choosing $\varepsilon=(\log(h/Q))^{-1/500}$, say. Thus our savings are comparable to those in~\cite[Theorem 3]{mr-annals}. Due to our reliance on typical factorizations, one cannot hope for larger savings than $((\log\log(h/Q))/\log(h/Q))^2$ in general. However, if one specializes to the case $f=\mu$ in our main theorems, one can easily adapt the proof to yield savings of the form $\ll (\log(h/Q))^{-2+o(1)}$ by applying the Siegel--Walfisz theorem in place of H\'alasz-type estimates. We leave the details to the interested reader.
    
    \item As in the work of Granville and Shao~\cite{granville-shao} on the Bombieri--Vinogradov theorem for multiplicative functions, we could obtain stronger bounds for~\eqref{form22} if we subtracted the contribution of more than one character from the sum of $f$ over an arithmetic progression. Moreover, it follows directly from our proof that if we subtracted the contribution of $\ll (\log x)^{C(\varepsilon)}$ characters, where $C(\varepsilon)>0$ is large, then there would be no exceptional $q$ at all in the theorem. We leave these modifications to the interested reader.

\end{itemize}

\section{Proof ideas}

We shall briefly outline some of the ideas that go into the proofs of our main results. 

\subsubsection{Proof ideas for the variance results}
We start by discussing the proof of the hybrid result, Theorem~\ref{MRAPHybrid}; the proof of our result on multiplicative functions in short progressions, Theorem~\ref{MRAPTHM}, is similar but slightly easier in some aspects.

As in the groundbreaking work of Matom\"aki--Radziwi\l{}\l{}~\cite{mr-annals}, we begin by applying a suitable version of Parseval's identity to transfer the problem to estimating an $L^2$-average of partial sums of $f$ twisted by characters from a family. Of course, since we are working with both intervals and arithmetic progressions, the right family of characters to employ are the \emph{twisted characters} $\{\chi(n)n^{it}\}_{\substack{\chi\pmod q \\ |t|\leq X/h}}$. In this way, we reduce our task to obtaining cancellation in 
\begin{align*}
\sum_{\chi\pmod q}\int_{t\in T_{\chi}}\Big|\sum_{X\leq n\leq 2X}f(n)\overline{\chi}(n)n^{-it}\Big|^2\, dt,    
\end{align*}
with $T_{\chi}=[-X/h,X/h]$ if $\chi\neq \chi_1$ and $T_{\chi_1}=[-X/h,X/h]\setminus [t_{\chi_1}-\varepsilon^{-10},t_{\chi_1}+\varepsilon^{-10}],$ with $\chi_1$ and $t_{\chi_1}$ as in the theorem (so $(\chi,t)\mapsto \mathbb{D}_q(f,\chi(n)n^{it};X)$ for $\chi\pmod q$ and $|t|\leq X$ is minimized at $(\chi_1,t_{\chi_1})$); the contribution from the deleted segment in $T_{\chi_1}$ accounts for our main term.

We make crucial use of the Ramar\'e identity, thus obtaining a factorization\footnote{Due to the restriction to reduced residue classes $a\pmod q$ in our theorems, we have desirable factorizations for typical integers only if $q$ is not divisible by an atypically large number of small primes, e.g. by almost all of the primes up to $(h/Q)^{0.01}$. This is what results in the need in our main theorems to restrict to typical moduli. This issue of course does not arise in the short interval setting of~\cite{mr-annals}.}  
\begin{align*}
\sum_{X\leq n \leq 2X} f(n)\overline{\chi}(n)n^{-it}\approx \sum_{P_j\le p\le Q_j}f(p)\overline{\chi}(p)p^{-it}\sum_{X/p\leq m\leq 2X/p}f(m)a_{m,P_j,Q_j}\overline{\chi}(m)m^{-it},    
\end{align*}
with parameters $P_j,Q_j$, $1 \leq j \leq J$, at our disposal, and the approximation being accurate in an $L^2$-sense (after splitting the $p$ variable into short intervals). Here $a_{m,P_j,Q_j}\coloneqq \frac{1}{1+\omega_{[P_j,Q_j]}(m)}$ is a well-behaved sequence, behaving essentially like the constant $1$ for the purposes of our argument. After having obtained this bilinear structure, we split the ``spectrum'' $\{\chi\pmod q\}\times [-X/h,X/h]$ into parts depending on which (if any) of the sums $\sum_{P_j\le p\le Q_j}f(p)\overline{\chi}(p)p^{-it}$ with $j\leq J$ exhibits cancellation. The contributions from different parts of the spectrum are bounded differently by establishing various mean and large value estimates for twisted character sums (see Section~\ref{sec: largesieves}), in analogy with~\cite[Section 4]{mr-annals} for Dirichlet polynomials.

The outcome of all of this is that we can reduce to the case where the longest of our twisted character sums, $\sum_{P_J\le p\le Q_J}f(p)\overline{\chi}(p)p^{-it}$, has (essentially) \emph{no cancellation} at all. It is this large spectrum case where we significantly deviate from~\cite{mr-annals}; in that work, the large spectrum is not the most difficult case to deal with, thanks to the Vinogradov--Korobov zero-free region for the Riemann zeta-function. In our setting, in turn, we encounter $L$-functions $L(s,\chi)$ with $\chi$ having very large conductor, and for these $L$-functions the known zero-free regions are very poor (the best region being the Landau--Page zero-free region $\sigma>1-\frac{c_0}{\log(q(|t|+1))}$, valid apart from possible Siegel zeros). At this point, we restrict the set of moduli in question to those $q \leq Q$ for which the functions $L(s,\chi)$ for every $\chi\pmod q$ (of large conductor) enjoy a suitable zero-free region (see Proposition~\ref{prop_Parseval2} and Lemma~\ref{le_Qset} for the definition of the region involved). Our bounds for the number of moduli omitted in this fashion come from log-free zero-density estimates for $L$-functions (Lemma~\ref{le_zerodensity}); in the case of pairwise coprime moduli, as in Corollary~\ref{MRAPTHM-prime}, the bound is much better thanks to there being no effect from a single bad character inducing many others.

Having restricted to such moduli we establish a bound essentially of the form
\begin{align}\label{eqq45}
\sup_{\chi\pmod q}\sup_{\substack{|t|\leq X\\\chi=\chi_1\Longrightarrow |t-t_{\chi_1}|\geq \varepsilon^{-10}}}\Big|\sum_{X\leq n \leq 2X} f(n)\overline{\chi}(n)n^{-it}\Big|\ll \varepsilon\frac{\phi(q)}{q}X    
\end{align}
for the \emph{sup norm} of the twisted character sums involved, and also prove that the large spectrum set under consideration is extremely small\footnote{One could use moment estimates (e.g. Lemma~\ref{PRIMMVTwithT}) to show that the large values set is $\ll (\log X)^{O_{\varepsilon}(1)}$ in size; however, in our case that would be a fatal loss, since the saving we get in~\eqref{eqq45} is at best $1/\log X$ and is therefore not enough to compensate this. In~\cite{mr-annals}, a Hal\'asz--Montgomery-type estimate for prime-supported Dirichlet polynomials is established to deal with the large spectrum; our Proposition~\ref{prop_largevalues_hyb} essentially establishes a hybrid version of this, but in a very different regime.}, that is, 
\begin{align*}
\sup_{P\in [X^{\varepsilon},X]}\Big|\Big\{(\chi,t)\in \{\chi\pmod q\}\times \mc{T}:\,\, |\sum_{P\leq p\leq 2P}f(p)\overline{\chi}(p)p^{-it}|\geq \frac{\varepsilon P}{\log P}\Big\}\Big|\ll \varepsilon^{-2},    
\end{align*}
with $\mathcal{T}\subset [-X,X]$ well-spaced. These two bounds are our two key Propositions~\ref{prop_sup_hyb} and~\ref{prop_largevalues_hyb} in the proof of the hybrid theorem. We need full uniformity in $|t|, q\leq X$, which makes the proofs somewhat involved: in particular, our proofs rely on some lemmas from the works of Koukoulopoulos~\cite{kou} and  Granville--Harper--Soundararajan~\cite{ghs} (as well as a result of Chang~\cite[Theorem 5]{chang} for Theorem~\ref{MRAPTHM-smooth} on smooth moduli).

\subsubsection{Proof ideas for the case of all moduli in the square-root range}
The starting point of the proof of Theorem~\ref{BVTHM} is the simple Lemma~\ref{le_comb} that allows us to conveniently decompose any $x^{\eta}$-smooth number into a product $n=dm$ with an appropriate choice of $d,m\in [x^{1/2-\eta},x^{1/2+\eta}]$.
However, the decoupling of the $d$ and $m$ variables here is somewhat delicate and requires some smooth number estimates.  After decoupling the variables (and extracting a further small prime factor), we have introduced a trilinear structure with two variables of almost equal length, which (by Cauchy--Schwarz) means that we can employ the techniques from previous sections to bound the mean squares of the product of three character sums involved. 

\subsubsection{Proof ideas for the Linnik-type results}
For the proof of our Linnik-type results, Theorems~\ref{theo_linnik1}(i)--(ii), we use similar ideas as for Theorem~\ref{BVTHM}, with a couple of additions. Since we only need a positive lower bound for the number of $n\equiv a \pmod q$ that are $E_3$ numbers,  we can require that these $n$ have prime factors from any intervals that we choose. Thanks to this flexibility in the sizes of the prime factors, we can get good bounds for the trilinear sums that arise. A key maneuver here is to count suitable $n$ with the logarithmic weight $1/n$, so that we will be able to utilize a modification of the ''Rodosskii bound'' from the works of Soundararajan~\cite{sound} and Harper~\cite{harper-smooth}, which establishes cancellation in logarithmically averaged character sums over primes assuming only a very narrow zero-free region. For smooth moduli, we have a suitable zero-free region by a result of Chang~\cite[Theorem 5]{chang}, whereas for prime $q$ we apply the log-free zero-density estimate to obtain a suitable region apart from a few bad moduli.

\section*{Structure of the paper} We will present the proofs of Theorems~\ref{MRAPTHM} and~\ref{MRAPHybrid}  in Subsections~\ref{sub: APs} and~\ref{sub: hybrid}, respectively. The necessary lemmas for proving these results are presented in Sections~\ref{sec: largesieves} and~\ref{sec: multiplicative}. Section~\ref{sec: propositions} in turn contains two propositions that are key ingredients in the proofs of the main theorems.  In Section~\ref{sec: smooth} we prove Theorem~\ref{MRAPTHM-smooth} on smooth moduli. Our result on smooth-supported functions in the square-root range is proved in Section~\ref{sec: bilinear}. Section~\ref{sec: linnik} in turn contains the proofs of the applications to Linnik-type theorems. We remark that Sections~\ref{sec: MR},~\ref{sec: bilinear} and~\ref{sec: linnik} can be read independently of each other, but they all depend on the work in Section~\ref{sec: propositions}.

\section{Notation}\label{sub: not}

We use the usual Vinogradov and Landau asymptotic notation $\ll, \gg$, $\asymp$, $O(\cdot), o(\cdot)$, with the implied constants being absolute unless otherwise stated. If we write $\ll_{\varepsilon}, \gg_{\varepsilon}$ or $O_{\varepsilon}(\cdot)$, this signifies that the implied constant depends on the parameter $\varepsilon$. 

We write $1_{S}(n)$ for the indicator function of a set $S$. The functions $\Lambda,$ $\phi$ and $\tau_k$ are the usual von Mangoldt, Euler phi and $k$-fold divisor functions, and $\pi(x)$ is the prime-counting function. By $P^{+}(n)$ and $P^{-}(n)$ we mean the largest and smallest prime factors of $n$, respectively. We say that $n$ is $y$-smooth if $P^{+}(n)\leq y$. We write $e(x)=e^{2\pi i x}$ for the complex exponential. The symbol $\rho\colon (0,\infty)\to [0,1]$ denotes the Dickman function, the unique solution to the delay differential equation $\rho(u-1)=-u\rho'(u)$ for $u>1$, with the initial data $\rho(u)=1$ for $0<u\leq 1$; see~\cite{hildebrand-tenenbaum} for further properties of this function.

The symbol $p$ is reserved for primes, whereas $j, k,m,n,q$ are positive integers.

Below we list for the reader's convenience the notation we introduce in later sections.

\section*{Nomenclature}

\begin{tabular}{p{5cm} p{10cm}} 

$\asum_{a \pmod{q}}$ & A sum over the invertible residue classes $\hspace{-0.1cm}\pmod q$\\

$\asum_{\chi \pmod{q}}$ & A sum over the primitive characters $\hspace{-0.1cm}\pmod q$\\

$\chi_0$ & The principal character\\

$\chi^{\ast}$ & The primitive character inducing the character $\chi$\\

$\cond(\chi)$ & The conductor of the character $\chi$\\

$\mathbb{Z}_{q}^{\times}$ & The set of invertible residue classes $\hspace{-0.1cm}\pmod q$\\

$\Omega_{[P,Q]}(n), \omega_{[P,Q]}(n)$ & The number of prime factors of $n$ from an interval $[P,Q]$, with and without multiplicities, respectively\\

$\Delta(q,Z)$ & Equation~\eqref{eq_Delta}\\

$\Psi_q(X,Y)$ & Equation~\eqref{eqq117}\\

 $\mathbb{D}_q(f,g;x)$ & Equation~\eqref{eqq123} \\ 

 $\mathbb{D}(f,g;y,x)$ & Equation~\eqref{eqq70} \\ 
 
 $F(\chi)$ & Equation~\eqref{eq_Fchi}\\

$L_y(s,f)$ & Equation~\eqref{eq_Ly}\\

$M_q(T)$ & Equation~\eqref{eq_Mq}\\

$N(\sigma,T,\chi)$ & Equation~\eqref{eq_N(sigma)}\\

$\mathcal{Q}_{x,\varepsilon,M}$ & Equation~\eqref{eqq22}\\

$V_t$ & Equation~\eqref{eq_Vt}\\
 
\end{tabular}

\section{Mean and large values estimates}

\label{sec: largesieves}

We begin this section with several standard $L^2$-bounds for sums twisted both by Dirichlet and Archimedean characters.

\begin{note1}
In what follows, we will seek to make all of our estimates as sharp as possible as a function of $q$, in particular obtaining factors of $\phi(q)/q$ in our estimates wherever possible. While this increases the lengths of some proofs (particularly in Section~\ref{sec: multiplicative}), it is critical in order for us to state our main variance estimates with no loss.
\end{note1}

\begin{lemma}[Large sieve for characters] \label{SIMPLEORTHO}
Let $q,M,N\geq 1$, and let $(a_n)_{n}$ be complex numbers. Then
$$\sum_{\chi \pmod{q}} \Big|\sum_{M < n \leq M+N} a_n\chi(n)\Big|^2 \ll \Big(\phi(q)+\frac{\phi(q)}{q}N\Big) \sum_{\substack{M < n \leq M+N\\ (n,q) = 1}} |a_n|^2.$$
\end{lemma}
\begin{proof}
This is~\cite[Theorem 6.2]{MonLS}. 
\end{proof}

\begin{lem}[Hybrid large sieve for characters] \label{L2MVTwithT}
Let $T,N, q \geq 1$. Then
$$\sum_{\chi \pmod q} \int_0^T \Big|\sum_{n \leq N} a_n\chi(n)n^{it}\Big|^2 dt \ll \Big(\phi(q)T+\frac{\phi(q)}{q}N\Big) \sum_{\substack{n\leq N\\ (n,q) = 1}} |a_n|^2.$$
\end{lem}

\begin{proof}
This is a slight sharpening of~\cite[Theorem 6.4]{MonLS} (more precisely, see (6.14) there).
\end{proof}

For the proof of Lemma~\ref{PRIMMVTwithT} below, we will also need a discrete version of the large sieve estimate, in which we sum over well-spaced sets. We say that a set $\mathcal{T}\subset \mathbb{R}$ is \emph{well-spaced} if $t,u\in \mathcal{T}$, $t\neq u$ implies $|t-u|\geq 1$. We give two such results below, one of which is sensitive to sparse families of characters.

\begin{lem}[Discrete hybrid large sieve for characters]\label{DISCMVT}
Let $T,N,q\geq 1$, and let $\mc{T} \subset [-T,T]$ be a well-spaced set. Then 
$$\sum_{\chi \pmod q}\sum_{t \in \mc{T}} \Big|\sum_{n \leq N} a_n\chi(n)n^{it}\Big|^2 \ll \Big(\phi(q)T + \frac{\phi(q)}{q}N\Big)\log(3N)\sum_{\substack{n \leq N\\ (n,q) = 1}} |a_n|^2.$$
\end{lem}
\begin{proof}
This result, which is a slight sharpening of~\cite[Theorem 7.4]{MonLS} (taking $\delta = 1$ there), is proved in a standard way by combining Gallagher's Sobolev-type lemma~\cite[Lemma 9.3]{iw-kow} with Lemma~\ref{L2MVTwithT}; we leave the details to the reader.
\end{proof}

\begin{lem}[Hal\'{a}sz--Montgomery large values estimate] \label{HALINT}
Let $T, q \geq 1$ and let $\mc{E}\subset \{\chi\pmod q\}\times [-T,T]$ be such that if $t \neq u$ and $(\chi,t), (\chi,u)\in \mc{E}$ then $|t-u|\geq 1$.
Then
$$\sum_{(\chi,t)\in \mc{E}} \Big|\sum_{n \leq N} a_n\chi(n)n^{it}\Big|^2 \ll \Big(\frac{\phi(q)}{q}N + |\mc{E}|(qT)^{1/2}\log(2qT)\Big)\sum_{\substack{n \leq N\\ (n,q) = 1}}|a_n|^2.$$
\end{lem}

\begin{proof}
This is a slight sharpening (paying attention to coprimality with $q$) of~\cite[Theorem 8.3]{MonLS} (see especially (8.16), taking $\delta = 1$ and $\sg_0 = 0$), and is proven in much the same way. We leave the details to the interested reader.
\end{proof}

When it comes to estimating the size of the large values set of a short twisted character sum supported on the primes, the following hybrid version of~\cite[Lemma 8]{mr-annals} will be important.

\begin{lem}[Basic large values estimate -- prime support]\label{PRIMMVTwithT}
Let $P,T\geq 2$. Let $\mc{T} \subset [-T,T]$ be well-spaced. Let 
$$P_{\chi}(s) \coloneqq  \sum_{P < p \leq 2P} a_p\chi(p)p^{-s},$$ 
where $|a_p| \leq 1$ for all $P < p \leq 2P$. Then for any $\alpha \in [0,1]$ we have 
\begin{align*}
&|\{(\chi,t) \in \{\chi' \pmod q\} \times \mc{T}: |P_{\chi}(it)| \geq P^{1-\alpha}\}|\ll (qT)^{2\alpha}\Big(P^{2\alpha}+\exp\Big(100\frac{\log(qT)}{\log P}\log\log (qT)\Big)\Big).
\end{align*}
\end{lem}

\begin{proof}
Without loss of generality, we may assume that $P$ and $T$ are larger than any given constant. Let $N$ be the number of pairs $(\chi,t)$ in question and $V\coloneqq P^{1-\alpha}$; then
\begin{align*}
N\leq  V^{-2k}\sum_{\chi \pmod q} \sum_{t \in \mc{T}} |P_{\chi}(it)|^{2k}    
\end{align*}
for any $k\geq 1$. We pick $k=\lceil \frac{\log(qT)}{\log P}\rceil$. Expanding out, we see that
\begin{align*}
P_{\chi}(s)^k = \sum_{P^k < n \leq (2P)^k} b(n) \chi(n)n^{-s},\quad \textnormal{where}\quad b(n)=\sum_{\substack{p_1\cdots p_k = n\\ p_j\in [P,2P]\,\, \forall j}} a_{p_1}\cdots a_{p_k}.
\end{align*}
By the discrete large sieve (Lemma~\ref{DISCMVT}), we have
\begin{align*}
\sum_{\chi \pmod q} \sum_{t \in \mc{T}} |P_{\chi}(it)|^{2k} 
&\ll (\phi(q)T+(2P)^k)\log(3\cdot (2P)^k) \sum_{P^k\leq n \leq (2P)^k} |b(n)|^2.
\end{align*}

We can then compute the mean square over $n$ as 
$$\sum_{P^k \leq  n \leq (2P)^k} |b(n)|^2 \leq \sum_{\substack{p_1 \cdots p_{k} = q_1 \cdots q_k\\ P\leq p_j,q_j \leq 2P}} 1\leq k!\Big(\sum_{P < p \leq 2P} 1\Big)^k  \leq k!\Big(\frac{2P}{\log P}\Big)^k.$$
This gives the bound
\begin{align*}
\sum_{\chi \pmod q} \sum_{t \in \mc{T}}|P_{\chi}(it)|^{2k} &\ll k!(\phi(q)T+(2P)^k)\log((2P)^{k+1})\Big(\frac{2P}{\log P}\Big)^k \\
&\leq (k+1)!\log(2P)\Big(1+\frac{\phi(q)T}{(2P)^k}\Big)\Big(\frac{4P^2}{\log P}\Big)^k.
\end{align*}
Multiplying this by $V^{-2k}$ and recalling the choices of $V$ and $k$, this becomes
\begin{align*}
\ll (qT)^{2\alpha}P^{2\alpha}\Big(\frac{8k}{\log P}\Big)^{k-1}.  
\end{align*}
If $\log P\geq 8k$ then this bound is $\ll (qT)^{2\alpha} P^{2\alpha}$; otherwise, we obtain the bound $\ll (qT)^{2\alpha}(e^{20}k)^k$ (for $P$ large enough). Together, these two bounds imply the claim.
\end{proof}

The proofs of the next two lemmas are almost identical to the proofs of the corresponding results in~\cite{mr-annals}, with the following small modifications. Firstly, one applies Lemma~\ref{L2MVTwithT}, rather than  the mean value theorem for Dirichlet polynomials. Secondly, the corresponding Dirichlet polynomials are considered on the zero line rather than the one line. Finally, the coefficients are supported on the integers $(n,q)=1$ which accounts for the extra factor $\phi(q)/q.$ We give the proof of one of them to illustrate the changes needed.

\begin{lem}\label{LEM13withT}
Let $q,T\geq 1$, $2 \leq Y_1 \leq Y_2$ and $\ell \coloneqq  \left\lceil \frac{\log Y_2}{\log Y_1}\right\rceil$. For $a_m$, $c_p$ $1$-bounded complex numbers, define
\begin{align*}
Q(\chi,s) \coloneqq  \sum_{Y_1\leq p\leq 2Y_1} c_p\chi(p)p^{-s}\quad \textnormal{and}\quad A(\chi,s) \coloneqq  \sum_{X/Y_2\leq m\leq 2X/Y_2} a_m\chi(m)m^{-s}.    
\end{align*} 
Then
$$\sum_{\chi \pmod q}\int_{-T}^T |Q(\chi,it)^{\ell}A(\chi,it)|^2 dt \ll\frac{\phi(q)}{q}XY_12^{\ell}\Big(\phi(q)T + \frac{\phi(q)}{q}XY_1 2^{\ell}\Big)(\ell+1)!^2.$$

Moreover, we have the same bound for
$$\sum_{\chi \pmod q} |Q(\chi,0)^{\ell}A(\chi,0)|^2 $$
when we put $T=1$ on the right-hand side.
\end{lem}

\begin{proof}
This is analogous to~\cite[Lemma 13]{mr-annals}. The Dirichlet polynomial $Q(\chi,s)^{\ell}A({\chi},s)$ has its coefficients supported on the interval
\[[Y_1^{\ell}\cdot X/Y_2,(2Y_1)^{\ell}\cdot 2X/Y_2]\subset [X,2^{\ell+1}Y_1X].\]
We now apply Lemma~\ref{L2MVTwithT} to arrive at
$$\sum_{\chi \pmod q}\int_{-T}^T |Q({\chi},it)^{\ell}A({\chi},it)|^2 dt\ll
\Big(\phi(q)T+\frac{\phi(q)}{q}2^{\ell}Y_1 X\Big)\sum_{\substack{X\le n\le 2^{\ell+1}Y_1X\\(n,q)=1}}\Big(\sum_{\substack{n=mp_1\cdots p_{\ell}\\Y_1\le p_1,\ldots, p_{\ell}\le 2Y_1,\\ X/Y_2\le m\le 2X/Y_2}}1\Big)^2.$$

We note that, for each $n$ in the outer sum, we have 
\[\sum_{\substack{n=mp_1\dots p_{\ell}\\Y_1\le p_1\ldots p_{\ell}\le 2Y_1,\\ X/Y_2\le m\le 2X/Y_2}}1\le \ell!\cdot \sum_{\substack{n=mr\\ p\vert r\Longrightarrow Y_1\le p\le 2Y_1}}1\coloneqq \ell!g(n)\]
where $g(n)$ is a multiplicative function defined by $g(p^k)=k+1$ for $Y_1\le p\le 2Y_1$ and $g(p^k)=1$ otherwise. Consequently,
\begin{equation}\label{intermediate}
\sum_{\chi \pmod q}\int_{-T}^T |Q({\chi},it)^{\ell}A({\chi},it)|^2 dt\ll
\Big(\phi(q)T+\frac{\phi(q)}{q}2^{\ell}Y_1X\Big)(\ell!)^2\sum_{\substack{X\le n\le 2^{\ell+1}Y_1X\\ (n,q)=1}}g(n)^2.
\end{equation}
Shiu's bound~\cite[Theorem 1]{shiu} in dyadic ranges yields
\begin{equation}\label{shiu-dyadic}
\sum_{\substack{Y\le n\le 2Y\\ (n,q)=1}}g(n)^2\ll Y\frac{\phi(q)}{q}\prod_{\substack{p\le Y\\p\nmid q}}\Big(1+\frac{|g(p)|^2-1}{p}\Big)\ll Y\frac{\phi(q)}{q}.
\end{equation}

We now split the right-hand side of~\eqref{intermediate} into dyadic ranges, apply~\eqref{shiu-dyadic} to each of them and sum the results up to finish the proof of the first claim. The second claim is proven in the same way, but using Lemma~\ref{SIMPLEORTHO} in place of Lemma~\ref{L2MVTwithT}.
\end{proof}

\begin{lem}\label{LEM12withT}
Let $X\geq H \geq 1$, $Q \geq P \geq 1$. Let $a_m,b_m,c_p$ be $1$-bounded sequences with $a_{mp} = b_mc_p$ whenever $p \nmid m$ and $P \leq p \leq Q$. Let $\Xi$ be a collection of Dirichlet characters modulo $q\geq 1$. Let \begin{align*}
Q_{v,H}(\chi,s) \coloneqq  \sum_{\substack{P\leq p\leq Q\\e^{v/H}\leq p \leq e^{(v+1)/H}}} c_p\chi(p)p^{-s},    
\end{align*}
and 
\begin{align*}
R_{v,H}(\chi,s)\coloneqq  \sum_{Xe^{-v/H}\leq m\leq 2Xe^{-v/H}} b_m\chi(m)m^{-s} \cdot \frac{1}{1+\omega_{[P,Q]}(m)},
\end{align*}
for each $\chi \in \Xi$ and $v\geq 0$. Let $\mc{T} \subset [-T,T]$ be measurable, and $\mc{I} \coloneqq  \{j \in \mb{Z} : \lfloor H \log P\rfloor \leq j\leq H \log Q\}$. Then
\begin{align*}
&\sum_{\chi \in \Xi} \int_{\mc{T}} \Big|\sum_{n \leq X} a_n\chi(n)n^{-it}\Big|^2 dt \ll H \log\Big(\frac{Q}{P}\Big) \sum_{j \in \mc{I}} \sum_{\chi \in \Xi} \int_{\mc{T}} \Big|Q_{j,H}(\chi,it)R_{j,H}(\chi,it)\Big|^2 dt \\
&+ \frac{\phi(q)}{q}X\Big(\phi(q)T+\frac{\phi(q)}{q}X\Big)\Big(\frac{1}{H}+\frac{1}{P}\Big) + \frac{\phi(q)}{q}X\Big(\sum_{\substack{n \leq X\\ (n,q) = 1}} |a_n|^2 1_{(n,\mc{P}) = 1}\Big),
\end{align*}
where $\mc{P}\coloneqq  \prod_{P \leq p \leq Q} p$.

Moreover, the same bound holds for 
\begin{align} \label{eq:noTL2}
\sum_{\chi \in \Xi}  \Big|\sum_{n \leq X} a_n\chi(n)\Big|^2    
\end{align}
with $T=1$ and the integration removed on the right-hand side.
\end{lem}

\begin{proof}
The proof is almost identical to the proof of~\cite[Lemma 12]{mr-annals}, the only slight difference being that after splitting the sum involving $a_n$ into short sums, one estimates the error terms by applying Lemma~\ref{L2MVTwithT} (or Lemma~\ref{SIMPLEORTHO} in the case of~\eqref{eq:noTL2}) instead of the mean value theorem for Dirichlet polynomials.
\end{proof}

\section{Lemmas on multiplicative functions}

\label{sec: multiplicative}

\subsection{Preliminaries}

Throughout this section, given $t \in \mb{R}$ we set
\begin{align}\label{eq_Vt}
V_t \coloneqq   \exp\Big(\log(3+|t|)^{2/3}\log\log(3+|t|)^{1/3}\Big).    
\end{align}
For $y \geq 2$, $\text{Re}(s) > 1$, and a multiplicative $f\colon \mathbb{N}\to \mathbb{U}$, we define 
\begin{align}\label{eq_Ly}
L_y(s,f) \coloneqq   \prod_{p > y} \sum_{k \geq 0} \frac{f(p^k)}{p^{ks}}.
\end{align}
Also recall the definition of the $\mathbb{D}_q$ distance from~\eqref{eqq123}, and let $\mathbb{D}\coloneqq  \mathbb{D}_1$.

We begin with two estimates for $L_y(s,f)$ from the work of Koukoulopoulos~\cite{kou}.

\begin{lemma}[Relating $L_y(s,f)$ to pretentious distance]\label{le3.2}
Let $x,y\geq 2$, $t\in \mathbb{R}$, and let $f\colon \mathbb{N}\to \mathbb{U}$ be multiplicative. Then
\begin{align*}
\log \left|L_y\left(1+\frac{1}{\log x}+it,f\right)\right|=\Re\left(\sum_{y<p\leq x}\frac{f(p)p^{-it}}{p}\right)+O(1). 
\end{align*}
\end{lemma}

\begin{proof}
This is~\cite[Lemma 3.2]{kou}.
\end{proof}

\begin{lemma}[Bounding $L_y(s,\chi)$]\label{le4.2}
Let $\varepsilon>0$. Let $q\geq 1$ and $s=\sigma+it$ with $\sigma>1$ and $t\in \mathbb{R}$. Let $y\geq qV_t$, and let $\chi\pmod q$ be a character. Then, if $|t|\geq \varepsilon/\log y$ or if $\chi$ is complex, we have $|L_y(s,\chi)|\asymp_{\varepsilon} 1$.
\end{lemma}

\begin{proof}
This is~\cite[Lemma 4.2]{kou}.
\end{proof}

In this section and the next, we need estimates for the count of zeros of  $L(s,\chi)$, namely
\begin{align}\label{eq_N(sigma)}
N(\sigma,T,\chi)\coloneqq \sum_{\substack{\rho:\,\,L(\rho,\chi)=0\\\Re(\rho)\geq \sigma\\|\Im(\rho)|\leq T}}1,    
\end{align}
where multiple zeros are counted according to their multiplicities. 

\begin{lemma}[Log-free zero-density estimate]\label{le_zerodensity} For $Q,T\geq 1$, $\frac{1}{2}\leq \sigma\leq 1$ and $\varepsilon>0$, we have 
\begin{align*}
\sum_{q\leq Q}\,\,\asum_{\chi\pmod q} N(\sigma,T,\chi)\ll_{\varepsilon} (Q^2T)^{(\frac{12}{5}+\varepsilon)(1-\sigma)}.    
\end{align*}
\end{lemma}

\begin{proof}
This is well known (see `Zeros Result 1 (iv)' in~\cite{harper-bv}). For $\frac{1}{2}\leq \sigma\leq 4/5$, say, the lemma follows from the work of Huxley~\cite{huxley}, whereas in the complementary region we can apply Jutila's log-free zero-density estimate~\cite{jutila} (with $12/5+\varepsilon$ replaced with the better exponent $2+\varepsilon$).
\end{proof}

\subsection{General estimates for partial sums of multiplicative functions}
In this subsection we collect various estimates for partial sums of $1$-bounded multiplicative functions.
\begin{lemma}[A Hal\'asz-type inequality] \label{le_hal_hyb}
Let $x \geq 10$ and $1 \leq q,T \leq 10x$. Let $f\colon \mb{N} \ra \mb{U}$ be a multiplicative function. Then
$$\frac{1}{x}\sum_{\substack{n \leq x\\ (n,q) = 1}} f(n)\ll \frac{\phi(q)}{q}\Big((M_q(T)+1)e^{-M_q(T)} + \frac{1}{\sqrt{T}} + (\log x)^{-1/4}\Big),$$
where 
\begin{align}\label{eq_Mq}
M_q(T) = M_q(f;x,T)\coloneqq  \inf_{|t|\leq T}\mathbb{D}_{q}(f,n^{it};x)^2.    
\end{align} 
\end{lemma}

\begin{proof}
We may assume that $T\leq \sqrt{\log x}$, since otherwise we can use  $M_q(T)\leq M_q(\sqrt{\log x})$ and the fact that $y\mapsto (y+1)e^{-y}$ is decreasing to reduce to the case $T=\sqrt{\log x}$. But then the claim follows\footnote{In~\cite[Corollary 2.2]{bgs}, it is assumed that $q\leq \sqrt{x}$, but the same proof works for $q\leq 10x$.} from~\cite[Corollary 2.2]{bgs}. 
\end{proof}

We also need a version of Hal\'asz's inequality that is sharp for sums that are restricted to rough numbers (i.e., integers having only large prime factors). This will be employed in the proof of Lemma~\ref{pls_hyb}.

\begin{lemma}[Hal\'asz over rough numbers] \label{le_ghs_rough}
Let $2 \leq y \leq x$, and let $f\colon \mb{N} \ra \mb{U}$ be multiplicative. Then
\begin{align*}
\frac{1}{x}\sum_{\substack{n \leq x\\P^{-}(n)>y}} f(n) \ll \frac{(1+M(f;(y,x],\frac{\log x}{\log y}))e^{-M(f;(y,x],\frac{\log x}{\log y})}}{\log y} + \frac{1}{\log x},
\end{align*}
where $M(f;(y,x],T)$ is defined for $T\geq 0$ by 
\begin{align*}
 M(f;(y,x],T) \coloneqq   \inf_{|t| \leq T} \mb{D}(f,n^{it};y,x)^2
\end{align*}
with 
\begin{align}\label{eqq70}
\mb{D}(f,g;y,x) \coloneqq   \Big(\sum_{y < p \leq x} \frac{1-\Re(f(p)\overline{g(p)})}{p}\Big)^{1/2}.
\end{align}
\end{lemma}

\begin{proof}
Without loss of generality, we may assume that $f(p^k)=0$ for all primes $p\leq y$ and all $k\geq 1$. We may also assume that $y\leq x^{1/2}$, since otherwise the estimate follows trivially from the prime number theorem.

A consequence of~\cite[Proposition 7.1]{ghs} (see in particular formula (7.3) there) implies that
\begin{align}\label{eq:mf}
\sum_{n \leq x} f(n) \ll (1+M)e^{-M}\frac{x}{\log y} + \frac{x}{\log x},
\end{align}
where $M$ is defined implicitly via
\begin{align*}
\sup_{|t| \leq \frac{\log x}{\log y}} \Big|\frac{F(1+1/\log x + it)}{1+1/\log x + it}\Big| = e^{-M} \frac{\log x}{\log y},
\end{align*}
where $F(s) \coloneqq   \prod_{p} \sum_{k \geq 0} f(p^k)/p^{ks}$ for $\text{Re}(s) > 1$. On the other hand, as $f(p^k) = 0$ for all $p \leq y$, by Lemma~\ref{le3.2} for any $t \in \mb{R}$ we have
\begin{align*}
|F(1+1/\log x+it)|\frac{\log y}{\log x} \asymp \exp\Big(-\sum_{y < p \leq x} \frac{1-\text{Re}(f(p)p^{-it})}{p}\Big) = e^{-\mb{D}(f,n^{it};y,x)^2},
\end{align*}
so that
\begin{align*}
e^{-M} \ll \sup_{|t| \leq \frac{\log x}{\log y}} \frac{e^{-\mb{D}(f,n^{it};y,x)^2}}{|1+1/\log x+it|} \ll e^{-M(f;(y,x],\frac{\log x}{\log y})}.
\end{align*}
In particular, $M(f;(y,x],\frac{\log x}{\log y}) \leq M + O(1)$. 

Since $t \mapsto (1+t)e^{-t}$ is decreasing, it now follows from~\eqref{eq:mf} that
\begin{align*}
\sum_{n \leq x} f(n) \ll \left(1+M\Big(f;(y,x],\frac{\log x}{\log y}\Big)\right)e^{-M(f;(y,x],\frac{\log x}{\log y})}\frac{x}{\log y} + \frac{x}{\log x},
\end{align*}
as claimed.
\end{proof}

In the proof of Theorem~\ref{MRAPHybrid}, we will also need the following three lemmas.

\begin{lem}[Twisting by $n^{it}$] \label{lem_twist}
 Let $\alpha \in \mb{R}$. Then for any $x \geq 3$ and any multiplicative $f\colon \mb{N} \ra \mb{U}$,
\begin{align*}
\frac{1}{x}\sum_{n \leq x} f(n)n^{i\alpha} = \frac{x^{i\alpha}}{1+i\alpha}\frac{1}{x}\sum_{n \leq x} f(n) + O\left(\frac{\log(2+|\alpha|)}{\log x}\exp\left(\mb{D}(f,1;x) \sqrt{(2+o(1)) \log\log x}\right)\right).
\end{align*}
\end{lem}
\begin{proof}
From~\cite[Lemma 7.1]{gs-decay}, we have the claimed estimate with the error term
$$
O\left(\frac{\log(2+|\alpha|)}{\log x}\exp\left(\sum_{p\leq x}\frac{|1-f(p)|}{p} \right)\right).
$$ 
Hence, the claim follows from 
$$
\sum_{p \leq x} \frac{|1-f(p)|}{p} \leq\left(\sum_{p\leq x}\frac{1}{p}\right)^{\frac{1}{2}} \left(\sum_{p \leq x} \frac{|1-f(p)|^2}{p}\right)^{\frac{1}{2}} \leq \left(\log\log x +O(1)\right)^{\frac{1}{2}} \left(2\sum_{p \leq x} \frac{1-\text{Re}(f(p))}{p}\right)^{\frac{1}{2}}.
$$
\end{proof}

\begin{lem}[Simplifying a Perron integral] \label{lem_MTHybrid}
Let $X,Z \geq 10$, with $1 \leq Z \leq (\log X)^{1/20}$. Let $1 \leq h \leq X$, and let $1 \leq q \leq h/10$. Let $g\colon \mb{N} \ra \mb{U}$ be multiplicative, and let $t_0$ be a minimizer of $t \mapsto \mb{D}(g,n^{it};X)$ on $|t| \leq X$. Then for every $x\in [X,2X]$ we have 
\begin{align*}
&\frac{1}{2\pi h} \int_{t_0-Z}^{t_0+Z} \Big(\sum_{\substack{n \leq 3X \\ (n,q) = 1}} g(n)n^{-it}\Big) \frac{(x+h)^{it}-x^{it}}{it}\, dt= \Big(\frac{1}{3hX}\sum_{\substack{n \leq 3X \\ (n,q) = 1}} g(n)n^{-it_0}\Big) \int_x^{x+h} v^{it_0} dv + O\Big(\frac{\phi(q)}{qZ^{1/2}}\Big).
\end{align*}
\end{lem}
\begin{proof}
We note that $\frac{(x+h)^{it}-x^{it}}{it} = \int_x^{x+h} v^{-1+it}dv$ for each $t \in [t_0-Z,t_0+Z]$. Inserting this into the left-hand side of the statement, swapping the orders of integration and making the change of variables $u \coloneqq   t-t_0$, we obtain
\begin{equation}\label{eq:doubleInt}
\frac{1}{2\pi h} \int_x^{x+h} v^{-1+it_0} \Big(\int_{-Z}^Z v^{iu} \sum_{\substack{n \leq 3X \\ (n,q) = 1}} g(n)n^{-it_0-iu}\, du \Big)\, dv.
\end{equation}

Let $M\coloneqq   \min_{|u| \leq \frac{1}{2}\log X} \mb{D}_q(g,n^{i(t_0+u)};X)^2$. By Lemma~\ref{le_hal_hyb}, if  $M\geq (1/4)\log\log X$, then
$$
\sup_{|u| \leq Z} \Big|\sum_{\substack{n \leq 3X \\ (n,q) = 1}} g(n)n^{-it_0-iu}\Big| \ll \frac{\varphi(q)}{q} (X(1+M)e^{-M}+X/(\log X)^{1/4}) \ll X(\log X)^{-1/4+o(1)},
$$
in which case the expression~\eqref{eq:doubleInt} can be bounded by
$$
\ll \frac{h}{hx} \cdot Z X(\log X)^{-1/4+o(1)} \ll (\log X)^{-1/5}
$$
for $X$ sufficiently large, given that $Z\leq (\log X)^{1/20}$. The claim follows in this case, so we may assume in the sequel that $M < (1/4)\log\log X$. 

Put $g_{t_0}(n)\coloneqq   g(n)n^{-it_0}$. Since $|u| \leq Z$, Lemma~\ref{lem_twist} yields
\begin{align*}
\sum_{\substack{n \leq 3X \\ (n,q) = 1}} g_{t_0}(n)n^{-iu} &= \frac{(3X)^{-iu}}{1-iu} \sum_{\substack{n \leq 3X \\ (n,q) = 1}} g_{t_0}(n) + O\Big(\frac{X (\log(2Z))}{\log X} e^{\sqrt{(2+o(1))M\log\log X}}\Big) \\
&=\frac{(3X)^{-iu}}{1-iu} \sum_{\substack{n \leq 3X \\ (n,q) = 1}} g_{t_0}(n) + O\Big(\frac{X}{(\log X)^{0.29-o(1)}}\Big),
\end{align*}
as $\sqrt{1/2}-1 < -0.29$. Furthermore, $0.29-1/20 > 1/5$, so upon inserting this estimate into~\eqref{eq:doubleInt} that expression becomes
\begin{equation}\label{eq:PerronType}
\Big(\sum_{\substack{n \leq 3X \\ (n,q) = 1}} g(n)n^{-it_0}\Big)\int_x^{x+h} v^{-1+it_0} \frac{I(v;3X)}{h}dv + O\Big(\frac{1}{(\log X)^{1/5}}\Big),
\end{equation}
where for $y \geq 1$ we have defined
$$
I(v;y) \coloneqq   \frac{1}{2\pi} \int_{-Z}^Z v^{iu}\frac{y^{-iu}}{1-iu} du.
$$

Using a standard, truncated version of Perron's formula (e.g.,~\cite[Proposition 5.54]{iw-kow}), if $y \neq v$ then
\begin{align*}
I(v;y) 
&= \frac{v}{y} \Big(\frac{1}{2\pi i} \int_{\text{Re}(s) = 1} \frac{(y/v)^s}{s}ds + O\Big(\frac{y/v}{Z|\log(y/v)|}\Big)\Big) \\
&= \frac{v}{y} 1_{y > v}  + O\Big(\frac{1}{Z|\log(y/v)|}\Big).
\end{align*}
As $[x,x+h]$ is disjoint from $[3X-\tfrac{3X}{\sqrt{Z}}, 3X + \tfrac{3X}{\sqrt{Z}}]$, we have $Z|\log(3X/v)| \gg Z^{1/2}$ for all $v \in [x,x+h]$ and~\eqref{eq:PerronType} becomes
\begin{align*}
&(3hX)^{-1}\Big(\sum_{\substack{n \leq 3X \\ (n,q) = 1}} g_{t_0}(n)\Big) \int_x^{x+h} v^{it_0} dv +O\Big(\frac{\phi(q)}{qhZ} \int_x^{x+h} \frac{dv}{|\log(3X/v)|}+\frac{1}{(\log X)^{1/5}}\Big) \\
&= (3hX)^{-1}\Big(\sum_{\substack{n \leq 3X \\ (n,q) = 1}} g_{t_0}(n)\Big)\int_x^{x+h} v^{it_0} dv + O\Big(\frac{\phi(q)}{qZ^{1/2}}+\frac{1}{(\log X)^{1/5}}\Big).
\end{align*}
Since $(\log X)^{-1/5}\ll \varphi(q)/(qZ^{1/2})$, this completes the proof.
\end{proof}

\subsection{Bounds on prime sums of twisted Dirichlet characters}
The following lower bound on the pretentious distance $\mb{D}$ between Dirichlet and Archimedean characters will enable us to show that $f$ can only correlate significantly with at most one Dirichlet character $\chi \pmod{q}$, which must then be $\chi_1$ (see~Proposition~\ref{prop_sup_hyb}).
\begin{lemma}[A pretentious distance bound]\label{le_L1chi_hyb}
Let $x \geq 10$, $1 \leq q \leq x$, and let $\chi$ be any non-principal Dirichlet character modulo $q$ induced by a primitive character $\chi^{\ast}$ modulo $q^{\ast}$. Then
\begin{align*}
\inf_{|t| \leq 10x} \mb{D}_q(\chi,n^{it};x)^2 \geq \frac{1}{4}\log\Big(\frac{\log x}{\log(2q^{\ast})}\Big) + O(1).
\end{align*}
\end{lemma}

\begin{remark}
 For the purpose of proving Theorem~\ref{MRAPTHM}, our estimates only require uniformity in the $t$-aspect for $|t| \leq \log x$, and in that regime Lemma \ref{le_L1chi_hyb} is easier to prove. However, in order to prove Theorem~\ref{MRAPHybrid}, we will need full uniformity in the much larger range $|t| \ll x$. The same remark applies to Lemma~\ref{pls_hyb} and several lemmas in Section~\ref{sec: propositions}.
\end{remark}

\begin{proof}
We may assume that $x$ is larger than any fixed absolute constant, since otherwise the bound is trivial upon choosing the term $O(1)$ appropriately. 
Let $t_0$ be a minimizer for the map $t\mapsto \mb{D}(\chi,n^{it};x)$ on $[-10x,10x]$. We split the proof of the lemma into two cases. 

\textbf{Case 1.} If $|t_0|\leq \log x$, then the claim follows directly from~\cite[Lemma 3.4]{bgs}.

\textbf{Case 2.} Next assume that  $|t_0|>\log x$. Let us write $\chi(n)=\chi^{*}(n)1_{(n,r)=1}$, where $\chi^{*}\pmod{q^{*}}$ induces $\chi$ and $(r,q^{*})=1$.  Let $y \coloneqq   q^{\ast}V_{10x}$; then we have $V_{10x}\leq y\leq \max\{ (q^{\ast})^{2},V_{10x}^2\}$. 

We now observe that, since $q^{*}<y$, we have
\begin{align*}
\mb{D}_q(\chi,n^{it_0};x)^2 &\geq  \text{Re}\Big(\sum_{y < p \leq x} \frac{1-\chi^{*}(p)p^{-it_0}}{p}\Big)-O\Big(\sum_{\substack{p\mid q\\p\geq y}}\frac{1}{p}\Big)\\
&= \log\Big(\frac{\log x}{\log y}\Big) - \log\Big|L_y\Big(1+\frac{1}{\log x} + it_0,\chi^{*}\Big)\Big| + O(1),
\end{align*}
where for the last line we used Lemma~\ref{le3.2} and the crude estimate $\sum_{p\mid q}1\ll  \log x=o(y)$.

Recalling $\log x \leq |t_0| \leq 10x$ and our choice of $y$, Lemma~\ref{le4.2} gives $|L_y(1+1/\log x+it_0,\chi^{\ast})| \asymp 1$. It follows that
\begin{align*}
\mb{D}_q(\chi,n^{it_0};x)^2 &\geq \log\Big(\frac{\log x}{\log y}\Big) + O(1)\geq \frac{1}{4}\log\Big(\frac{\log x}{\log(2q^{\ast})}\Big)  + O(1),
\end{align*}
where for the last inequality we used $y\leq \max\{(q^{*})^2,V_{10x}^2\}$.
\end{proof}

The following pointwise bound for twisted character sums over primes will be needed in the proof of Proposition~\ref{prop_largevalues_hyb}.

\begin{lemma}[Character sums over primes]\label{pls_hyb}
Let $x\geq 10$, $X=x^{(\log x)^{1/25}}$, and $1\leq q\leq x$. Let $h$ be a fixed smooth function supported on $[1/2,4]$. Then, for $\varepsilon\in (0,1)$ and for any character $\chi\pmod q$ with $\cond(\chi)\leq x^{\varepsilon}$, uniformly in the range $|t|\leq X$ we have
\begin{align}\label{eqq17}
\Big|\sum_{n}\Lambda(n)\chi(n)n^{-it}h\Big(\frac{n}{x}\Big)\Big| \ll_{h} \varepsilon \Big(\log^3 \frac{1}{\varepsilon}\Big)x +\frac{x}{(\log x)^{0.3}}+\frac{x}{t^2+1}.
\end{align}
Moreover, the $\frac{x}{t^2+1}$ term can be deleted for all but possibly one non-principal $\chi\pmod q$, and this $\chi$ (if it exists) must be real and satisfy $L(\beta,\chi)=0$ for some real $\beta>1-c_0/(\log q)$ for some absolute constant $c_0>0$.
\end{lemma}

\begin{remark}\label{rem2}
By looking at the proof of Lemma~\ref{pls_hyb}, it is clear that~\eqref{eqq17} works also for the sharp weight $h(u)=1_{u\in (0,1]}$ if $x/(t^2+1)$ is replaced with $x/(|t|+1)$ there. The $1/(t^2+1)$ decay is helpful when we apply Lemma~\ref{pls_hyb} in the proof of Proposition~\ref{prop_largevalues_hyb} to ensure that when~\eqref{eqq17} is summed over a well-spaced set of $t$ the resulting bound will not be too large.
\end{remark}

\begin{proof} Without loss of generality, we may assume that $x$ is larger than any given constant,  that $\varepsilon \geq (\log x)^{-0.4}$, and that $\varepsilon$ is smaller than any fixed constant. If $\chi$ is induced by $\chi^{\ast}\pmod{q^{\ast}}$, we have
\begin{align*}
\sum_{n}\Lambda(n)\chi(n)n^{-it}h\Big(\frac{n}{x}\Big)=\sum_{n}\Lambda(n)\chi^{\ast}(n)n^{-it}h\Big(\frac{n}{x}\Big)+O_h((\log x)^2),    
\end{align*}
and as the error term is small, may assume that $\chi$ is primitive and $q=q^{\ast}$.

We split into cases depending on the sizes of $q$ and $t$. 

\textbf{Case 1.} Suppose first that $q=1$. Then $\chi$ is identically $1$, and in that case by Mellin inversion we have
\begin{align*}
\sum_{n}\Lambda(n)n^{-it}h\Big(\frac{n}{x}\Big)=-\frac{1}{2\pi i}\int_{2-i\infty}^{2+i\infty} \frac{\zeta'}{\zeta}(s+it)\widetilde{h}(s)x^s\, ds.   
\end{align*}
Since $h$ is smooth, its Mellin transform $\tilde{h}$ satisfies $|\widetilde{h}(s)|\ll_h 1/(1+|s|^{10})$ for $\Re(s)\in [-100,100]$. Hence, shifting the line of integration to $\Re(s)=b:=1-(\log x)^{1/10}$, we obtain
\begin{align*}
\sum_{n}\Lambda(n)n^{-it}h\Big(\frac{n}{x}\Big)=-\frac{1}{2\pi i}\int_{b-i\log x}^{b+i\log x} \frac{\zeta'}{\zeta}(s+it)\widetilde{h}(s)x^s\, ds+\widetilde{h}(1-it)x+O_h\left(\frac{x}{\log x}\right),    
\end{align*}
and this is $\ll x/(t^2+1)+O_h(x/\log x)$ by using the Vinogradov--Korobov bound $|\frac{\zeta'}{\zeta}(s+it)|\ll \log x$ in the region of the integrand. We may thus assume that $q^{\ast} \geq 2$.

\textbf{Case 2.} Suppose then that  $2\leq q^{\ast}\leq (\log x)^{10},|t|\leq (\log x)^{10}$. Then~\eqref{eqq17} follows straightforwardly from partial summation and the Siegel--Walfisz theorem (with a better bound of $\ll_h x(\log x)^{-100}$).

\textbf{Case 3.} Next, suppose $q^{\ast}>(\log x)^{10}$, $|t|\leq (\log x)^{10}$. We apply the explicit formula (proven similarly to~\cite[Proposition 5.25]{iw-kow}) 
\begin{align}\label{eq:explicit}
\sum_{n} \Lambda(n)\chi^{\ast}(n)n^{-it}h\Big(\frac{n}{x}\Big)=-\sum_{\substack{\substack{\rho=\beta+i\gamma:\\L(\rho,\chi^{\ast})=0\\|\gamma-it|\leq T\\0\leq \beta\leq 1}}}x^{\rho-it}\widetilde{h}(\rho-it)+O_h\Big(\frac{x}{T}(\log^3(qx(|t|+2)))\Big),    
\end{align}
where we choose $T=(\log x)^{100}$ to make the error term small.

 Let $D=\log(x^{\varepsilon}(|t|+2))$. Note that by the Landau--Page theorem~\cite[Theorem II.8.25]{Ten} we have the zero-free region $L(s,\chi^{\ast})\neq 0$ for $\Re(s)\geq 1-c_0/(\log D)$ for some constant $c_0>0$, apart from possibly one zero $\rho=\beta$, which has to be real and simple; additionally, such an exceptional zero can only exist for at most one character $\chi^{\ast}$ of conductor $\leq x^{\varepsilon}$, which has to be real and non-principal. Applying the bound $\widetilde{h}(s)\ll 1/(1+|s|^{10})$ for $\Re(s)\in [-100,100]$, the contribution of $\rho=\beta$ to the right of~\eqref{eq:explicit} is certainly 
\begin{align}\label{eqq136}
\ll_h \frac{x}{t^2+1},
\end{align}
which is admissible. Moreover, the contribution of $\Re(\rho)\leq 9/10$ to~\eqref{eq:explicit} is trivially $\ll x^{91/100}$.

By splitting the sum in~\eqref{eq:explicit} into pieces $\Re(s)\in (1-(k+1)c_0/(\log D),1-kc_0/\log D]$, $\Im(s)\in [T,2T]$, the sum becomes
\begin{align*}
\ll_h \frac{x}{t^2+1}+x^{91/100}+\sum_{1\leq k\leq (\log D)/(5c_0)}\,\sum_{\substack{T=2^j\\j\geq 0}}x^{1-\frac{kc_0}{\log D}}\frac{N\Big(1-\frac{kc_0}{\log D},2T,\chi^{\ast}\Big)}{T^{10}+1}.    
\end{align*}

The log-free zero density estimate (see Lemma~\ref{le_zerodensity}) allows us to bound this by
\begin{align*}
\ll_h \frac{x}{t^2+1}+x^{91/100}+\sum_{1\leq k\leq (\log D)/(5c_0)}\,\sum_{\substack{T=2^j\\j\geq 0}}x^{1-\frac{kc_0}{\log D}}\frac{(q^{\ast})^{6kc_0/(\log D)}}{T^{5}+1}\ll x^{-\frac{c_0}{2\log D}}.    
\end{align*}
by the geometric sum formula and the fact that $q^{\ast}\leq x^{1/1000}$. Noting that $x^{-c_0/(2\log D)}\ll \varepsilon^{100}$ for $|t|\leq (\log x)^{10}\leq q^{\ast}\leq x^{\varepsilon}$, this case has now been handled.

\textbf{Case 4.} We are left with the case $|t|>(\log x)^{10}$, $q^{\ast}\geq 2$. Since $q^\ast = \text{cond}(\chi) \leq x^{\e}$ by assumption, we may assume that $2\leq q^{\ast} \leq x^{1/50000}$ by selecting $\e$ smaller if necessary. Since $|t|$ is large, we no longer need the smoothing factor $h(n/x)$, and in fact by partial summation (and the fact that $h'$ is bounded) we see that~\eqref{eqq17} in the regime under consideration follows once we prove
\begin{align}\label{eqq140}
\Big|\sum_{n\leq x'}\Lambda(n)\chi(n)n^{-it}\Big|\ll \varepsilon\log^3\left(\frac{1}{\varepsilon}\right)x'+\frac{x'}{(\log x)^{0.3}}
\end{align}
for any $x' \in [x/2,4x]$. In what follows, for notational convenience we denote $x'$ by $x$. 

Put $y = (q^{\ast})^4V_X^{100}$, so that for $q^{\ast} \leq x^{1/50000}$ we have $y \leq x^{1/10000}$. We define 
\begin{align*}
\mu_y(m) &\coloneqq   \mu(m)1_{P^-(m) > y}, \\
\log_y m &\coloneqq   (\log m)1_{P^-(m) > y}, 
\end{align*}
and as in~\cite[Section 7]{ghs} we make use of the convolution identity
\begin{align*}
\Lambda(n)1_{P^{-}(n)>y}=\mu_y\ast \log_y(n),\quad n>y.   
\end{align*}
By the prime number theorem, we then see that for any $t \in \mb{R}$ we have
\begin{align*}
\sum_{n \leq x} \Lambda(n)\chi(n)n^{-it}&= \sum_{y^2 < n \leq x} \Lambda(n)\chi(n)n^{-it} + O(y^2)\\
&= \sum_{y^2 < md \leq x} \mu_y(m)\chi(m)m^{-it}\log_y(d)\chi(d)d^{-it} + O(y^2+x^{1/3}).
\end{align*}

Let $M,D\in [y,x]$ be parameters that satisfy $MD = x$, with $D \leq x^{1/2}$. Using the hyperbola method, we have
\begin{align*}
&\sum_{y^2<n\leq x}\Lambda(n)\chi(n)n^{-it}=T_1+T_2 + O(y^2+x^{1/3}),\\
T_1 &\coloneqq   \sum_{m \leq M} \mu_y(m)\chi(m)m^{-it}\sum_{y^2/m < d \leq x/m} \log_y(d) \chi(d)d^{-it} \\
T_2 &\coloneqq   \sum_{d \leq D} \log_y(d)\chi(d)d^{-it}\sum_{\substack{y^2/d < m \leq x/d\\m>M}} \mu_y(m)\chi(m)m^{-it}.
\end{align*}

We first deal with $T_2$. By Hal\'{a}sz's theorem for rough numbers (Lemma~\ref{le_ghs_rough}), for each $d \leq D$ the inner sum is
$$
\ll \frac{x}{d}\Big(\frac{(N+1)e^{-N}}{\log y} + \frac{\log \log M}{\log M}\Big) + \frac{y^2}{d},
$$
where we have defined
\begin{align*}
N \coloneqq   \inf_{|u| \leq \log x} \sum_{y<p \leq M} \frac{1-\text{Re}(\mu_y(p)\chi(p)p^{-i(t+u)})}{p}.
\end{align*}
As $D \leq x^{1/2}$, $M\geq x^{1/2}$ and $\mu_y(p)=-1_{p>y}$, it follows (as in the proof of Lemma~\ref{le_L1chi_hyb}) that
\begin{align}\label{eqq20}
N &\geq \inf_{|u| \leq \log x} \sum_{y < p \leq x} \frac{1+\text{Re}(\chi(p)p^{-i(t+u)})}{p} +O(1)\nonumber \\
&\geq \log \frac{\log x}{\log y}+ \inf_{|u| \leq \log x} \log|L_y(1+1/\log x + i(t+u),\chi)| + O(1).
\end{align}

Now since $y > q^{\ast}V_{2X}$, Lemma~\ref{le4.2} tells us that
\begin{align}\label{eqq21}
|L_y(1+1/\log x + iw,\chi)| \asymp 1
\end{align}
for $\chi$ complex and $|w|\leq 2X$, or for $\chi$ real and $1\leq |w|\leq 2X$. Note that since $|t|\geq (\log x)^{10}$ in~\eqref{eqq20} by assumption, we have $|t+u|\geq 1$ there, and thus~\eqref{eqq21} holds in any case for $w=t+u$, $|u|\leq \log x$.

The above implies that $N \geq \log((\log x)/(\log y)) - O(1)$. Hence, by partial summation and the estimate $\sum_{d\leq u}1_{P^{-}(d)>y}\ll u/(\log y)$ coming from Selberg's sieve, we have 
\begin{align*}
T_2 &\ll \frac{x\log \frac{\log x}{\log y}}{\log x}\sum_{\substack{d \leq D\\ P^-(d) > y}} \frac{\log d}{d} + (y\log x)^2\\
&\ll \frac{x\log \frac{\log x}{\log y}}{\log x}\Big(\frac{\log D}{\log y} + \int_y^D \Big(\sum_{\substack{d \leq u \\ P^-(d) > y}} 1\Big) \log\left( \frac{u}{e}\right)\frac{du}{u^2}\Big) +(y\log x)^2\\
&\ll x\frac{(\log D)^2\log \frac{\log x}{\log y}}{(\log x)(\log y)} + (y\log x)^2,
\end{align*}
for all non-principal characters $\chi$ modulo $q$ (recalling that $y\leq x^{1/10000}$).

We next estimate $T_1$. By partial summation, the inner sum in $T_1$, for each $m \leq M$, is
\begin{align*}
&\Big|\sum_{\substack{y^2/m < d \leq x/m\\ P^-(d) > y}} (\log d)\chi(d) d^{-it}\Big|\ll (\log x)\max_{y\leq u_1\leq u_2\leq x/m}\Big|\sum_{\substack{u_1\leq d\leq u_2\\P^{-}(d)>y}}\chi(d)d^{-it}\Big|\coloneqq  R(m).
\end{align*}

Recalling that $y=(q^{\ast})^4V_X^{100}$, we apply~\cite[Lemma 2.4]{kou} to the $R(m)$ terms, obtaining
$$
R(m) \ll \frac{\log x}{\log y} \Big((x/m)^{1-1/(30\log y)} + (x/m)^{1-1/(100\log V_t)}\Big),
$$
and since $y\geq V_t^{100}$, the second term can be ignored.

Summing over $m\leq M$, and using Selberg's sieve to bound the number of integers with $P^{-}(m)>y$, we conclude that $T_1$ is bounded by
\begin{align*}
\sum_{\substack{m \leq M\\ P^-(m) > y}} |R(m)|&\ll x\frac{\log x}{\log y}x^{-1/(30\log y)}\sum_{\substack{m\leq M\\P^{-}(m)>y}}m^{-1+1/(30\log y)}\\
&\ll x\Big(\frac{\log x}{\log y}\Big)^2\Big(\frac{x}{M}\Big)^{-1/(30\log y)}.
\end{align*}

Putting this all together and recalling $|t| \leq X$, we find that
\begin{align*}
T_1 &\ll x\Big(\frac{\log x}{\log y}\Big)^2\Big(\frac{x}{M}\Big)^{-1/(30\log y)},\\
T_2 &\ll x\frac{(\log (x/M))^2\log \frac{\log x}{\log y}}{(\log x)(\log y)}.
\end{align*}
We select $M=x/y^{1000\log(\log x/\log y)}\in [x^{1/2},x]$ (so in particular $y\leq x/M=D\leq x^{1/2}$, as required). Then $\log(x/M) = 1000 \log y \log\Big(\frac{\log x}{\log y}\Big)$ and thus, as $q^{\ast} \leq x^{1/10}$, we have
$$T_1 + T_2 \ll x\Big(\frac{\log y}{\log x}\Big)^{30}+x \frac{\log y}{\log x}\log^3\Big(\frac{\log x}{\log y}\Big)\ll x \frac{\log y}{\log x}\log^3\Big(\frac{\log x}{\log y}\Big).$$
If $q^{\ast} \leq V_X$ then $\log y \ll (\log x)^{0.694}$ for large enough $x$, and hence the bound reduces to $\ll x/(\log x)^{0.3}$. On the other hand, if $V_X < q^{\ast} \leq x^{\e}$ then the above bound becomes $\ll x \frac{\log q^{\ast}}{\log x} \log^3\Big(\frac{\log x}{\log q^{\ast}}\Big)\ll \varepsilon\log^3(1/\varepsilon)x$. This proves~\eqref{eqq140}, and thus completes the proof of the lemma.
\end{proof}

\section{Key propositions}

\label{sec: propositions}

The goal of this section is to prove two key propositions, namely Propositions~\ref{prop_sup_hyb} and~\ref{prop_largevalues_hyb}. For the proofs of both of these propositions, we will need good bounds on the number of Dirichlet characters whose $L$-functions have a bad zero-free region.

The log-free zero density estimate is easily employed to yield the following.

\begin{lemma}\label{le_Qset} Let $x\geq 10$, $\varepsilon\in ((\log x)^{-1/20},1)$, and $1/(\log \log x)\leq M\leq \e^{20}\log x/(20\log\log x)$, and define the set
\begin{align}\label{eqq22}
\mathcal{Q}_{x,\varepsilon,M}\coloneqq \Bigg\{q\leq x: \prod_{\substack{\chi\pmod q\\\cond(\chi)>x^{\varepsilon^{20}}}} L(s,\chi)\neq 0\quad \textnormal{for}\quad \Re(s)\geq 1-\frac{M(\log \log x)}{\log x},\quad |\Im(s)|\leq 3x\Bigg\}.   \end{align}
Then for $1\leq  Q\leq x$ we have $|[1,Q]\setminus \mathcal{Q}_{x,\varepsilon,M}|\ll Qx^{-\varepsilon^{20}/2}$. Moreover, there exists a set $\mathcal{B}_{x,\varepsilon,M}\subset [x^{\varepsilon^{20}},x]$ of size $\ll (\log x)^{10M}$ such that every integer in $[1,x]\setminus \mathcal{Q}_{x,\varepsilon,M}$ is a multiple of some element of $\mathcal{B}_{x,\varepsilon,M}$.
\end{lemma}

\begin{proof}
If $Q\leq x^{\varepsilon^{20}}$, then trivially $[1,Q]\cap \mathbb{Z}\subseteq \mathcal{Q}_{x,\varepsilon,M}$, so there is nothing to be proved. We may thus assume that $Q>x^{\varepsilon^{20}}$.

Let
\begin{align*}
\mathcal{B}_{x,M}\coloneqq  \Big\{q\leq x:\,\, \exists \, \sg \geq 1-\frac{M\log \log x}{\log x},\quad |t|\leq 3x \quad \text{ with } \prod_{\substack{\chi \pmod q\\\chi\,\, \textnormal{primitive}}}L(\sg + it,\chi) = 0\Big\}.      
\end{align*}
By Lemma~\ref{le_zerodensity}, we have 
\begin{align*}
|\mathcal{B}_{x,M}| \leq \sum_{q \leq x} \, \, \asum_{\chi \pmod{q}} N(1-\tfrac{M\log\log x}{\log x}, 3x, \chi) \ll (x^{3})^{(12/5+0.1)\tfrac{M\log \log x}{\log x}}\ll(\log x)^{10M}.    
\end{align*}

Since $L(s,\chi)$ and $L(s,\chi')$ have the same zeros in the region $\text{Re}(s) > 0$ if $\chi$ and $\chi'$ are induced by the same character, we see that every $q \leq x$ with $q\not \in \mathcal{Q}_{x,\varepsilon,M}$ is a multiple of some element of $\mathcal{B}_{x,\varepsilon,M} := \mathcal{B}_{x,M}\cap[x^{\varepsilon^{20}},x]$, and each such element has $\leq Qx^{-\varepsilon^{20}}+1$ multiples up to $Q$. Thus
\begin{align*}
|[1,Q]\setminus\mathcal{Q}_{x,\varepsilon,M}|\ll (\log x)^{10M}Qx^{-\varepsilon^{20}}\ll Qx^{-\varepsilon^{20}/2},    
\end{align*}
since $M\leq \e^{20}(\log x)/(20\log\log x)$, $\varepsilon>(\log x)^{-1/20}$, and $Q > x^{\varepsilon^{20}}$.
\end{proof}

The next lemma will be a crucial ingredient in the proof of Propositions~\ref{prop_sup_hyb} and~\ref{prop_largevalues_hyb}.

\begin{lemma}\label{le_primesums}
Let $x\geq 10$ and $(\log x)^{-1/50}\leq \varepsilon\leq 1$. For a character $\chi\pmod q$, let  
\begin{align*}
u_{\chi}=\begin{cases}1,\quad \chi\textnormal{ principal or } \chi\textnormal{ real and } \cond(\chi)\leq x^{\varepsilon^{20}}\\0,\quad \textnormal{otherwise.}\end{cases}\quad v_{\chi}=\begin{cases}1,\quad \chi\quad \textnormal{principal}\\
0,\quad \textnormal{otherwise}.
\end{cases}
\end{align*}Let  $q\in \mathcal{Q}_{x,\varepsilon,\varepsilon^{-6}}$.
\begin{enumerate}[label=\upshape(\roman*)]
    \item Uniformly for $x^{\varepsilon^{5.5}}\leq P\leq x$, we have
\begin{align*}
\sup_{\chi\pmod q}\sup_{\varepsilon^{-10}u_{\chi}\leq |t|\leq 2.1x}\Big|\sum_{n\leq P}\Lambda(n)\chi(n)n^{-it}\Big|\ll \varepsilon^{10}P.   
\end{align*}

\item We have 
    \begin{align*}
\inf_{\chi\pmod q}\inf_{\tfrac{v_{\chi}}{\sqrt{\log x}} \leq |t|\leq 2.1x}\mathbb{D}_q(\chi(n)n^{it},1;x)^2\geq 5.5\log\frac{1}{\varepsilon}+O(1).    
\end{align*}
\end{enumerate}
\end{lemma}

\begin{proof}[Proof of (i).]
 We may assume that $x$ is large enough and that $\varepsilon>0$ is small enough. 

Suppose first that $\cond(\chi)>x^{\varepsilon^{20}}$. In that case we shall show the stronger bound
\begin{align}\label{eqq24}
\sup_{|t|\leq 2.1x}\Big|\sum_{n\leq P}\Lambda(n)\chi(n)n^{-it}\Big|\ll \frac{P}{(\log P)^{100}}.    
\end{align}

By Perron's formula, we have
\begin{align}\label{eqq27}
\sum_{n\leq P}\Lambda(n)\chi(n)n^{-it}=-\frac{1}{2\pi i}\int_{1+1/\log x-iT}^{1+1/\log x+iT}\frac{L'}{L}(s+it,\chi)\frac{P^s}{s}\, ds+O\Big(\frac{P}{(\log P)^{100}}\Big)  \end{align}
where $T\coloneqq (\log x)^{1000}$. Recall that by the definition of $\mathcal{Q}_{x,\varepsilon,\varepsilon^{-6}}$ the function $L(s,\chi)$ has the zero-free region $\Re(s) \geq 1-\sigma_0\coloneqq 1-\varepsilon^{-6}(\log \log x)/(\log x)$, $|\Im(s)|\leq 3x$. Shift the line of integration in~\eqref{eqq27} to $\Re(s)= 1-\sigma_0/2$. By~\cite[Lemma 11.1]{mv}, we have $$\Big|\frac{L'}{L}(s,\chi)\Big|\ll (\log x)^2$$ whenever $9/10\leq \Re(s)\leq 2$, $|\Im(s)|\leq 10x$, and the distance from $s$ to the nearest zero of $L(\cdot, \chi)$ is $\geq \frac{1}{\log x}$.  Hence, we obtain for~\eqref{eqq27} the bound
\begin{align*}
\ll P^{1-\tfrac{\sigma_0}{2}}(\log x)^2(\log \log x)\ll P^{1-\varepsilon^{-1/2}\tfrac{\log \log P}{4\log P}} \ll \frac{P}{(\log P)^{100}}.
\end{align*}

Suppose then that $\cond(\chi)\leq x^{\varepsilon^{20}}$. Then, since $\cond(\chi)\leq P^{\varepsilon^{14}}$, by Lemma~\ref{pls_hyb} and Remark~\ref{rem2} we have
\begin{align*}
\Big|\sum_{n\leq P}\Lambda(n)\chi(n)n^{-it}\Big|\ll \varepsilon^{14}\log^3\Big(\frac{1}{\varepsilon}\Big)P+\frac{P}{(\log P)^{0.3}}+\frac{P}{1+|t|},  
\end{align*}
where the last term can be deleted if $\chi$ is complex. Since $(\log P)^{-0.3}\leq \varepsilon^{10}$ and by assumption $|t|\geq \varepsilon^{-10}$ if $\chi$ is real and $\cond(\chi)\leq x^{\varepsilon^{20}}$, we obtain the desired bound. 

\noindent\emph{Proof of (ii).} Suppose first that $\chi$ is principal. Let $y=V_{10x}$ in the notation of~\eqref{eq_Vt} (in particular, $\log y \geq \sqrt{\log x}$). Note that $\sum_{p\mid q,  p\geq y}1/p\ll 1$. Then by Mertens's theorem and Lemma~\ref{le3.2}, we have
\begin{align*}
 \mathbb{D}_q(\chi(n)n^{it},1;x)^2&=\mathbb{D}_q(n^{it},1;x)^2\geq \log \frac{\log x}{\log y}+\Re\left(\sum_{y<p\leq x}\frac{1}{p^{1+it}}\right)+O(1)\\
 &=\log \frac{\log x}{\log y}+\log\left|L_y\left(1+\frac{1}{\log x}+it,1\right)\right|+O(1).
\end{align*}
Lemma~\ref{le4.2} tells us that $|L_y(1+\frac{1}{\log x}+it,1)|\asymp 1$ for $\tfrac{1}{\log y} \leq |t|\leq 10 x$, so we obtain
\begin{align}\label{eq106}
\inf_{\tfrac{1}{\sqrt{\log x}} \leq |t|\leq 10x}\mathbb{D}_q(n^{it},1;x)^2\geq \left(\frac{1}{3}-o(1)\right)\log \log x,    
\end{align}
which suffices. 

If $\chi$ is non-principal (so that $v_{\chi} = 0$) and $\cond(\chi)\leq x^{\varepsilon^{20}}$, then Lemma~\ref{le_L1chi_hyb} gives the desired bound
\begin{align*}
\mathbb{D}_q(\chi(n)n^{it},1;x)^2\geq  \frac{11}{40}\log\left(\frac{\log x}{\log x^{\varepsilon^{20}}}\right)+O(1)=5.5\log \frac{1}{\varepsilon}+O(1).   
\end{align*}
We are then left with the case where $\cond(\chi)> x^{\varepsilon^{20}}$.

Since $\sum_{p\mid q, p>x^{\varepsilon^{5.5}}}1/p\ll 1$, we have
\begin{align*}
\mathbb{D}_q(\chi(n) n^{it},1;x)^2 \geq\Re\Bigg( \sum_{x^{\varepsilon^{5.5}} \leq  p \leq x} \frac{1-\chi(p)p^{it}}{p}\Bigg)+ O(1).
\end{align*}
By Mertens's theorem, this is
\begin{align*}
\geq 5.5\log\frac{1}{\e} +S+ O(1),    
\end{align*}
where by partial summation
\begin{align*}
S:=\sum_{x^{\varepsilon^5}\leq p\leq x}\frac{\chi(p)p^{-it}}{p}= \frac{1}{x\log x}\sum_{x^{\varepsilon^5}\leq p\leq x}\chi(p)(\log p)p^{-it}+\int_{x^{\varepsilon^5}}^x \frac{(\log y+1)\sum_{x^{\varepsilon^5}\leq p\leq y}\chi(p)(\log p)p^{-it}}{y^2\log^2 y}dy.   
\end{align*}
From part (i) we now see that $S=O(1)$, completing the proof.
\end{proof}

\begin{prop}[Sup norm bound for twisted sums of a multiplicative function]\label{prop_sup_hyb} Let $x\geq 10$ and $(\log x)^{-1/50}\leq \varepsilon\leq 1$. Let $f\colon \mathbb{N}\to \mathbb{U}$ be a multiplicative function. Let $(\chi_1,t_{\chi_1})$ be a point minimizing the map $(\chi,t)\mapsto \mathbb{D}_q(f,\chi(n)n^{it};x)$ among $\chi\pmod q$ and $|t|\leq x$. 

Let $10 \leq P,Q \leq x$ with $\tfrac{\log Q}{\log P} \leq \varepsilon^{-1/6}$, and let $g: \mb{N} \ra [0,1]$ be any multiplicative function with the property that $g(p) = 1$ for all $p \notin [P,Q]$. 

Then, with the notation of Lemma~\ref{le_Qset}, for all $q\in \mathcal{Q}_{x,\varepsilon,\varepsilon^{-6}}$ we have
\begin{align}\label{eqq38}
\sup_{\substack{\chi \pmod q\\\chi\neq \chi_1}}\sup_{|t| \leq x/2}\sup_{y\in [x^{0.1},x]}\Big|\frac{1}{y}\sum_{n\leq y} f(n)g(n)\bar{\chi}(n)n^{-it}\Big|\ll \varepsilon \frac{\phi(q)}{q}.    
\end{align}
In addition, for all $1 \leq Z \leq x$ and $1\leq q\leq x$ we have 
\begin{align}\label{eqq39}
\sup_{\substack{|t| \leq x \\ |t-t_{\chi_1}| \geq Z}}\sup_{y\in [x^{0.1},x]} \Big|\frac{1}{y}\sum_{n\leq y} f(n)g(n)\bar{\chi_1}(n)n^{-it}\Big| \ll \frac{\phi(q)}{q}\Big((\log x)^{-1/15} +\frac{1}{\sqrt{Z}}\Big).
\end{align}
\end{prop}

\begin{remark}\label{rem_sup} 
    For the proofs of Theorem~\ref{MRAPTHM} and~\ref{BVTHM}, we need a version of this proposition where the supremum over $t$ is over the smaller range $[-\frac{1}{2}\log x,\frac{1}{2}\log x]$, and $(\chi_1,t_{\chi_1})$ is taken be a minimizing point of $(\chi,t)\mapsto \mathbb{D}_q(f,\chi(n)n^{it};x)$ with $|t|\leq \log x$. The same proof applies to this case, and we can obtain a similar variant of Corollary~\ref{cor_sup} as well.
    \end{remark}
    
    \begin{remark}
    The same arguments as in Subsection~\ref{subsec: vino} show that we cannot prove~\eqref{eqq38} for \emph{all} $q\leq x$ without settling Vinogradov's conjecture at the same time. However, in the smaller range of $q\leq x^{\varepsilon^{20}}$  there are no exceptional moduli in Proposition~\ref{prop_sup_hyb}; cf.~\cite[Lemma 3.1]{bgs} for a related result in this range.
     \end{remark}

\begin{proof}
We begin with the first claim. We may assume in what follows that $x$ is larger than any fixed constant and that $\e$ is smaller than any fixed constant. 

Suppose for the sake of contradiction that there is a character $\chi \neq \chi_1\pmod q$ and a real number $t \in [-x/2,x/2]$ for which 
$$\Big|\sum_{n\leq y} f(n)g(n)\overline{\chi}(n)n^{-it}\Big| \geq \e \frac{\phi(q)}{q}y$$
for some $y\in [x^{0.1},x]$. Owing to $\e > (\log x)^{-1/50}$ and the fact that $\sum_{y\leq p\leq x}\frac{1}{p}\ll 1$, Lemma~\ref{le_hal_hyb} implies that there is some $v \in [-\frac{1}{2}\log x,\frac{1}{2}\log x]$ for which
\begin{align*}
\mb{D}_q(fg,\chi(n) n^{i(t+v)};x)^2 \leq 1.001 \log\frac{1}{\e} + O(1).
\end{align*}
Since $fg(p) = f(p)$ for all $p \notin [P,Q]$ we have that
\begin{align}
&\max_{\alpha \in \mb{R}}|\mb{D}_q(fg,\chi(n)n^{i\alpha};x)^2 - \mb{D}_q(f,\chi(n)n^{i\alpha};x)^2| \nonumber\\
&= \max_{\alpha \in \mb{R}}\left|\sum_{\substack{P \leq p \leq Q \\ p \nmid q}} \frac{(1-g(p))\text{Re}(f(p)\bar{\chi}(p)p^{-i\alpha})}{p}\right| \leq \sum_{P \leq p \leq Q} \frac{1-g(p)}{p} \leq \log\left(\frac{\log Q}{\log P}\right) + O(1), \label{eq:withgDist}
\end{align}
and thus as $\log Q \leq \varepsilon^{-1/6} \log P$ we obtain
$$
\mb{D}_q(f,\chi(n)n^{i(t+v)};x)^2 \leq 1.17 \log(1/\varepsilon) + O(1).
$$
According to the definition of $\chi_1$, we also have $\mb{D}_q(f,\chi_1(n) n^{it_{\chi_1}};x)^2 \leq 1.17\log(1/\e) + O(1)$ with $t_{\chi_1}\in [-x,x]$. As such, the pretentious triangle inequality implies that
\begin{align*}
\mathbb{D}_q(\chi_1(n) n^{it_{\chi_1}},\chi(n) n^{i(t+v)};x)^2&\leq \left(\mb{D}_q(f,\chi(n) n^{i(t+v)};x) + \mb{D}_q(f,\chi_1(n) n^{it_{\chi_1}};x) \right)^2\\
&\leq 5.48\log\frac{1}{\varepsilon}+O(1).     
\end{align*}
But since $\chi_1\overline{\chi}$ is nonprincipal, this contradicts Lemma~\ref{le_primesums}(ii).

We proceed to the second claim of the proposition. We may assume that $Z\geq 2$. Suppose $|t-t_{\chi_1}| \geq Z$ and $|t| \leq x$. Let $|u| \leq Z/2$, so that $Z/2 \leq |t+u-t_{\chi_1}| \leq 2x$. By the definition of $t_{\chi_1}$ and the triangle inequality, 
\begin{align*}
2\mb{D}_q(f,\chi_1(n) n^{i(t+u)};x) &\geq \mb{D}_q(f,\chi_1(n) n^{i(t+u)};x) + \mb{D}_q(f,\chi_1(n) n^{it_{\chi_1}};x)\\
&\geq \mathbb{D}_q(1,n^{i(t+u-t_{\chi_1})};x).
\end{align*}

From~\eqref{eq106}, we see that
\begin{align}\label{eqq162}
\inf_{1 \leq |\alpha|\leq 2x }\mb{D}_q(n^{i\alpha},1;x)^2\geq \Big(\frac{1}{3}-o(1) \Big)\log \log x+O(1).
\end{align}
Therefore, we conclude that 
\begin{align*}
\mb{D}_q(f,\chi_1(n) n^{i(t+u)};x)^2\geq \Big(\frac{1}{12}-o(1)\Big)\log \log x.
\end{align*}
Using \eqref{eq:withgDist} and $\varepsilon \geq (\log x)^{-1/50}$, we deduce that for $x$ sufficiently large,
$$
\mb{D}_q(fg,\chi_1(n) n^{i(t+u)};x)^2\geq \Big(\frac{1}{12}-o(1)\Big)\log \log x - \log\left(\frac{\log Q}{\log P}\right) \geq \frac{1}{15} \log\log x.
$$
Applying the Hal\'asz-type bound of Lemma~\ref{le_hal_hyb} with $T = Z/2$, this yields
$$\Big|\sum_{n\leq y} f(n)g(n)\bar{\chi_1}(n)n^{-it}\Big| \ll \frac{\phi(q)}{q}\Big((\log x)^{-1/15}+\frac{1}{\sqrt{Z}}\Big)y,$$
for every $|t| \leq x$ satisfying $|t-t_{\chi_1}| \geq Z$, as claimed.
\end{proof}

We will also require a variant of Proposition~\ref{prop_sup_hyb} for sums weighted by the factor $1/(1+\omega_{[P,Q]}(m))$ that arises in the statement of Lemma~\ref{LEM12withT}.

\begin{cor}\label{cor_sup} Let $x\geq R\geq 10$, $\varepsilon\in ((\log x)^{-1/50},1)$ and $(\log x)^{-0.1}<\alpha<\beta<1$, with $\beta/\alpha \leq \varepsilon^{-1/6}$.  Set $P=x^{\alpha}$, $Q=x^{\beta}$ and for $f\colon \mathbb{N}\to \mathbb{U}$ multiplicative consider the twisted character sum
\begin{align*}
R(\chi,s)\coloneqq \sum_{R\leq m\leq 2R}\frac{f(m)\overline{\chi}(m)m^{-s}}{1+\omega_{[P,Q]}(m)}.    
\end{align*}
Let $(\chi_1,t_{\chi_1})$ be a point minimizing the map $(\chi,t)\mapsto \mathbb{D}_q(f,\chi(n)n^{it};x)$ for $\chi\pmod q$ and $|t|\leq x$. Then, with the notation of Lemma~\ref{le_Qset}, for $q\in \mc{Q}_{x,\varepsilon,\varepsilon^{-6}}$ we have
\begin{align}\label{eqq153}
\sup_{\substack{\chi\pmod q\\\chi\neq \chi_1}}\sup_{|t|\leq x/2}\sup_{R\in [x^{1/2},x]}\frac{1}{R}|R(\chi,it)|\ll \varepsilon \frac{\phi(q)}{q}.    
\end{align}
Furthermore, for $1\leq Z\leq (\log x)^{1/10}$ we have
\begin{align}\label{eqq153b}
\sup_{\substack{|t| \leq x \\ |t-t_{\chi_1}| \geq Z}} \sup_{R\in [x^{1/2},x]}\frac{1}{R}|R(\chi_1,it)|\ll \frac{1}{\sqrt{Z}}\frac{\phi(q)}{q}.     
\end{align}

\end{cor}

\begin{remark}
The case $q=1$ of the corollary is a variant of~\cite[Lemma 3]{mr-annals}. 
\end{remark}

\begin{proof} We may assume that $x$ is larger than any fixed constant, and $\varepsilon$ is smaller than any fixed constant, since otherwise both results are immediate from the trivial bound $|R(\chi,it)| \leq |\{R \leq m \leq 2R : (m,q) = 1\}$.
 
The proof of both \eqref{eqq153} and \eqref{eqq153b} rely on the following simple observation: for any $m \geq 1$ we have
$$
\frac{1}{1+\omega_{[P,Q]}(m)} = \int_0^1 r^{\omega_{[P,Q]}(m)} dr.
$$
For each $r \in (0,1]$ the map $g_r(m) := r^{\omega_{[P,Q]}(m)}$ is a multiplicative function satisfying $0 \leq g_r \leq 1$, and such that $g_r(p) = 1$ for every prime $p \notin [P,Q]$. 

Given $\chi \pmod{q}$ and $t \in \mb{R}$, we therefore have
$$
\frac{1}{R}|R(\chi,it)| = \int_0^1 \left(\frac{1}{R} \sum_{R \leq m \leq 2R} f(m)g_r(m)\bar{\chi}(m)m^{-it}\right) dr.
$$
With regards to \eqref{eqq153} there is thus some $r_0 \in (0,1]$ such that if $g = g_{r_0}$ then
\begin{align*}
&\sup_{\substack{\chi\pmod q\\\chi\neq \chi_1}}\sup_{|t|\leq x/2}\sup_{R\in [x^{1/2},x]}\frac{1}{R}|R(\chi,it)| \\
&\leq \int_0^1 \left(\sup_{\substack{\chi\pmod q\\\chi\neq \chi_1}}\sup_{|t|\leq x/2}\sup_{R\in [x^{1/2},x]}\left|\frac{1}{R} \sum_{R \leq m \leq 2R} f(m)g_r(m)\bar{\chi}(m)m^{-it}\right|\right) dr \\
&\leq \sup_{\substack{\chi\pmod q\\\chi\neq \chi_1}}\sup_{|t|\leq x/2}\sup_{R\in [x^{1/2},x]}\left|\frac{1}{R} \sum_{R \leq m \leq 2R} f(m)g(m)\bar{\chi}(m)m^{-it}\right|,
\end{align*}
and similarly towards \eqref{eqq153b} there is some $r_1 \in (0,1]$ such that if $g = g_{r_1}$ then
$$
\sup_{\substack{|t| \leq x \\ |t-t_{\chi_1}| \geq Z}}\sup_{R\in [x^{1/2},x]}\frac{1}{R}|R(\chi_1,it)| \leq \sup_{\substack{|t| \leq x \\ |t-t_{\chi_1}| \geq Z}}\sup_{R\in [x^{1/2},x]}\left|\frac{1}{R} \sum_{R \leq m \leq 2R} f(m)g(m)\bar{\chi}_1(m)m^{-it} \right|.
$$
The corollary is proved upon applying \eqref{eqq38} and \eqref{eqq39} of Proposition \ref{prop_sup_hyb}, respectively, to each of the last two estimates.
\end{proof}

\begin{prop}[Sharp large values bound for weighted sums of twisted characters]\label{prop_largevalues_hyb}  Let $x\geq 10$,   $(\log x)^{-1/30}<\varepsilon\leq \delta\leq  1/2$, and $\eta=\max\{(\log x)^{-1/30},\varepsilon^4\}$. Let $(a_p)_p$ be $1$-bounded complex numbers, let $\mc{T} \subset [-x,x]$ be a well-spaced set, and let $\mc{S} \coloneqq  \{(\chi,t) : \chi \pmod{q},\,\,t \in \mc{T}\}$. Define 
\begin{align*}
N_{q,\mc{S}}\coloneqq \sup_{x^{\eta}\leq P\leq x}\Big|\Big\{(\chi,t) \in \mc{S} : \Big|\frac{\log P}{\delta P}\sum_{P\leq p\leq (1+\delta)P}a_p\bar{\chi}(p)p^{-it}\Big|\geq \varepsilon\Big\}\Big|.    
\end{align*}
Then, with the notation of Lemma~\ref{le_Qset}, for $q\in \mathcal{Q}_{x,\varepsilon,\varepsilon^{-6}}$ we have $N_{q,\mc{S}}\ll \varepsilon^{-2}\delta^{-1}$. The implied constant is absolute.
\end{prop}

\begin{remark}
For the proof of Theorem~\ref{MRAPTHM} will only require the special, simpler case $\mc{T}=\{0\}$.
\end{remark}

\begin{proof}
We may assume without loss of generality that $\varepsilon>0$ is smaller than any fixed constant. Let $P \in [x^{\eta},x]$ yield the set of largest cardinality that is counted by $N_{q,\mc{S}}$, and let $\mc{B}_{q,\mc{S}}$ denote the set of pairs $(\chi,t)$ yielding the large values counted by $N_{q,\mathcal{S}}$ at scale $P$. We have, for some unimodular $c_{\chi,t}$,
\begin{align}\label{eqq147}\begin{split}
\frac{\e\delta P}{\log P}N_{q,\mc{S}} &\leq \sum_{(\chi, t) \in \mc{B}_{q,\mc{S}}} \Big|\sum_{P \leq p \leq (1+\delta)P} a_p\bar{\chi}(p)p^{-it}\Big| = \sum_{(\chi,t)\in \mc{B}_{q,\mc{S}}} c_{\chi,t}\sum_{P \leq p \leq (1+\delta)P} a_p\bar{\chi}(p)p^{-it} \\
&= \sum_{P \leq p \leq (1+\delta)P} a_p\sum_{(\chi,t)\in \mc{B}_{q,\mc{S}}} c_{\chi,t} \bar{\chi}(p)p^{-it}.
\end{split}
\end{align}
Applying the Cauchy--Schwarz and Brun--Titchmarsh inequalities, this is 
\begin{align}\label{eqq141}
&\ll \Big(\frac{\delta P}{\log P}\Big)^{1/2} \Big(\sum_{P \leq p \leq (1+\delta)P} \Big|\sum_{(\chi,t)\in \mc{B}_{q,\mc{S}}} c_{\chi,t}\bar{\chi}(p)p^{-it}\Big|^2\Big)^{1/2} \nonumber\\
&\ll \Big(\frac{\delta P}{\log P}\Big)^{1/2} \Big(\sum_{P \leq n \leq (1+\delta)P} \frac{\Lambda(n)}{\log P} \Big|\sum_{(\chi,t)\in \mc{B}_{q,\mc{S}}} c_{\chi,t}\bar{\chi}(n)n^{-it}\Big|^2\Big)^{1/2}. \end{align}

Let $h$ be a smooth function supported on $[1/2,2]$ with $h(u)=1$ for $u\in [1,3/2]$, and $0\leq h(u)\leq 1$ for all $u$. We insert the weight $h(n/P)$ into the $n$ sum in~\eqref{eqq141} and expand out the square, obtaining the upper bound
\begin{align*}
&\ll \Big(\frac{\delta P}{(\log P)^2}\Big)^{1/2} \Big(\sum_{(\chi_1,t_1)\in \mc{B}_{q,\mc{S}}}\sum_{(\chi_2,t_2)\in \mc{B}_{q,\mc{S}}} \Big|\sum_{n} \Lambda(n)\chi_1\bar{\chi_2}(n)n^{-i(t_1-t_2)}h\Big(\frac{n}{P}\Big)\Big|\Big)^{1/2}\\
&= \Big(\frac{\delta P}{(\log P)^2}\Big)^{1/2}(S_1+S_2)^{1/2},
\end{align*}
where we let $S_1$ be the sum over the pairs with $\cond(\chi_1\overline{\chi_2})\leq x^{\varepsilon^{20}}$ and $S_2$ be the sum over the pairs with $\cond(\chi_1\overline{\chi_2})> x^{\varepsilon^{20}}$.

We first treat $S_1$. If $\cond(\chi_1\overline{\chi_2})\leq x^{\varepsilon^{20}}$, then $\cond(\chi_1\overline{\chi_2})\leq P^{(\eta^{-1/20}\varepsilon)^{20}}\leq P^{\varepsilon^{10}}$, so by Lemma~\ref{pls_hyb} (and the fact that $x\leq P^{\eta^{-1}}\leq P^{(\log P)^{1/25}}$) for some non-principal real character $\xi_1\pmod q$ we have
\begin{align}\label{eqq142}
\Big|\sum_{n} \Lambda(n)\chi_1\bar{\chi_2}(n)n^{-i(t_1-t_2)}h\Big(\frac{n}{P}\Big)\Big|\ll \varepsilon^{10}\log^3\Big(\frac{1}{\varepsilon}\Big)P+\frac{P}{(\log P)^{0.3}}+ \frac{P1_{\chi_1\overline{\chi_2}\in \{\chi_0, \xi_1\}}}{|t_1-t_2|^2+1},
\end{align}
with $\chi_0\pmod q$ the principal character. 

Since $|\mathcal{B}_{q,\mc{S}}|=N_{q,\mc{S}}$ and $(\log P)^{-0.3}\leq \varepsilon^{5}$, summing~\eqref{eqq142} over $(\chi_1,t_1),(\chi_2,t_2) \in \mc{B}_{q,\mc{S}}$ shows that the contribution of the characters with small conductor obeys the bound
\begin{align}\label{eqq103}\begin{split}
S_1&\ll \varepsilon^{5}N_{q,\mc{S}}^2 P+\sum_{(\chi_1,t_1) \in \mc{B}_{q,\mc{S}}} \sum_{\chi_2\pmod q}\sum_{t_2\in \mathcal{T}} \frac{P}{|t_1-t_2|^2+1}1_{\chi_1\overline{\chi_2}\in \{\chi_0,\xi_1\}}\\
&\ll \varepsilon^{5}N_{q,\mc{S}}^2P+N_{q,\mc{S}}P\sum_{k\in \mathbb{Z}}\frac{1}{k^2+1}\\
&\ll (\varepsilon^{5}N_{q,\mc{S}}^2+N_{q,\mc{S}})P,
\end{split}
\end{align}
where in the second-to-last step we used the fact that $\mc{T}$ is well-spaced.

We then consider the contribution of $S_2$. By Lemma~\ref{le_primesums}(i) and partial summation (and the fact that $P\geq x^{\eta}\geq x^{\varepsilon^4}$), for $q\in \mathcal{Q}_{x,\varepsilon,\varepsilon^{-6}}$, we have
\begin{align}\label{eqq101}
\sup_{\substack{\chi\pmod q\\\cond(\chi)>x^{\varepsilon^{20}}}}\sup_{|t| \leq x} \Big|\sum_{n} \Lambda(n)\chi(n)n^{-it}h\Big(\frac{n}{P}\Big)\Big| \ll \varepsilon^{10}P.
\end{align}
On the other hand, for $\chi$ principal we use Lemma~\ref{pls_hyb} with $q=1$ and $\varepsilon=1/\log P$ to estimate 
\begin{align}\label{eqq102}
 \Big|\sum_{n} \Lambda(n)\chi(n)n^{-it}h\Big(\frac{n}{P}\Big)\Big| \leq \Big|\sum_{n} \Lambda(n)n^{-it}h\Big(\frac{n}{P}\Big)\Big|+O((\log P)^2) \ll  \frac{P}{(\log P)^{0.3}}+ \frac{P}{|t|^2+1}
\end{align}
for $|t|\leq x$.

Applying~\eqref{eqq101} or~\eqref{eqq102}  with $\chi = \chi_1\bar{\chi}_2$ and $t = t_1-t_2$ for each  $(\chi_1,t_1),(\chi_2,t_2) \in \mc{B}_{q,\mc{S}}$ counted by $S_2$, and handling the contributions from these as in~\eqref{eqq103}, we find 
\begin{align*}
S_2\ll (\varepsilon^{10} N_{q,\mc{S}}^2+N_{q,\mc{S}})P.    
\end{align*}
Combining the bounds on $S_1$ and $S_2$ with~\eqref{eqq147} and~\eqref{eqq141}, we see that
\begin{align*}
\frac{\varepsilon \delta P}{\log P}N_{q,\mc{S}}\ll \frac{\delta^{1/2}P}{\log P}(\varepsilon^{5} N_{q,\mc{S}}+N_{q,\mc{S}}^{1/2}),     
\end{align*}
and since $\varepsilon\leq \delta$ and $\varepsilon>0$ is small enough, we deduce from this that $N_{q,\mc{S}}\ll \varepsilon^{-2}\delta^{-1}$, which was to be shown.
\end{proof}

\begin{remark}\label{rmk_prime} If in Proposition~\ref{prop_sup_hyb} or~\ref{prop_largevalues_hyb} we restrict to a set $\mc{Q}^{\prime}\subset [1,x]$ of pairwise coprime moduli $q$, then by Lemma~\ref{le_Qset} the sizes of the corresponding sets of exceptional $q\leq x$ are $\ll (\log x)^{10\varepsilon^{-6}}$. Moreover, under GRH there are no exceptional moduli.
\end{remark}

\section{Variance in progressions and short intervals}

\label{sec: MR}

\subsection{Typical number of prime factors}
Before proceeding to the proofs of our main theorems, we elaborate on some observations about \emph{typical} moduli (in the sense of Definition~\ref{def-typical}) that were made in the introduction.

Let $\omega_{[P,Q]}(n)\coloneqq  |\{p\mid n:\,\, p\in [P,Q]\}|$ denote the number of prime factors of $n$ belonging to the interval $[P,Q]$. Given $y\geq 1$, define
\begin{align}\label{eq_Delta}
\Delta(q,y)\coloneqq  \max_{z\geq y}\frac{\omega_{[z,2z]}(q)}{z/\log z},
\end{align}
which gives the maximal relative density of prime divisors of $q$ on a dyadic subinterval of $[y,\infty)$. Clearly, if $q$ is $y$-typical in the sense of Definition~\ref{def-typical}, then $\Delta(q,y)\leq 1/50+o(1)$ by the prime number theorem. Note also that $0\leq \Delta(q,y)\ll 1$ always.

\begin{lemma}[Density of atypical integers]\label{le_delta}
Let $Q\geq 1$, $y\geq 1$. Then the number of $q\leq Q$ that are not $y$-typical is $\ll Q\exp(-10^{-4}y)$. Moreover, if $y\geq 1000(\log Q)$ and $Q$ is large enough, then all $q\leq Q$ are $y$-typical. 
\end{lemma}

\begin{proof}
We may assume that $y$ is large enough. Note that by dyadic summation, if
\begin{align*}
\Delta(q,y)\leq \frac{1}{500},
\end{align*}
then $q$ is $y$-typical. The second claim of the lemma follows directly from this and the estimate $\omega(q)\leq (1+o(1))(\log q)/(\log \log q)$.

We are left with the first claim. Observe that for any fixed $c>0$,
by the union bound we have
\begin{align*}
|\{q\leq Q:\,\, \Delta(q,y)>c\}|&\leq \sum_{2^j\geq y}\Big|\Big\{q\leq Q:\,\, \omega_{[2^{j-1},2^{j}]}(n)\geq \frac{c\cdot 2^j}{5\log(2^j)}\Big\}\Big|\\
&\leq \sum_{2^j\geq y}\sum_{\substack{\mathcal{P}\subset \mathbb{P}\cap [2^{j-1},2^j]\\|\mathcal{P}|\geq c\cdot 2^j/(5\log(2^j))}}\Big|\Big\{q\leq Q:\,\prod_{p\in \mathcal{P}}p\mid q\Big\}\Big|\\
&\ll \sum_{2^j\geq y}2^{2^{j}/(\log(2^j))}\cdot Qe^{-c\cdot 2^{j-1}/5}\\
&\ll Qe^{-cy/20}.
\end{align*}

Applying this with $c=1/500$, we obtain the claim.
\end{proof}

From~Lemma~\ref{le_delta}, we deduce the claims made before Theorem~\ref{MRAPTHM-smooth} that all $q\leq x$ are $(x/Q)^{\varepsilon^2}$-typical if $Q=o(x/(\log x)^{1/\varepsilon^2})$ and otherwise the number of $q\leq x$ that are not  $(x/Q)^{\varepsilon^2}$-typical is $\ll \exp(-10^{-4}(x/Q)^{\varepsilon^2})$.

\subsection{Parseval-type bounds} \label{subsec:DeduceMainThms}

We reduce the proofs of Corollary~\ref{MRAPTHM-prime} and Theorems~\ref{MRAPTHM} and~\ref{MRAPHybrid} to $L^2$ bounds for (twisted) character sums. 

\begin{prop}\label{prop_Parseval1} Let $1\leq Q\leq x/10$ and $(\log(x/Q))^{-1/200}\leq \varepsilon\leq 1$, and let $f\colon \mathbb{N}\to \mathbb{U}$ be multiplicative. Let $\chi_1$ be a character $\hspace{-0.1cm}\pmod q$ minimizing the distance $\inf_{|t|\leq \log x}\mathbb{D}_q(f,\chi(n)n^{it};x)$. Then, with the notation of Lemma~\ref{le_Qset}, for $q\in \mathcal{Q}_{x,\varepsilon^6,\varepsilon^{-80}}\cap [1,Q]$ we have
\begin{align}\label{form22.1}
 \sum_{\substack{\chi\pmod q\\\chi\neq \chi_1}} \Big|\sum_{n \leq x} f(n)\overline{\chi}(n)\Big|^2 \ll \varepsilon^{1-3\Delta(q,(x/Q)^{\e})}\Big(\frac{\phi(q)}{q}x\Big)^2.
\end{align}
Moreover, assuming GRH,~\eqref{form22.1} holds for all $q \in [1,Q]$.
\end{prop}

Regarding Corollary~\ref{MRAPTHM-prime}, we in fact prove the following generalisation.

\begin{thm}\label{MRAPTHM-coprime} Let $\mathcal{Q}'\subset [1,Q]$ be any set of pairwise coprime numbers. Corollary~\ref{MRAPTHM-prime} continues to hold if the moduli $p$, rather than being prime, are taken to be $(x/Q)^{\varepsilon^2}$-typical elements of $\mathcal{Q}'$, and if the right-hand side of~\eqref{form23} is replaced with $\varepsilon\phi(p)(x/p)^2$.
\end{thm}

\begin{proof}[Deduction of Theorems~\ref{MRAPTHM-coprime},~\ref{MRAPTHM} from Proposition~\ref{prop_Parseval1}] We apply Proposition~\ref{prop_Parseval1} with $\varepsilon^{1.1}$ in place of $\varepsilon$. We have $\Delta(q,(x/Q)^{\varepsilon^{1.1}})\leq 1/50+o(1)$ by the assumption that $q$ is $(x/Q)^{\varepsilon^2}$-typical. Therefore, $(\varepsilon^{1.1})^{1-3\Delta(q,(x/Q)^{\varepsilon^{1.1}})}\ll \varepsilon^{1.1\cdot (1-3/50+o(1))}\ll \varepsilon$.  

By orthogonality of Dirichlet characters, we have the Parseval-type identity
\begin{align*}
 \frac{1}{\varphi(q)}\sum_{\substack{\chi\pmod q\\\chi\not \in \Xi}}\Big|\sum_{n \leq x} f(n)\overline{\chi}(n)\Big|^2=\asum_{a (q)}\Big|\sum_{\substack{n\leq x\\n\equiv a(q)}}f(n)-\sum_{\chi\in \Xi}\frac{\chi(a)}{\phi(q)}\sum_{n\leq x}f(n)\overline{\chi}(n)\Big|^2  
\end{align*}
for any $\Xi\subset \{\chi\pmod q\}$. Each of the claimed results follows from this, as the corresponding bounds for exceptional moduli in each theorem may be deduced from Lemma~\ref{le_Qset} (see also Remark~\ref{rmk_prime}): for any $1 \leq Q \leq x$ we have $|[1,Q]\setminus \mathcal{Q}_{x,(\varepsilon^{1.1})^6,(\varepsilon^{1.1})^{-80}}|\ll Qx^{-\varepsilon^{200}}$, while if $\mc{Q}' \subset [1,x]$ is a set of pairwise coprime numbers then $|\mathcal{Q}'\setminus \mathcal{Q}_{x,(\varepsilon^{1.1})^6,(\varepsilon^{1.1})^{-80}}|\ll (\log x)^{\varepsilon^{-100}}$.

Finally, assuming GRH, each of the claims holds without exception. 
\end{proof}

Similarly, we will deduce Theorem~\ref{MRAPHybrid}  from the following proposition.

\begin{prop}\label{prop_Parseval2} Let $1\leq Q\leq h/10$ and $(\log(h/Q))^{-1/200}\leq \varepsilon\leq 1$, and let $f\colon \mathbb{N}\to \mathbb{U}$ be multiplicative. Let $\chi_1$ be a character $\hspace{-0.1cm}\pmod q$ minimizing the distance $\inf_{|t| \leq X} \mb{D}_q(f,\chi(n) n^{it};X)$, and let $t_{\chi_1} \in [-X,X]$ be a point that minimizes $\mb{D}_q(f,\chi_1(n) n^{it};X)$. Let $Z_{\chi_1} =\varepsilon^{-10}$ and $Z_{\chi}=0$ for $\chi\neq \chi_1$. Then, with the notation of Lemma~\ref{le_Qset}, for all $q\in \mathcal{Q}_{X,\varepsilon^6,\varepsilon^{-80}}\cap [1,Q]$ and for $T=(X/h) (h/Q)^{0.01\varepsilon}$, we have
\begin{align*}
\sum_{\chi \pmod{q}} \int_{\substack{|t-t_{\chi}| \geq Z_{\chi}\\ |t| \leq T}} \Big|\sum_{n\leq 3X}f(n)\overline{\chi}(n)n^{-it}\Big|^2\, dt\ll \varepsilon^{1-3\Delta(q,(h/Q)^{\varepsilon})}\Big(\frac{\phi(q)}{q}X\Big)^2.
\end{align*}
Moreover, assuming either the GRH or that $Q \leq X^{\e^{150}}$, the exceptional set of $q$ vanishes.
\end{prop}

\begin{proof}[Deduction of Theorem~\ref{MRAPHybrid} from Proposition~\ref{prop_Parseval2}] We use Proposition~\ref{prop_Parseval2} with $\varepsilon^{1.1}$ in place of $\varepsilon$. 

By Lemma~\ref{lem_MTHybrid}, for $x\in [X,2X]$ the second term inside the square in~\eqref{eqq151} is 
\begin{align*}
\frac{\chi_1(a)}{\phi(q)} \cdot \frac{1}{2\pi} \int_{t_{\chi_1}-\varepsilon^{-11}}^{t_{\chi_1} + \varepsilon^{-11}} \Big(\sum_{n \leq 3X} f(n)\overline{\chi_1}(n)n^{-it}\Big) \frac{(x+h)^{it}-x^{it}}{it}dt+O\Big(\varepsilon^{11/2} \frac{h}{q}\Big).    
\end{align*}
Let us call the main term here $\mathcal{M}(X;x,q,a)$.

By the Cauchy--Schwarz inequality, this implies that~\eqref{eqq151} is
\begin{align*}
\ll \asum_{a\pmod{q}} \int_X^{2X}\Big|\sum_{\substack{x < n \leq x+h\\ n \equiv a \pmod{q}}} f(n) -\mathcal{M}(X;x,q,a)\Big|^2 dx + \varepsilon X \phi(q)\Big(\frac{h}{q}\Big)^2.
\end{align*}

We will now show that the following Parseval-type bound holds: for $1 \leq q\leq h \leq X$, we have
\begin{align}\label{eqq31}\begin{split}
&\asum_{a\pmod{q}} \int_X^{2X}\Big|\sum_{\substack{x < n \leq x+h\\ n \equiv a \pmod{q}}} f(n) -\mathcal{M}(X;x,q,a)\Big|^2 dx \ll \max_{T \geq X/h} \frac{h}{T\phi(q)}\sum_{\chi \pmod{q}}  \int_{\substack{|t-t_{\chi}| \geq Z_{\chi}\\ |t| \leq T}} |F(\chi,it)|^2 dt, 
\end{split}
\end{align}
where $F(\chi,s)\coloneqq  \sum_{n\leq 3X}f(n)\overline{\chi}(n)n^{-s}$ and $Z_{\chi}=\varepsilon^{-11}$ if $\chi=\chi_1$ and $Z_{\chi}=0$ otherwise. 

Once we have this, the case where the maximum in~\eqref{eqq31} is attained with $T\geq X/h\cdot (h/Q)^{0.01\varepsilon^{1.1}}$ can be bounded using Lemma~\ref{L2MVTwithT} as
\begin{align*}
\ll h\left(1+\frac{X}{qT}\right)\frac{\phi(q)}{q}X \leq \left(\frac{1}{h/Q} + \frac{X}{Th}\right) \frac{\phi(q)}{q^2} Xh^2 \ll \e \frac{\phi(q)}{q^2}Xh^2,   
\end{align*}
since $(\log h/Q)^{-1/200}\leq \varepsilon^{1.1}$ certainly implies $(h/Q)^{-0.01\varepsilon^{1.1}}\ll \varepsilon$. This contribution is small enough for Theorem~\ref{MRAPHybrid}. If instead $T\in [X/h,X/h\cdot (h/Q)^{0.01\varepsilon^{1.1}}]$,  we have $\frac{h}{T\phi(q)}\ll \frac{h^2}{\phi(q)X}$, so the bound of Proposition~\ref{prop_Parseval2} (with $\varepsilon^{1.1}$ in place of $\varepsilon$) suffices.  

The proof of~\eqref{eqq31} follows closely that of~\cite[Lemma 14]{mr-annals} (here we choose to work on the $0$-line rather than on the $1$-line for convenience, though). Let us write $\mc{I}_{\chi}\coloneqq   (t_{\chi}-Z_{\chi},t_{\chi} + Z_{\chi}]$, where $Z_{\chi}$ is as above. Recall that $\mc{I}_{\chi}=\emptyset$ if $\chi\neq \chi_1$. We note first of all that\footnote{Here the integral is interpreted as zero if $\mc{I}_{\chi}$ is empty.} 
\begin{align*}
&\sum_{\substack{x < n \leq x+h\\ n \equiv a \pmod{q}}} f(n) - \mc{M}(X;x,q,a) \\
&= \sum_{\chi \pmod{q}} \frac{\chi(a)}{\phi(q)} \Big(\sum_{x < n \leq x+h} f(n)\bar{\chi}(n) - \frac{1}{2\pi } \int_{\mc{I}_{\chi}} F(\chi,it) \frac{(x+h)^{it}-x^{it}}{it} dt\Big),
\end{align*}
so that by orthogonality of Dirichlet characters we find
\begin{align*}
&\asum_{a \pmod{q}}\Big|\sum_{\substack{x < n \leq x+h\\ n \equiv a \pmod{q}}} f(n) - \mc{M}(X;x,q,a)\Big|^2 \\
&= \frac{1}{\phi(q)} \sum_{\chi \pmod{q}} \Big|\sum_{x < n \leq x+h} f(n)\bar{\chi}(n) - \frac{1}{2\pi i} \int_{\mc{I}_{\chi}}F(\chi,it) \frac{(x+h)^{it}-x^{it}}{t}dt\Big|^2.
\end{align*}

Now, by Perron's formula (taking the line of integration $\text{Re}(s) = c \ra 0^+$ since $P_{f\bar{\chi}}$ is finitely supported), whenever $x,x+h$ are not integers, for each $\chi$ we have
$$\sum_{x < n \leq x+h} f(n)\bar{\chi}(n) = \frac{1}{2\pi i} \int_{-\infty}^{\infty} F(\chi,it) \frac{(x+h)^{it}-x^{it}}{t} dt,$$
so that, if $\mathcal{L}$ is the expression on the left-hand side of~\eqref{eqq31}, we have
$$\mc{L} = \frac{1}{\phi(q)} \sum_{\chi \pmod{q}} \int_X^{2X} \Big|\frac{1}{2\pi i}\int_{\mb{R} \bk \mc{I}_{\chi}} F(\chi,it) \frac{(x+h)^{it}-x^{it}}{t} dt\Big|^2 dx.$$

Repeating the trick at the bottom of page 22 of~\cite{mr-annals}, we can find some point $u \in [-3h/X, 3h/X]$ for which
\begin{equation*}
\mc{L} \ll \frac{1}{\phi(q)} \sum_{\chi \pmod{q}} \int_X^{2X} \Big|\int_{\mb{R}\bk \mc{I}_{\chi}} F(\chi,it) x^{it}\frac{(1+u)^{it}-1}{t} dt\Big|^2 dx.
\end{equation*}
The rest of the proof of~\eqref{eqq31} then follows that of~\cite[Lemma 14]{mr-annals} almost verbatim (adding a smooth weight to the $x$ integral, expanding the square and swapping the order of integration).
\end{proof}

\subsection{Proof of hybrid theorem}\label{sub: hybrid}

We may of course assume in what follows that $h/Q$ is larger than any given absolute constant and that $\varepsilon>0$ is smaller than any given positive constant.

We have shown that to prove Theorem~\ref{MRAPHybrid} it is enough to prove Proposition~\ref{prop_Parseval2}, i.e., that
\begin{align} \label{eqq32}
\sum_{\chi \pmod{q}} \int_{\substack{|t-t_{\chi}| \geq Z_{\chi} \\|t| \leq T}}|F(\chi,it)|^2 dt\ll \varepsilon^{1-3\Delta(q,(h/Q)^{\varepsilon})}\Big(\frac{\phi(q)}{q}X\Big)^2
\end{align}
for $T= (X/h)(h/Q)^{0.01\varepsilon}$, where
\begin{align*}
F(\chi,s)\coloneqq  \sum_{n\leq 3X}f(n)\overline{\chi}(n)n^{-s}.   
\end{align*}

As in~\cite{mr-annals}, we restrict the support of $F(\chi,s)$ to integers with typical factorization. Define a ``well-factorable'' set $\mathcal{S}$ as follows. For $1\leq Q\leq h/10$, $\varepsilon\in ((\log  \frac{h}{Q})^{-1/200},1)$ and $2\leq j\leq J-1$, set 
\begin{align*}
P_1 &= Q_1^{\varepsilon},\quad  Q_1 =  h/Q,
\\
P_j &= \exp\Big(j^{4j}(\log Q_1)^{j-1}\log P_1\Big), \quad Q_j = \exp\Big(j^{4j+2} (\log Q_1)^j\Big),\\
P_J&=X^{\varepsilon^2},\quad Q_J=X^{\varepsilon},
\end{align*}
with $J \geq 2$ being chosen minimally subject to the constraint $J^{4J+2}(\log Q_1)^J > (\log X)^{1/2}$. (If $J=2$, only use the definitions of $P_1,Q_1,$ $P_J,Q_J$.)  

Then let 
\begin{align*}
\mathcal{S}\coloneqq  \{n\leq x: \omega_{[P_i,Q_i]}(n)\geq 1\quad \forall 1 \leq i\leq J\}.    
\end{align*}
One sees that for $2\leq j\leq J$ the inequalities
\begin{align}\label{eqq33}
\frac{\log \log Q_j}{\log P_{j-1}-1}\leq \frac{\eta}{4j^2},\quad   \frac{\eta}{j^2}\log P_j\geq 8\log Q_{j-1}+16\log j 
\end{align}
hold for fixed $\eta\in (0,1)$ and large enough $h/Q$ (the $j=2$ case follows from the assumption $\log(h/Q)>\varepsilon^{-100}$, and for the $j=J$ case it is helpful to note that $J\ll \log \log X$ and $P_{J-1}\gg\exp((\log \log X)^{10})$ if $J\geq 3$), and thus the $P_j$, $Q_j$ satisfy all the same requirements as in~\cite{mr-annals}. A simple sieve upper bound shows that 
\begin{align*}
|[1,3X]\setminus\mathcal{S}|\ll \sum_{j\leq J}X\frac{\log P_j}{\log Q_j}\ll X\frac{\log P_1}{\log Q_1}\sum_{j\geq 1}\frac{1}{j^{2}}\ll \varepsilon X.  
\end{align*}

We next define
\begin{align*}
&H_1=H_J=H\coloneqq  \lfloor \varepsilon^{-1}\rfloor,\quad H_j\coloneqq  j^2P_1^{0.1} \quad \textnormal{ for }\quad 2\leq j\leq J-1;\nonumber\\    
&\mc{I}_{j}\coloneqq  [\lfloor H_j\log P_j\rfloor, H_j\log Q_j];\nonumber \\
&Q_{v,H_j}(\chi,s)\coloneqq  \sum_{e^{v/H_j}\leq p<e^{(v+1)/H_j}}f(p)\overline{\chi}(p)p^{-s} \quad \text{ for $v \in \mc{I}_j$};\\
&R_{v,H_j}(\chi,s)\coloneqq  \sum_{m\leq 3Xe^{-v/H_j}}\frac{f(m)\overline{\chi}(m)m^{-s}}{1+\omega_{[P_j,Q_j]}(m)} \text{ for $v \in \mc{I}_j$.}\nonumber
\end{align*}

We split the set
\begin{align*}
\mc{E}\coloneqq  \{(\chi,t)\in \{\chi\pmod q\}\times [-X,X]:\,\, |t-t_{\chi}|\geq Z_{\chi},\,\, |t|\leq T\}    
\end{align*}
as $\mc{E}=\bigcup_{j\leq J-1}\mc{X}_j\cup\mc{U}$ with
\begin{align*}
\mc{X}_1&=\{(\chi,t)\in \mc{E}:\,\, |Q_{v,H_1}(\chi,it)|\leq e^{(1-\alpha_1)v/H_1}\quad \forall v\in \mc{I}_1\},\\
\mc{X}_j&=\{(\chi,t)\in \mc{E}:\,\, |Q_{v,H_j}(\chi,it)|\leq e^{(1-\alpha_j)v/H_j}\quad \forall v\in \mc{I}_j\}\setminus \bigcup_{i\leq j-1}\mc{X}_i,\quad\quad  2\leq j\leq J-1,\\
\mc{U}&=\mc{E}\setminus \bigcup_{i\leq J-1}\mc{X}_i,
\end{align*}
where we take
\begin{align*}
\alpha_j=\frac{1}{4}-\eta\Big(1+\frac{1}{2j}\Big),\quad \eta=0.01.    
\end{align*}

We may of course write, for some (possibly empty) sets $\mc{T}_{j,\chi}\subset [-T,T]$,
\begin{align*}
\mc{X}_j=\bigcup_{\chi\pmod q}\{\chi\}\times \mc{T}_{j,\chi}.    
\end{align*}

By Lemma~\ref{LEM12withT}, for each $1 \leq j \leq J-1$ we have
\begin{align*}
\sum_{\chi \pmod q} \int_{\mc{T}_{j,\chi}} \Big|\sum_{n \leq 3X}f(n)\overline{\chi}(n)n^{-it}\Big|^2 dt
&\ll H_j \log\frac{Q_j}{P_j} \sum_{v \in \mc{I}_j} \sum_{\chi \pmod q} \int_{\mc{T}_{j,\chi}} |Q_{v,H_j}(\chi,it)|^2|R_{v,H_j}(\chi,it)|^2 dt\\
&\quad + \Big(\frac{\phi(q)}{q}X\Big)^2\Big(\frac{1}{H_j} + \frac{1}{P_j}+\prod_{\substack{P_j\leq p\leq Q_j\\p\nmid q}}\Big(1-\frac{1}{p}\Big)\Big),
\end{align*}
the integrals here being interpreted as zero if $\mc{T}_{j,\chi} = \emptyset$. By our choices of $P_j$ and $H_j$, the error terms involving $1/H_j$ or $1/P_j$ are $\ll \varepsilon (\phi(q)/q \cdot X)^2$ when summed over $j\leq J-1$, since $\log(h/Q)\geq \varepsilon^{-100}$ by assumption. After summing over $j\leq J-1$, the error terms involving $\prod_{\substack{P_j\leq p\leq Q_j\\p\nmid q}}(1-\frac{1}{p})$ contribute
\begin{align*}
&\ll  \Big(\frac{\phi(q)}{q}X\Big)^2 \sum_{j\leq J-1}\frac{\log P_j}{\log Q_j}\prod_{\substack{p\mid q\\P_j\leq p\leq Q_j}}\Big(1+\frac{1}{p}\Big)\ll \Big(\frac{\phi(q)}{q}X\Big)^2\sum_{j\leq J-1}\frac{\varepsilon}{j^2}\prod_{\substack{p\mid q\\P_j\leq p\leq Q_j}}\Big(1+\frac{1}{p}\Big).
\end{align*}

In terms of the $\Delta(\cdot)$ function defined in~\eqref{eq_Delta}, for $j\leq J-1$ we have
\begin{align*}
 \prod_{\substack{p\mid q\\P_j\leq p\leq Q_j}}\Big(1+\frac{1}{p}\Big)&\ll \exp\Big(\sum_{\substack{p\mid q\\P_j\leq p\leq Q_j}}\frac{1}{p}\Big)\ll \exp\Big(\sum_{2^k\in [P_j/2,Q_j]}\frac{\Delta(q,P_1)}{\log (2^k)}\Big)\\
 &\ll \Big(\frac{\log Q_j}{\log P_j}\Big)^{\frac{1}{\log 2}\Delta(q,P_1)}\ll(j^2\varepsilon^{-1})^{\frac{1}{\log 2}\Delta(q,P_1)}.
\end{align*}
Hence, on multiplying by $\varepsilon/j^2$ and summing over $j\leq J-1$, for $\Delta(q,P_1)\leq 1/3$, we get a contribution of $\ll \varepsilon^{1-3\Delta(q,P_1)}$, which is the desired savings (the $j$ sum converges since $2(\frac{1}{3\log 2}-1)<-1$). For $\Delta(q,P_1)>1/3$, in turn, we simply use the triangle inequality to note that the trivial bound $\ll (\phi(q)/q\cdot X)^2$ for~\eqref{eqq32} coming from Lemma~\ref{L2MVTwithT} (after forgetting the condition $|t-t_{\chi}| \geq Z_{\chi}$) is good enough.

Making use of the assumption defining $\mc{X}_j$, we have
\begin{align*}
&H_j \log \frac{Q_j}{P_j} \sum_{v \in \mc{I}_j} \sum_{\chi \pmod q} \int_{\mc{T}_{j,\chi}} |Q_{v,H_j}(\chi,it)|^2|R_{v,H_j}(\chi,it)|^2 dt \\
&\ll H_j \log \frac{Q_j}{P_j} \sum_{v \in \mc{I}_j} e^{(2-2\alpha_j) v/H_j} \sum_{\chi \pmod q} \int_{\mc{T}_{j,\chi}} |R_{v,H_j}(\chi,it)|^2 dt =: E_j
\end{align*}
It thus remains to bound $E_j$ for $1 \leq j \leq J-1$, as well as the contributions from the pairs $(\chi,t)\in \mc{U}$.\\

\textbf{Case of} $\mc{X}_1$. For the pairs in $\mc{X}_1$, we crudely extend the $t$-integral to $[-T,T]$ and apply Lemma~\ref{L2MVTwithT} to arrive at
\begin{align*}
E_1 &\ll  H_1\log \frac{Q_1}{P_1} \sum_{v \in \mathcal{I}_1} e^{(2-2\alpha_1) v/H_1}\Big(\frac{\phi(q)}{q}Xe^{-v/H_1} + \phi(q)T\Big)\cdot \frac{\phi(q)}{q}Xe^{-v/H_1} \\
&\ll \Big(\frac{\phi(q)}{q}X\Big)^2H_1\log Q_1 \cdot \sum_{v \in \mc{I}_1} e^{-2\alpha_1 v/H_1}\\
&\ll \Big(\frac{\phi(q)}{q}X\Big)^2H_1\log Q_1\cdot \frac{1}{P_1^{2\alpha_1}}\cdot \frac{1}{1-e^{-2\alpha_1/H_1}}\\
&\ll \Big(\frac{\phi(q)}{q}X\Big)^2P_1^{-0.1}H_1^2,
\end{align*}
where on the second line we used $Xe^{-v/H}\geq X/Q_1\geq q T$ by the assumptions $T=(X/h)\cdot P_1^{0.01}$ and $Q_1 \leq h/(QP_1^{0.01})$. We see that the contribution of $E_1$ is small enough, since $H_1\ll \varepsilon^{-1}$ and $P_1^{-0.1} = (h/Q)^{-0.1\varepsilon} \ll \varepsilon^{10}$.\\

\textbf{Case of} $\mc{X}_j$. Let $2\leq j\leq J-1$.  We partition 
$$\mc{X}_j = \bigcup_{r \in \mc{I}_{j-1}} \mc{X}_{j,r}$$ 
where $\mc{X}_{j,r}$ is the set of $(\chi,t) \in \mc{X}_j$ such that $r$ is the minimal index in $\mc{I}_{j-1}$ with $|Q_{r,H_{j-1}}(\chi,it)| > e^{(1-\alpha_{j-1})r/H_{j-1}}$. Letting $r_0 \in \mc{I}_{j-1}$ and $v_0 \in \mc{I}_j$ denote the choices of $r$ and $v$, respectively, with maximal contribution, we obtain
\begin{align*}
E_j &\ll H_j (\log Q_j) |\mc{I}_j||\mc{I}_{j-1}|\\
&\quad \cdot \sum_{\chi \pmod q} e^{(2-2\alpha_j) v_0/H_j}\int_{-T}^{T} \Big(|Q_{r_0,H_{j-1}}(\chi,it)|/e^{(1-\alpha_{j-1})r_0/H_{j-1}}\Big)^{2\ell_{j,r_0}}|R_{v_0,H_{j}}(\chi,it)|^2 dt,
\end{align*}
where $\ell_{j,r_0}\coloneqq  \lceil\frac{v_0/H_j}{r_0/H_{j-1}}\rceil > 1$.

Using $|\mc{I}_{j-1}| \leq |\mc{I}_j|\ll H_j \log Q_j$, this becomes
\begin{align*}
E_j &\ll \Big(H_j \log Q_j\Big)^3 e^{(2-2\alpha_j)v_0/H_j -(2- 2\alpha_{j-1})\ell_{j,r_0}r_0/H_{j-1}}\sum_{\chi \pmod q} \int_{-T}^{T}|Q_{r_0,H_{j-1}}(\chi,it)^{\ell_{j,r_0}} R_{v_0,H_j}(\chi,it)|^2 dt.
\end{align*}
We apply Lemma~\ref{LEM13withT} to obtain
$$\sum_{\chi \pmod q}\int_{-T}^{T}|Q_{r_0,H_{j-1}}(\chi,it)^{\ell_{j,r_0}} R_{v_0,H_j}(\chi,it)|^2 dt \ll \Big(\frac{\phi(q)}{q}Xe^{r_0/H_{j-1}}2^{\ell_{j,r_0}}\Big)^2((\ell_{j,r_0}+1)!)^2.$$
We have $\ell_{j,r_0}\geq \frac{v_0/H_j}{r_0/H_{j-1}}$, whence using $2^{\ell}(\ell+1)!\ll \ell^{\ell}$ we get
\begin{align}\label{eq14}
E_j \ll (H_{j}\log Q_j)^3 \Big(\frac{\phi(q)}{q}Xe^{r_0/H_{j-1}}\Big)^2e^{2(\alpha_{j-1}-\alpha_{j})v_0/H_j + 2\ell_{j,r_0}\log \ell_{j,r_0}}.
\end{align}
Since $\ell_{j,r_0}\leq \frac{v_0/H_j}{r_0/H_{j-1}}+1$ and $r_0/H_{j-1}\geq \log P_{j-1}-1$, $v_0/H_j\leq \log Q_j$, we have
$$\ell_{j,r_0}\log \ell_{j,r_0}\leq \frac{v_0}{H_j} \frac{\log\log Q_{j}}{\log P_{j-1} - 1} + \log\log Q_j + 1.$$
Thus,~\eqref{eq14} is
\begin{align*}
\ll \Big(\frac{\phi(q)}{q}Xe^{r_0/H_{j-1}}\Big)^2H_j^3(\log Q_j)^5\exp\Big(\Big(2\frac{\log \log Q_{j}}{\log P_{j-1}-1}+2(\alpha_{j-1}-\alpha_{j})\Big)v_0/H_j\Big).    
\end{align*}

By~\eqref{eqq33} and the choice of the $\alpha_j$, we have the inequalities
\begin{align*}
\frac{\log \log Q_{j-1}}{\log P_j-1}\leq \frac{\eta}{4j^2},\quad \alpha_{j-1}-\alpha_{j}\leq -\frac{\eta}{2j^2},\quad \log Q_{j}\leq Q_{j-1}^{1/24},
\end{align*}
so we get
\begin{align*}
E_j&\ll \Big(\frac{\phi(q)}{q}X\Big)^2H_j^3 (\log Q_j)^5 Q_{j-1}^{2}P_j^{-\frac{\eta}{2j^2}}\\
&\ll \Big(\frac{\phi(q)}{q}X\Big)^2j^6P_1^{0.3}Q_{j-1}^{2+5/24}P_j^{-\frac{\eta}{2j^2}}\\
&\ll \Big(\frac{\phi(q)}{q}X\Big)^2j^6 Q_{j-1}^{3}P_j^{-\frac{\eta}{2j^2}}.
\end{align*}

Again by~\eqref{eqq33}, we have the inequality
\begin{align*}
\frac{\eta}{j^2}\log P_j\geq 8\log Q_{j-1}+16\log j,    
\end{align*}
so
\begin{align*}
E_j\ll \Big(\frac{\phi(q)}{q}X\Big)^2\frac{1}{j^2Q_{j-1}}\ll \Big(\frac{\phi(q)}{q}X\Big)^2\frac{1}{j^2P_1}.    
\end{align*}
Summing over $j$ gives
$$\sum_{2 \leq j \leq J-1} E_j \ll \Big(\frac{\phi(q)}{q}X\Big)^2P_1^{-1},$$
and this is acceptable. It remains to deal with $\mc{U}$.\\

\textbf{Case of} $\mc{U}$. Let us write 
\begin{align*}
\mc{U}=\bigcup_{\chi\pmod q}\{\chi\}\times \mc{T}_{\chi}.    
\end{align*}

By Lemma~\ref{LEM12withT} and the definitions of $P_J$, $Q_J$ and $H\coloneqq  H_J$, we have
\begin{align}\label{eqq37}\begin{split}
 \sum_{\chi\pmod q}\int_{\mc{T}_{\chi}} |F(\chi,it)|^2 dt &\ll H\log \frac{Q_J}{P_J}\sum_{v \in \mc{I}_J} \sum_{\chi \pmod{q}}\int_{\mc{T}_{\chi}} |Q_{v,H}(\chi,it)|^2|R_{v,H}(\chi,it)|^2 dt\\
 &+\Big(\frac{\phi(q)}{q} X\Big)^2\Big(\frac{1}{H}+\frac{1}{P_J}+\varepsilon\prod_{\substack{P_J\leq p\leq Q_J\\p\mid q}}\Big(1+\frac{1}{p}\Big)\Big).  
 \end{split}
 \end{align}
 Since $H=\lfloor \varepsilon^{-1}\rfloor$ and $\prod_{\substack{P_J\leq p\leq Q_J\\p\mid q}}(1+\frac{1}{p})\ll \varepsilon^{-3\Delta(q,P_1)}$ (similarly to the $\mc{X}_j$ case), the terms on the second line of~\eqref{eqq37} contribute $\ll \varepsilon^{1-3\Delta(q,P_1)}(\frac{\phi(q)}{q}X)^2$. Thus we have 
 \begin{align}\label{eqq37b}
 \sum_{\chi\pmod{q}} \int_{\mc{T}_{\chi}} |F(\chi,it)|^2 dt &\ll (H\log \frac{Q_J}{P_J})^2  \sum_{\chi \pmod{q}} \int_{\mc{T}_{\chi}} |Q_{v_0,H}(\chi,it)|^2 |R_{v_0,H}(\chi,it)|^2dt\nonumber\\
 &+ \varepsilon^{1-3\Delta(q,P_1)}\Big(\frac{\phi(q)}{q}X\Big)^2\nonumber\\
 &\ll H^2\varepsilon^{2}(\log X)^{2} \sum_{\chi \pmod{q}} \int_{\mc{T}_{\chi}} |Q_{v_0,H}(\chi,it)|^2 |R_{v_0,H}(\chi,it)|^2 dt\\
 &+ \varepsilon^{1-3\Delta(q,P_1)}\Big(\frac{\phi(q)}{q}X\Big)^2\nonumber
\end{align}
for some $v_0\in [H\varepsilon^2 \log X-1,H\varepsilon\log X]$, with $H=\lfloor \varepsilon^{-1}\rfloor.$ 

We discretize the integral, so that the term on the right of~\eqref{eqq37b} is bounded by
\begin{align} \label{eq:MTUCase}
\ll H^2\varepsilon^2(\log X)^2\sum_{\chi\pmod q}\sum_{t\in \mc{T}_{\chi}'}|Q_{v_0,H}(\chi,it)|^2|R_{v_0,H}(\chi,it)|^2    
\end{align}
for some (possibly empty) well-spaced set $\mathcal{T}_{\chi}'\subset \mathcal{T}_{\chi}\subset [-T,T]$. 

Let us define the discrete version of $\mc{U}$ as
\begin{align*}
\mc{U}'=\bigcup_{\chi\pmod q}\{\chi\}\times \mc{T}_{\chi}'.    
\end{align*}
We consider separately the subsets
\begin{align*}
\mc{U}_S:&=\Big\{(\chi,t)\in \mc{U}': |Q_{v_0,H}(\chi,it)|\leq \varepsilon^2 \frac{e^{v_0/H}}{v_0}\Big\},\\
\mc{U}_L:&=\Big\{(\chi,t)\in \mc{U}': |Q_{v_0,H}(\chi,it)|> \varepsilon^2 \frac{e^{v_0/H}}{v_0}\Big\};
\end{align*}
note that by the Brun--Titchmarsh inequality the trivial upper bound is $|Q_{v_0,H}(\chi,it)|\ll e^{v_0/H}/v_0$.

We start with the $\mathcal{U}_S$ case. Observe that if $(\chi,t) \in \mc{U}_S$ then $|Q_{v',H_{J-1}}(\chi,it)| > e^{(1-\alpha_{J-1})v'/H_{J-1}}$ for some $v' \in \mc{I}_{J-1}$. Applying our large values estimate, Lemma~\ref{PRIMMVTwithT}, together with the fact that $qT\ll X^{1+o(1)}$, this leads to
\begin{align*}
|\mc{U}_S|\ll (qT)^{2\alpha_{J-1}}(Q_{J-1}^{2\alpha_{J-1}}+(\log X)^{200\varepsilon^{-2}})\ll X^{0.49},    
\end{align*}
since $\alpha_{J-1}\leq 1/4-\eta$ and $\eta=0.01$. Hence, by the Hal\'asz--Montgomery inequality for twisted character sums (Lemma~\ref{HALINT}), we have
\begin{align*}
&\sum_{(\chi,t) \in \mc{U}_S} |Q_{v_0,H}(\chi,it)|^2 |R_{v_0,H}(\chi,it)|^2\ll \varepsilon^{4}\frac{e^{2v_0/H}}{v_0^2}\sum_{(\chi,t) \in \mathcal{U}_S}|R_{v_0,H}(\chi,it)|^2\\
&\ll \varepsilon^{4}\frac{e^{2v_0/H}}{v_0^2}\Big(\frac{\phi(q)}{q}Xe^{-v_0/H}+(qT)^{1/2}(\log(2qT))|\mathcal{U}_S|\Big)\frac{\phi(q)}{q}Xe^{-v_0/H}\\
&\ll \varepsilon^{4}\frac{e^{2v_0/H}}{v_0^2}\Big(\frac{\phi(q)}{q}Xe^{-v_0/H}\Big)^2\ll H^{-2}\cdot (\log X)^{-2}\Big(\frac{\phi(q)}{q}X\Big)^2,
\end{align*}
since $\frac{\phi(q)}{q}Xe^{-v_0/H}\gg X^{0.999}$ and $v_0\gg H\varepsilon^2\log X$. This bound is admissible after multiplying by $H^2\varepsilon^2(\log X)^2$.

Now we turn to the $\mc{U}_L$ case. We recall that our modulus satisfies $q\in  \mathcal{Q}_{x,\varepsilon^6,\varepsilon^{-80}}$, 
and note that $\varepsilon^7>(\log X)^{-1/13}$.

By Proposition~\ref{prop_largevalues_hyb} (with $\e \to \varepsilon^2$ and $\delta = e^{1/H}-1 \asymp 1/H$ and $\varepsilon>0$ small enough), for $q\in \mathcal{Q}_{x,\varepsilon^6,\varepsilon^{-80}}$, we have $|\mc{U}_L|\ll \varepsilon^{-4}H\ll \varepsilon^{-5}$. In addition, by Corollary~\ref{cor_sup} (with $\varepsilon$ replaced by $\varepsilon' := \varepsilon^6$, $\alpha$ and $\beta$ replaced, respectively, by $\varepsilon^2$ and $\varepsilon$ so that $\beta/\alpha = \varepsilon^{-1} = (\varepsilon')^{-1/6}$), for $q\in \mathcal{Q}_{x,\varepsilon^6,\varepsilon^{-80}}\subseteq \mathcal{Q}_{x,\varepsilon',(\varepsilon')^{-6}}$ we have the pointwise bound 
\begin{align*}
\sup_{\substack{\chi\pmod q\\\chi\neq \chi_1}}\sup_{|t|\leq X}|R_{v_0}(\chi,it)|/(Xe^{-v_0/H})\ll \varepsilon'\frac{\phi(q)}{q} = \varepsilon^6 \frac{\phi(q)}{q}.
\end{align*}
Hence we can bound the contribution of the pairs $(\chi,t)$ with $\chi\neq \chi_1$ by
\begin{align*}
\sum_{\substack{(\chi,t)\in \mc{U}_L\\\chi\neq \chi_1}}|Q_{v_0,H}(\chi,it)|^2|R_{v_0,H}(\chi,it)|^2&\ll |\mc{U}_L|\frac{e^{2v_0/H}}{v_0^2}\sup_{\substack{\chi\pmod q\\\chi\neq \chi_1}}\sup_{|t|\leq X}|R_{v_0,H}(\chi,it)|^2\\
&\ll \frac{\varepsilon^{-5+12}}{\varepsilon^4H^2(\log X)^2}\Big(\frac{\phi(q)}{q}X\Big)^2, 
\end{align*}
and this multiplied by the factor $H^2\varepsilon^2(\log X)^2$ yields a more than sufficient bound for~\eqref{eq:MTUCase}.

The contribution of $\chi=\chi_1$, in turn, is bounded using Corollary~\ref{cor_sup} in the form that
\begin{align*}
\sup_{Z\leq |t-t_{\chi_1}|\leq x}\frac{|R_{v_0,H}(\chi_1,it)|}{Xe^{-v_0/H}}\ll \frac{1}{\sqrt{Z}}\frac{\phi(q)}{q}
\end{align*}
for $Z = \varepsilon^{-10} \leq (\log X)^{1/20}$ and for $q$ as before. This yields
\begin{align*}
\sum_{(\chi_1,t)\in \mc{U}_L}|Q_{v_0,H}(\chi_1,it)|^2|R_{v_0,H}(\chi_1,it)|^2&\ll |\mc{U}_L|\frac{e^{2v_0/H}}{v_0^2}\sup_{Z \leq |t-t_{\chi_1}|\leq X}|R_{v_0,H}(\chi_1,it)|^2\\
&\ll \frac{\varepsilon^{-5+10}}{\varepsilon^4 H^2(\log X)^2}\Big(\frac{\phi(q)}{q}X\Big)^2. 
\end{align*}
This multiplied by $H^2\varepsilon^2(\log X)^2$ produces a good enough bound, finishing the proof of Proposition~\ref{prop_Parseval2}, and hence that of Theorem~\ref{MRAPHybrid}.

\begin{proof}[Proof of Corollary~\ref{HybridReal}] 
Our task is to simplify the main term in Theorem~\ref{MRAPHybrid} in the case of a real-valued multiplicative $f\colon \mathbb{N}\to [-1,1]$. We work with the same set of moduli $q\in  \mathcal{Q}_{x,\varepsilon^6,\varepsilon^{-80}}$ as in Proposition~\ref{prop_Parseval2}. By the triangle inequality and Theorem~\ref{MRAPHybrid}, it suffices to show that
\begin{align}\label{eq_corHyb}
\int_X^{2X} \Big|\frac{1}{3X}\sum_{n \leq 3X} f(n)\bar{\chi_1}(n) - \Big(\frac{1}{h} \int_x^{x+h} v^{it_{\chi}} dv\Big) \frac{1}{3X}\sum_{n \leq 3X} f(n)\bar{\chi_1}(n)n^{-it_{\chi_1}}\Big|^2\, dx \ll \varepsilon X\Big(\frac{\phi(q)}{q}\Big)^2.
\end{align}

Set $Y := 3X$. By the triangle inequality for pretentious distance, we have
\begin{align}\label{eqq163}
2\mathbb{D}_q(f,\chi_1(n)n^{it_{\chi_1}};Y)\geq \mathbb{D}_q(f^2,\chi_1^2(n)n^{2it_{\chi_1}};Y)\geq \mathbb{D}_q(1,\chi_1^2(n)n^{2it_{\chi_1}};Y)-\mathbb{D}_q(f^2,1;Y).    
\end{align}
We then split into several cases.

\textbf{Case 1.} Suppose first that $\mathbb{D}_q(f^2,1;Y)^2> 2\log(1/\varepsilon)$. Applying the Cauchy--Schwarz inequality, for  
$t\in \{0,t_{\chi_1}\}$ we obtain
\begin{align*}
\Big|\sum_{n\leq Y}f(n)\overline{\chi_1}(n)n^{-it}\Big|\ll \Big(\frac{\varphi(q)}{q}Y\Big)^{1/2}\Big(\sum_{\substack{n\leq Y \\ (n,q) = 1}}f(n)^2\Big)^{1/2}.    
\end{align*}
By an elementary upper bound~\cite[(1.85)]{iw-kow} for mean values of multiplicative functions taking values in $[0,1]$, we have
\begin{align*}
\sum_{\substack{n\leq Y \\ (n,q) = 1}}f(n)^2\ll \frac{\phi(q)}{q}Y\prod_{\substack{p\leq Y \\ p \nmid q}}\left(1+\frac{f(p)^2-1}{p}\right)\ll \frac{\phi(q)}{q} Y\exp(-\mathbb{D}_q(f^2,1;Y)^2)\ll \varepsilon^{2} \frac{\phi(q)}{q}Y.    
\end{align*}
Hence~\eqref{eq_corHyb} holds. 

\textbf{Case 2.} Suppose then that $\mathbb{D}_q(f^2,1;Y)^2\leq 2\log(1/\varepsilon)$ and that either $\chi_1$ is complex or $|t_{\chi_1}|\geq 2\log x$. Then by Lemma~\ref{le_primesums}(ii) (with $\varepsilon^6$ in place of $\varepsilon$) from~\eqref{eqq163} (with $t$ in place of $t_{\chi_1}$) we obtain
\begin{align}\label{eqq161a}
\inf_{|t-t_{\chi_1}|\leq \log x}\mathbb{D}_q(f,\chi_1(n)n^{it};Y)^2\geq \Big(\frac{\sqrt{5.5\cdot 6}-\sqrt{2}}{2}\Big)^2\log \frac{1}{\varepsilon}>4\log \frac{1}{\varepsilon}.    
\end{align}
By Lemma~\ref{le_hal_hyb} and the minimal property of $t_{\chi_1}$,  
we then obtain
\begin{align}\label{eqq161}
\max_{t\in \{0,t_{\chi_1}\}}\Big|\sum_{n\leq Y}f(n)\overline{\chi_1}(n)n^{-it}\Big|\ll \varepsilon^3\frac{\varphi(q)}{q}Y,   
\end{align}
which implies~\eqref{eq_corHyb}.

\textbf{Case 3.} Finally, suppose that $\mathbb{D}_q(f^2,1;Y)\leq 2\log(1/\varepsilon)$ and that $\chi_1$ is real and $|t_{\chi_1}|\leq 2\log x$. We may assume that 
\begin{align}\label{eqq165}
\mathbb{D}_q(f,\chi_1(n)n^{it_{\chi_1}};Y)^2\leq 2\log \frac{1}{\varepsilon},    
\end{align}
since otherwise~\eqref{eqq161a} holds by the minimal property of $t_{\chi_1}$ and then we can conclude as before. But then by \eqref{eqq163}, 
\begin{align*}
 \mathbb{D}_q(1,\chi_1^2(n)n^{2it_{\chi_1}};Y)-\mathbb{D}_q(f^2,1;Y)\leq \left(8\log \frac{1}{\varepsilon}\right)^{1/2},   
\end{align*}
so by the assumption on $\mathbb{D}_q(f^2,1;Y)$ and the fact that $\chi_1$ is real we deduce
\begin{align*}
 \mathbb{D}_q(1,n^{2it_{\chi_1}};Y)^2\leq\left(2^{1/2}+8^{1/2}\right)^2\log \frac{1}{\varepsilon} = 18\log \frac{1}{\varepsilon}.    
\end{align*}
By~\eqref{eq106} (and the fact that $\varepsilon\geq (\log X)^{-1/200}$) this is a contradiction unless $|t_{\chi_1}|\leq \varepsilon$.

For $|t_{\chi_1}|\leq \varepsilon$, in turn, we have
\begin{align*}
\frac{1}{h}\Big(\int_x^{x+h} v^{it_{\chi_1}} dv\Big) \frac{1}{Y}\sum_{n \leq Y} f(n)\overline{\chi_1}(n) n^{-it_{\chi_1}}= \frac{x^{it_{\chi_1}}}{Y} \sum_{n \leq Y} f(n)\overline{\chi_1}(n)n^{-it_{\chi_1}} + O\Big(\varepsilon\frac{\phi(q)}{q}\Big).
\end{align*}
by estimating the integrand trivially. Moreover, the expression on the right is by~\eqref{eqq165} and Lemma~\ref{lem_twist} equal to
\begin{align}\label{eqq164}
\frac{1}{Y} \sum_{n \leq Y} f(n)\overline{\chi_1}(n)+ O\Big(\varepsilon\frac{\phi(q)}{q}\Big)    
\end{align}

This proves~\eqref{eq_corHyb} also in this case, and now all the cases have been dealt with.
\end{proof}

Corollary~\ref{cor_mobius} follows quickly from Corollary~\ref{HybridReal}. 

\begin{proof}[Proof of Corollary~\ref{cor_mobius}] We apply Corollary~\ref{HybridReal} with $f=\mu$. Note that the set $\mathcal{Q}_{X,\varepsilon}$ in Corollary~\ref{HybridReal} contains all positive integers $\leq X^{\varepsilon^{200}}$. Now, let $c_0$ be a small enough absolute constant, and let $m\leq X^{\varepsilon^{200}}$ be a modulus for which $L(s,\chi^{*})$ for some real character $\chi^{*}\pmod m$ has a real zero $>1-c_0/\log(X^{\varepsilon^{200}})$ (if it exists). By the Landau--Page theorem and Siegel's theorem, all such $m$ are multiples of a single number $q_0\gg (\log X)^{A}$ (if no such $m$ exists, set $q_0=1$).

It then suffices to show that for $q$ not divisible by $q_0$ we have
\begin{align}\label{equ201}
\Big|\sum_{n\leq 3X}\mu(n)\chi(n)\Big|\ll \varepsilon^{1/2} \frac{\varphi(q)}{q}X    
\end{align}
for all characters $\chi\pmod q$. We may assume that $\chi$ is non-principal, as otherwise the claim follows from the prime number theorem (with e.g. de la Vall\'{e}e-Poussin error term). By Lemma~\ref{le_hal_hyb}, we have~\eqref{equ201} provided that
\begin{align}\label{equ203}
\inf_{|t|\leq \log x}\mathbb{D}_q(\mu,\chi(n)n^{it};3X)^2\geq \log(1/\varepsilon),   
\end{align}
say. By Mertens's theorem, for $|t|\leq \log x$ we can lower bound
\begin{align}\label{equ202}
\mathbb{D}_q(\mu,\chi(n)n^{it};3X)^2\geq\sum_{X^{\varepsilon^{10}}\leq p\leq X}\frac{1+\textnormal{Re}(\chi(p)p^{it})}{p}= 10\log \frac{1}{\varepsilon}+\textnormal{Re}\left(\sum_{X^{\varepsilon^{10}}\leq p\leq X}\frac{\chi(p)p^{it}}{p}\right)+O(1).
\end{align}

By Lemma~\ref{pls_hyb} and Remark~\ref{rem2}, for any $q^{1/\varepsilon^{100}}\leq y\leq X$ and $q$ not divisible by $q_0$ we have
\begin{align}\label{equ204}
\Big|\sum_{n\leq y}\Lambda(n)\chi(n)n^{it}\Big|\ll \varepsilon^{100} \Big(\log^3\frac{1}{\varepsilon}\Big)y+\frac{y}{(\log y)^{0.3}}.
\end{align}
Using~\eqref{equ204} and partial summation, we conclude that the left-hand side of~\eqref{equ202} is $\geq 9\log(1/\varepsilon)$, say. We thus obtain~\eqref{equ203}, and hence~\eqref{equ201}.
\end{proof}

\subsection{The case of arithmetic progressions}
\label{sub: APs}
In this subsection we prove Theorem~\ref{MRAPTHM} and Theorem~\ref{MRAPTHM-coprime}. As shown in Section~\ref{subsec:DeduceMainThms}, it suffices to prove Proposition~\ref{prop_Parseval1}. 

\begin{proof}[Proof of Proposition~\ref{prop_Parseval1}] We may plainly assume that $x$ is larger than any fixed constant and that $\varepsilon>0$ is smaller than any fixed positive constant. 

The proof follows the same lines as that of Proposition~\ref{prop_Parseval2}, and we merely highlight the main differences. For $x \geq 10$, $1 \leq Q \leq x/100$ and $(\log(x/Q))^{-1/200} \leq \varepsilon \leq 1$ we set
\begin{align*}
P_1 &= Q_1^{\varepsilon}, \quad Q_1 =  x/Q\\
P_j &= \exp\Big(j^{4j}(\log Q_1)^{j-1}\log P_1\Big), \quad Q_j = \exp\Big(j^{4j+2}(\log Q_1)^{j}\Big), \quad 2 \leq j \leq J-1 \\
P_J &= x^{\varepsilon^2}, \quad Q_J = x^{\varepsilon},
\end{align*}
where $J \geq 2$ is the smallest integer with $J^{4J+2}(\log Q_1)^J > (\log x)^{1/2}$. (If $J = 2$ then only define $P_1,Q_1,P_J,Q_J$ as above.)

In analogy to the definitions made in the proof of Proposition~\ref{prop_Parseval2}, we also define
\begin{align}\label{eq_Fchi}
&F(\chi)\coloneqq  \sum_{n\leq 3x}f(n)\overline{\chi}(n),\\
&H_1 = H_J = H \coloneqq   \llf \varepsilon^{-1}\rrf, \quad H_j \coloneqq   j^2P^{0.1} \text{ for } 2 \leq j \leq J-1, \nonumber\\
&\mc{I}_{j}\coloneqq  [\lfloor H_j\log P_j\rfloor, H_j\log Q_j] \text{ for } 2 \leq j \leq J-1, \nonumber \\
&Q_{v,H_j}(\chi)\coloneqq  \sum_{e^{v/H_j}\leq p<e^{(v+1)/H_j}}f(p)\overline{\chi}(p),\nonumber \\
&R_{v,H_j}(\chi)\coloneqq  \sum_{m\leq 3xe^{-v/H_j}}\frac{f(m)\overline{\chi}(m)}{1+\omega_{[P_j,Q_j]}(m)} \text{ for } v \in \mc{I}_j, \ 1 \leq j \leq J.\nonumber
\end{align}

Finally, for $q\geq 1$ and $2\leq j\leq J-1$, let us write
\begin{align*}
\mc{X}_1:&=\{\chi \neq \chi_1\pmod q:\,\, |Q_{v,H_1}(\chi)| \leq e^{(1-\alpha_1) v/H_1}\,\, \forall\, v\in \mc{I}_1\},\\
\mc{X}_j:&=\{\chi \neq \chi_1\pmod q:\,\, |Q_{v,H_j}(\chi)| \leq e^{(1-\alpha_j) v/H_j}\,\, \forall\, v\in \mc{I}_j\}\setminus \bigcup_{i\leq j-1}\mc{X}_i,\quad \quad 2\leq j\leq J-1\\
\mc{U}:&=\{\chi \neq \chi_1\pmod q\}\setminus \bigcup_{i\leq J-1}\mc{X}_i,
\end{align*}
where, as before, we put
\begin{align*}
\alpha_j = \frac{1}{4} - \eta\Big(1+\frac{1}{2j}\Big),\quad \eta=0.01
\end{align*}
for each $1 \leq j \leq J-1$. Similarly to the proof of Proposition~\ref{prop_Parseval2}, the proof of Proposition~\ref{prop_Parseval1} (and hence of Theorem~\ref{MRAPTHM}) splits into the cases $\chi\in \mathcal{X}_1,\ldots, \mathcal{X}_{J-1}, \mathcal{U}$, depending on which character sum is small or large. 

The introduction of the typical factorizations corresponding to the set $\mathcal{S}$ is handled, as above, using Lemma~\ref{LEM12withT} (more precisely,~\eqref{eq:noTL2} there), which gives
\begin{align} \label{eq_L2charbd}
\sum_{\chi \in \mathcal{X}_j}|F(\chi)|^2\ll H_j\log\frac{Q_j}{P_j} \sum_{v\in I_j}\sum_{\chi \in \mc{X}_j}|Q_{v,H_j}(\chi)R_{v,H_j}(\chi)|^2 +\Big(\frac{\phi(q)}{q} x\Big)^2\Big(\frac{1}{H_j}+\frac{1}{P_j}+\prod_{\substack{P_j\leq p\leq Q_j\\p\nmid q}}(1-\frac{1}{p})\Big).   
\end{align}
When summed over $1 \leq j \leq J-1$, the error terms are small, analogously to the proof of Theorem~\ref{MRAPHybrid}.

Letting $E_j$ denote the main term on the right of~\eqref{eq_L2charbd}, we apply the same arguments, but with Lemma~\ref{SIMPLEORTHO} in place of Lemma~\ref{L2MVTwithT} for $j = 1$, and for $2 \leq j \leq J-1$ we use the second statement of Lemma~\ref{LEM13withT}, rather than the first. In this way we obtain
\begin{align*}
E_1 &\ll \Big(\frac{\phi(q)}{q}x\Big)^2 H_1^2P_1^{-0.1} \ll \varepsilon \Big(\frac{\phi(q)}{q}x\Big)^2 \\
\sum_{2 \leq j \leq J-1} E_j &\ll \Big(\frac{\phi(q)}{q}x\Big)^2 \sum_{2 \leq j \leq J-1} \frac{1}{j^2Q_{j-1}} \ll \Big(\frac{\phi(q)}{q}x\Big)^2 P_1^{-1},
\end{align*}
which is sufficient.

In the case of $\mc{U}$, we apply Lemma~\ref{LEM12withT} once again with the choices $P_J$ and $Q_J$. As above, we find a $v_0 \in \mc{I}_J$ such that
$$
\sum_{\chi \in \mc{U}} |F(\chi)|^2 \ll (H\log Q_J)^2 \sum_{\chi \in \mc{U}} |Q_{v_0,H}(\chi)|^2|R_{v_0,H}(\chi)|^2 + \varepsilon^{1-3\Delta(q,P_1)} \Big(\frac{\phi(q)}{q}x\Big)^2,
$$
estimating the error term as for the sets $\mc{X}_j$, but invoking the specific choices of $H_J$, $P_J$ and $Q_J$. 

As in the proof of Proposition~\ref{prop_Parseval2}, we split $\mc{U}$ further into the subsets
\begin{align*}
\mc{U}_S &\coloneqq   \left\{\chi \neq \chi_1 : |Q_{v_0,H}(\chi)| \leq \varepsilon^2\frac{e^{v_0/H}}{v_0}\right\} \cap \mc{U} \\
\mc{U}_L &\coloneqq   \left\{\chi \neq \chi_1 : |Q_{v_0,H}(\chi)| > \varepsilon^2\frac{e^{v_0/H}}{v_0}\right\} \cap \mc{U}.
\end{align*}
We combine Lemma~\ref{PRIMMVTwithT} (with $\mc{T} = \{0\}$ this time) with Lemma~\ref{HALINT} (wherein $\mc{E}$ consists of points $(\chi,0)$), and argue as in the proof of Proposition~\ref{prop_Parseval2} to obtain that
$$
\sum_{\chi \in \mc{U}_S} |Q_{v_0,H}(\chi)|^2|R_{v_0,H}(\chi)|^2 \ll (H \log x)^{-2} \Big(\frac{\phi(q)}{q}x\Big)^2,
$$
which, when multiplied by $(H\log Q_J)^2 \ll \varepsilon^2 (H\log x)^2$ yields an acceptable bound. 

We treat the $\mc{U}_L$ case in essentially the same way as in the proof of Proposition~\ref{prop_Parseval2}, and in fact the claim is simpler, as it suffices to combine Proposition~\ref{prop_largevalues_hyb} (with the same parameter choices as in the previous proof) with Corollary~\ref{cor_sup} (taking Remark~\ref{rem_sup} into account).
 
\end{proof}

\section{The case of smooth moduli}

\label{sec: smooth}

In this section, we prove Theorem~\ref{MRAPTHM-smooth} on the variance of multiplicative functions in arithmetic progressions to all smooth moduli. A key additional ingredient compared to the proof of Corollary~\ref{MRAPTHM-prime} is the following estimate for short sums of Dirichlet characters with smooth conductor.

\begin{lemma}\label{ChangGen}
Let $q,N\geq 1$ with $P^+(q)\le N^{0.001}$ and $N\geq q^{C/(\log \log q)}$ where $C>0$ is a large absolute constant. Then, uniformly for any non-principal character $\chi\pmod q$ and any $M\ge 1$,
\begin{align}\label{eqqn9}
\Big|\sum_{M\le n\le M+N}\chi(n)\Big|\ll N\exp\Big(-\frac{1}{4}\sqrt{\log N}\Big).    
\end{align}

\end{lemma}
\begin{proof}
We may assume that $q$, and thus $N$, is larger than any fixed constant, since the claim is immediate otherwise. We note moreover that for $N\ge q,$ the estimate~\eqref{eqqn9} follows directly from the P\'olya--Vinogradov inequality, and thus we can assume that $N<q.$

The result~\eqref{eqqn9} holds for primitive $\chi\pmod q$ (in a wider range than stated above and with $\exp(-\sqrt{\log N})$ in place of $\exp(-\frac{1}{4}\sqrt{\log N})$) by a result of Chang~\cite[Theorem 5]{chang}. Indeed, Chang's estimate holds in the regime $\log N>(\log q)^{1-c}+C'\log(2\frac{\log q}{\log q'})\frac{\log q'}{\log q}\frac{\log q}{\log \log q}$ for some $c,C'>0$ and with $q'=\prod_{p\mid q}p$, so as $u\log \frac{2}{u}\leq 1$ for $u\leq 1$ the range in Chang's result contains the range $N>q^{2C'/\log \log q}$. 

Let now $\chi\pmod q$ be a non-principal character induced by a primitive character $\chi'\pmod{q'}$ with $q'\mid q$, so that $\chi(n) = \chi'(n)1_{(n,q) = 1}$. By M\"obius inversion,
\begin{align*}
\sum_{M\leq n\leq M+N}\chi(n)&=\sum_{d\mid q}\mu(d)\chi'(d)\sum_{M/d\leq m\leq (M+N)/d}\chi'(m).
\end{align*}
Note that in our range $\sqrt{N}\geq q^{0.5C/(\log \log q)}$ and $\sqrt{N}\tau(q)\ll N^{0.9},$ thus taking $C=10C'$ and using Chang's strengthening of~\eqref{eqqn9} for the primitive character $\chi'\pmod{q'},$ we arrive at
\begin{align*}
\sum_{M\leq n\leq M+N}\chi(n)&\ll  \sum_{\substack{d\mid q\\d\leq \sqrt{N}}}\Big|\sum_{M/d\leq m\leq (M+N)/d}\chi'(m) \Big|+\sum_{\substack{d\mid q\\d>\sqrt{N}}}\Big(\frac{N}{d}+1\Big)\\&
\ll  N\exp\Big(-\frac{1}{2}\sqrt{\log N}\Big)\sum_{d\mid q}\frac{1}{d}+\sqrt{N}\tau(q)\\&\ll N\exp\Big(-\Big(\frac{1}{2}+o(1)\Big)\sqrt{\log N}\Big),   
\end{align*}
and the result follows for large enough $N$.
\end{proof}

Lemma~\ref{le_moments} below, which uses Lemma~\ref{ChangGen} as an input, allows us to improve on Proposition~\ref{prop_largevalues_hyb} for smooth moduli. It provides good upper bounds for the frequency of large character sums $\hspace{-0.1cm}\pmod q$ over primes \emph{without any exceptional} smooth $q$.

\begin{lemma}\label{le_moments} Let $q\geq P\geq 1$ be integers with  $P^+(q)\le P^{1/10000}.$ Suppose also that $P > q^{1/(\log \log q)^{0.9}}.$ Then for $1\geq \delta\geq \exp(-(\log P)^{0.49})$ and $V\geq \exp(-(\log P)^{0.49})$ and for any complex numbers $|a_p|\le 1,$ we have 
\begin{align}\label{eqqn7}
\Big|\Big\{\chi\pmod q:\,\, \Big|\sum_{P\leq p\leq (1+\delta)P}a_p\chi(p)\Big|\geq V\frac{\delta P}{\log P}\Big\}\Big|\ll (CV^{-1})^{6\log (qP)/\log P},     
\end{align}
with the implied constant and $C>1$ being absolute.
\end{lemma}

\begin{proof} We begin by noting that, under our assumptions, Lemma~\ref{ChangGen} implies 
\begin{align}\label{eqqn10}
\sum_{n\in I}\chi(n)\ll \delta P\exp\Big(-\frac{1}{10}\sqrt{\log P}\Big)    
\end{align}
whenever $\chi\pmod q$ is non-principal and $I$ is an interval of length $|I| \in [P^{0.2},P]$.

Let $R$ be the quantity on the left-hand side of~\eqref{eqqn7}. For any $k\in \mathbb{N}$ we have by Chebyshev's inequality
\begin{align*}
R&\ll \Big(\frac{\log P}{\delta P}\Big)^{2k}V^{-2k} \sum_{\chi \pmod{q}}\Big|\sum_{P\leq p\leq (1+\delta)P}a_p\chi(p)\Big|^{2k}\nonumber\\
&= \Big(\frac{\log P}{\delta P}\Big)^{2k}V^{-2k} \sum_{P\leq p_1,\ldots, p_{2k}\leq (1+\delta)P}a_{p_1}\cdots a_{p_k}\overline{a_{p_{k+1}}}\cdots \overline{a_{p_{2k}}}\sum_{\chi \pmod{q}}\chi(p_1\cdots p_k)\bar{\chi}(p_{k+1}\cdots p_{2k})\nonumber\\
&\leq   \Big(\frac{\log P}{\delta P}\Big)^{2k}V^{-2k}\varphi(q)\sum_{\substack{P\leq p_1,\ldots, p_{2k}\leq (1+\delta)P\\(p_1\cdots p_{2k},q)=1}}1_{p_1\cdots p_k\equiv p_{k+1}\cdots p_{2k}\pmod{q}}.
\end{align*}
We pick $k=\lfloor \frac{3\log (qP)}{\log P} \rfloor$, so that $3\leq k\ll (\log \log q)^{0.9}$. 

Let $\nu(n)$ be the sieve majorant coming from the linear sieve with sifting level $D=P^{\rho}$ and sifting parameter $z=P^{\rho^2}$, where $\rho>0$ is a small enough absolute constant (say $\rho=1/100$). The sieve weight takes the form
\begin{align*}
\nu(n)=\sum_{\substack{d\mid n\\d\leq P^{\rho}}} \lambda_d
\end{align*}
for some $\lambda_d\in [-1,1]$. Then $R$ is bounded by
\begin{align}\label{eqq118}
R&\ll \Big(\frac{\log P}{\delta P}\Big)^{2k}V^{-2k}\varphi(q)\sum_{\substack{P\leq n_1,\ldots,n_{2k}\leq (1+\delta) P\\(n_1\cdots n_{2k},q)=1}}\nu(n_1)\cdots \nu(n_{2k})1_{n_1\cdots n_k\equiv n_{k+1}\cdots n_{2k}\pmod q}\nonumber\\
&=\Big(\frac{\log P}{\delta P}\Big)^{2k}V^{-2k}\sum_{\chi \pmod{q}}\Big|\sum_{P\leq n\leq (1+\delta)P}\nu(n)\chi(n)\Big|^{2k}. 
\end{align}

The contribution of the principal character to the $\chi$ sum is
\begin{align*}
\leq \Big(\sum_{P\leq n\leq (1+\delta)P}\nu(n)\Big)^{2k}\ll \Big(\frac{2\rho^{-2}\delta P}{\log P}\Big)^{2k}    
\end{align*}
by the linear sieve, and this contribution is admissible by setting $C=2\rho^{-2}$ in the lemma.
Consider next when $\chi$ is non-principal. Exchanging the order of summation and applying~\eqref{eqqn10}, we have the upper bound
\begin{align*}
 \sum_{P\leq n\leq (1+\delta)P}\nu(n)\chi(n)&=\sum_{d\leq P^{\rho}}\lambda_d \chi(d) \sum_{P/d\leq m\leq (1+\delta)P/d}\chi(m) \ll P(\log P)\exp\Big(-\frac{1}{10}\sqrt{\log P}\Big)\\&\ll P\exp\Big(-\frac{1}{15}\sqrt{\log P}\Big).
\end{align*}
Hence the contribution of the non-principal characters to the $\chi$ sum in~\eqref{eqq118} is bounded by
\begin{align*}
\ll P^{2}\exp\left(-\frac{2}{15}\sqrt{\log P}\right)\sum_{\chi \pmod{q}}\Big|\sum_{P\leq n\leq (1+\delta)P}\nu(n)\chi(n)\Big|^{2(k-1)},   
\end{align*}
and expanding out the moment again, this is
\begin{align*}
&\ll  P^{2}\exp\Big(-\frac{\sqrt{\log P}}{15}\Big) \phi(q)\sum_{\substack{P\leq n_1,\ldots,n_{2(k-1)}\leq (1+\delta) P\\(n_1\cdots n_{2(k-1)},q)=1}}\nu(n_1)\cdots \nu(n_{2(k-1)})1_{n_1\cdots n_{k-1}\equiv n_{k}\cdots n_{2(k-1)}\pmod q}\\
&\ll P^{2}\exp\Big(-\frac{\sqrt{\log P}}{15}\Big)\phi(q) \sum_{\substack{P\leq n_1,\ldots,n_{2(k-1)}\leq (1+\delta) P\\(n_1\cdots n_{2(k-1)},q)=1}}\tau(n_1)\cdots \tau(n_{2(k-1)})1_{n_1\cdots n_{k-1}\equiv n_{k}\cdots n_{2(k-1)}\pmod q}.
\end{align*}

Merging variables, this becomes
\begin{align*}
&\ll P^{2}\exp\Big(-\frac{\sqrt{\log P}}{15}\Big)\phi(q) \sum_{\substack{m_1,m_2\leq (2P)^{k-1}\\m_1\equiv m_2\pmod q\\(m_1m_2,q)=1}}\tau_{2(k-1)}(m_1)\tau_{2(k-1)}(m_2)\\
&= P^{2}\exp\Big(-\frac{\sqrt{\log P}}{15}\Big)\phi(q)\sum_{\substack{m_1\leq (2P)^{k-1}\\(m_1,q)=1}}\tau_{2(k-1)}(m_1)\sum_{\substack{m_2\leq (2P)^{k-1}\\m_2\equiv m_1\pmod q}}\tau_{2(k-1)}(m_2).
\end{align*}

Shiu's bound~\cite{shiu} shows that the inner sum is $\ll \frac{(2P)^{k-1}}{\varphi(q)}(\log (2P)^k)^{2(k-1)-1}$, as $q\leq (2P)^{0.9(k-1)}$ by our choice of $k$. Thus the whole expression above is
\begin{align*}
\ll P^{2k}\exp\Big(-\frac{1}{20}\sqrt{\log P}\Big),    
\end{align*}
since $(k\log P)^{2k}\ll \exp((\log P)^{0.01})$. When we multiply this contribution by $(\log P/(\delta P))^{2k}V^{-2k}$ and recall the assumptions $\delta, V\geq \exp(-(\log P)^{0.49})$ and the fact that $k\ll \log \log P$, we see that
\begin{align*}
R&\ll (CV^{-1})^{2k}+(\delta^{-1}V^{-1})^{2k}\exp\Big(-\frac{1}{30}\sqrt{\log P}\Big)\\
&\ll (CV^{-1})^{2k}+1,
\end{align*}
which, recalling our choice of $k$, is what was to be shown.
\end{proof}

Our next lemma improves on Proposition~\ref{prop_sup_hyb} for smooth moduli (apart from the $t$-aspect).

\begin{lemma}\label{le_smoothsup} Let $x\geq 10$, $\kappa>0$ and $2\leq P<Q\leq x$. Then for all $q\leq x$ satisfying $P^+(q)\le q^{\kappa^{100}}$ and for any multiplicative function $f\colon \mathbb{N}\to \mathbb{U}$, if $\chi_1\pmod q$ is defined as in Theorem~\ref{MRAPTHM} then 
\begin{align*}
\sup_{\substack{\chi \pmod q\\\chi\neq \chi_1}}\sup_{y\in [x^{\kappa},x]}\Big|\frac{1}{y}\sum_{\substack{n\leq y \\ (n,[P,Q]) = 1}}f(n)\bar{\chi}(n)\Big|\ll \kappa \left(\frac{\log Q}{\log P}\right)\frac{\phi(q)}{q}.    
\end{align*}
\end{lemma}

\begin{proof}
We may assume in what follows that $\kappa>0$ is small enough (adjusting the implied constant  if necessary). We may also assume $q^{\kappa^{100}}\geq 2$, so $\kappa\gg (\log q)^{-0.01}$. 

Note that $n \mapsto 1_{(n,[P,Q]) = 1}$ is multiplicative, and that for any $g_1,g_2\colon \mb{N} \ra \mb{U}$ and any $y\geq 2$ we have
$$
\mb{D}_q(g_1 1_{(\cdot,[P,Q]) = 1}, g_2;y)^2 \geq \mb{D}_q(g_1,g_2;y)^2 - \sum_{P \leq p \leq Q} \frac{1}{p} = \mb{D}_q(g_1,g_2;y)^2 - \log\left(\frac{\log Q}{\log P}\right) + O(1).
$$
Hence, following the beginning of the proof of Proposition~\ref{prop_sup_hyb} almost verbatim, we obtain the result once we prove that
\begin{align*}
\sup_{|t|\leq (\log q)^{0.02}}\mathbb{D}_q(\xi,n^{it};x^{\kappa})^2\geq 5.5\log \frac{1}{\kappa} + \log frac{q}{\phi(q)} + O(1)    
\end{align*}
for all non-principal characters $\xi\pmod q$.

From Lemma~\ref{le3.2}, it follows that
\begin{align*}
\mathbb{D}(\xi,n^{it};x^{\kappa})^2\geq \log \log x^{\kappa}-\log |L(\sigma+it,\xi)|-O(1),
\end{align*}
where $\sigma=1+1/(\log x^{\kappa})$, so it suffices to show that
\begin{align*}
\sup_{|t|\leq (\log q)^{0.02}}|L(\sigma+it,\xi)|\ll \kappa^{6.5} \frac{\phi(q)}{q}\log q   
\end{align*}
for all $q\leq x$ satisfying $P^{+}(q)\leq q^{\kappa^{100}}$. By partial summation and the P\'{o}lya--Vinogradov theorem,
\begin{align*}
L(\sigma+it,\xi)=\sum_{n\leq q(|t|+1)}\frac{\xi(n)}{n^{\sigma+it}}+O(1). 
\end{align*}

Let $q'=q^{10000\kappa^{100}}$. Then 
\begin{align*}
|L(\sigma+it,\xi)|\ll \frac{\phi(q)}{q}\log q'+  \Big|\sum_{q'\leq n\leq q(|t|+1)}\frac{\xi(n)}{n^{\sigma+it}}\Big|+1.  
\end{align*}
The first term on the right-hand side is acceptable. For the second term, we apply partial summation to write it  as
\begin{align}\label{eqq42}
\sum_{q'\leq n\leq q(|t|+1)}\frac{\xi(n)}{n^{\sigma+it}}=\frac{S(q',q(|t|+1))}{(q(|t|+1))^{\sigma+it}} + (\sigma+it)\int_{q'}^{q(|t|+1)}S(q',u)u^{-1-\sigma-it}\, du, \end{align}
where
\begin{align*}
S(M,N)\coloneqq  \sum_{M\leq n\leq N}\xi(n).     
\end{align*}

We are now in a position to apply Lemma~\ref{ChangGen} (which is applicable, as $P^{+}(q)\leq q^{\kappa^{100}}\leq (q')^{0.0001}$). The triangle inequality and Lemma~\ref{ChangGen} imply that whenever $q' < u \leq 2q'$  we have
$$
|S(q',u)| \leq |S(q'/2,q')| + |S(q'/2,u)| \ll u \exp\big(-\frac{1}{4}\sqrt{\log (q'/2)}\big) \ll u \exp\big(-\frac{1}{10}\sqrt{\log u}\big),
$$
a bound that continues to hold for all $u \geq 2q'$ directly from Lemma~\ref{ChangGen}. 
Thus, we may bound the right-hand side of~\eqref{eqq42} by  
\begin{align*}
\ll 1+(1+|t|)\int_{q'}^{\infty}\frac{1}{u\exp(\frac{1}{10}\sqrt{\log u})}\, du\ll 1\ll \kappa^{10}\frac{\phi(q)}{q}\log q,
 \end{align*}
 since $\phi(q)/q\gg 1/\log \log q$ and $\kappa\geq (\log q)^{-0.01}$. This concludes the proof. 
\end{proof}

From the previous lemma, we derive the following variant of Corollary~\ref{cor_sup} for smooth moduli $q$, again without exceptions.
\begin{cor}\label{cor_smoothsup} Let $x\geq R\geq 10$, $\kappa\in ((\log x)^{-1/100},1)$, and $10\leq P\leq Q/2\leq x$. Let the twisted character sum $R(\chi,s)$, multiplicative function $f\colon \mathbb{N}\to \mathbb{U}$ and character $\chi_1\pmod q$ be defined as in Corollary~\ref{cor_sup}. Then for $2\leq q\leq x$ satisfying $P^{+}(q)\leq q^{\kappa^{100}}$ we have
\begin{align*}
\sup_{\substack{\chi\pmod q\\\chi\neq \chi_1}}\sup_{R\in [x^{1/2},x]}\frac{1}{R}|R(\chi,0)|\ll \kappa \Big(\frac{\log x}{\log P}\Big)^2\frac{\phi(q)}{q}.    
\end{align*}
\end{cor}

\begin{proof} 
Applying the hyperbola method, 
we see that
\begin{align}\label{eqq44}\begin{split}
|R(\chi,0)|&\ll \Big|\sum_{\substack{m_1\leq Rx^{-\kappa}\\p\mid m_1\Longrightarrow p\in [P,Q]}}\frac{f(m_1)\overline{\chi}(m_1)}{1+\omega_{[P,Q]}(m_1)}
\sum_{\substack{R/m_1\leq m_2 \leq 2R/m_1 \\ (m_2,[P,Q]) = 1}}f(m_2)\overline{\chi}(m_2)\Big|\\
&+\sum_{\substack{m_2\leq 2x^{\kappa}\\(m_2,[P,Q])=1\\(m_2,q)=1}}\sum_{\substack{Rx^{-\kappa}<m_1\leq 2R/m_2\\p\mid m_1\Longrightarrow p\in [P,Q]}}1.
\end{split}
 \end{align}
Since $R/m_1\geq x^{\kappa}$ holds in the first sum on the right, we can apply Lemma~\ref{le_smoothsup} to bound this sum by
\begin{align*}
\ll \kappa\frac{\log Q}{\log P}\frac{\phi(q)}{q} \Big(\sum_{\substack{m_1\leq x\\p\mid m_1\Longrightarrow p\in [P,Q]}}\frac{R}{m_1}\Big)
&\ll \kappa \frac{\log Q}{\log P} \frac{\phi(q)}{q} R\prod_{P\leq p\leq Q}\Big(1-\frac{1}{p}\Big)^{-1}\\
&\ll \kappa\frac{\phi(q)}{q} \Big(\frac{\log Q}{\log P}\Big)^2R.  
\end{align*}
The second sum on the right of~\eqref{eqq44}, in turn, is bounded using Selberg's sieve by
\begin{align*}
\leq \sum_{\substack{m_2\leq 2x^{\kappa}\\(m_2,[P,Q])=1\\(m_2,q)=1}}\sum_{\substack{m_1\leq 2R/m_2\\ P^-(m_1) > P}} 1 \ll \sum_{\substack{m_2\leq 2x^{\kappa}\\(m_2,q)=1}}\frac{R}{m_2\log P}\ll \kappa \frac{\log x}{\log P} \frac{\phi(q)}{q}R,    
\end{align*}
using the fact that $p|q \Rightarrow p \leq q^{\kappa^{100}} < 2x^{\kappa}$ in the final step.
\end{proof}

\begin{proof}[Proof of Theorem~\ref{MRAPTHM-smooth}.]  Inspecting the proof of Theorem~\ref{MRAPHybrid}, the result of that theorem holds for any modulus $q\leq x$ satisfying, for $H=\lfloor \varepsilon^{-1} \rfloor,$ the bounds  
\begin{align}\label{eqqn6}
\sup_{P\in [x^{\varepsilon^2},x^{\varepsilon}]}\Big|\Big\{\chi\pmod q:\,\, \Big|\sum_{P\leq p\leq Pe^{1/H}}f(p)\chi(p)\Big|\geq \varepsilon^{2}/10 \cdot \frac{P}{H\log P}\Big\}\Big|\ll K(\varepsilon),    
\end{align} 
and for $P = x^{\varepsilon^2}$ and $Q = x^{\varepsilon}$,
\begin{align}\label{eqqn43}
\sup_{\substack{\chi\pmod q\\\chi\neq \chi_1}}\Big|\frac{1}{R}\sup_{R\in [x^{1/2},x]}\sum_{R\leq m\leq 2R}\frac{f(m)\overline{\chi}(m)}{1+\omega_{[P,Q]}(m)}\Big|\ll  \frac{\varepsilon}{K(\varepsilon)^{1/2}}\frac{\varphi(q)}{q}   
\end{align}
for some function $K(\varepsilon)\geq 1$. Indeed, it is only the $\mc{U}_L$ case of the proof of Theorem~\ref{MRAPTHM} where we need to assume something about the modulus $q$, and the assumptions that we need there are precisely a large values estimate of the form~\eqref{eqqn6} together with a pointwise bound of the type~\eqref{eqqn43}.

We then establish~\eqref{eqqn6} and~\eqref{eqqn43}. Let $P^{+}(q)\leq q^{\varepsilon'}$ with $\varepsilon'=\exp(-\varepsilon^{-3})$. Lemma~\ref{le_moments} (where we take $V=\varepsilon^{2}/10$ and $\delta=e^{1/H}-1$) readily provides~\eqref{eqqn6} with $K(\varepsilon)=\varepsilon^{-100\varepsilon^{-2}}$ (assuming as we may that $\varepsilon>0$ is smaller than any fixed constant).

Corollary~\ref{cor_smoothsup} in turn gives~\eqref{eqqn43} (with the same $K(\varepsilon)=\varepsilon^{-100\varepsilon^{-2}}$ as above) when we take $\kappa=\varepsilon^{5}K(\varepsilon)^{-1/2}$ there, which we can do since $P^{+}(q)\leq q^{\varepsilon'}\leq q^{\kappa^{100}}$. This completes the proof.
\end{proof}

\section{All moduli in the square-root range}

\label{sec: bilinear}

\subsection{Preliminary lemmas}

For the proof of Theorem~\ref{BVTHM}, we need a few estimates concerning smooth and rough numbers to bound the error terms arising from exhibiting good factorizations for smooth numbers in Lemmas~\ref{le_decouple} and~\ref{le_reduce2}.

\begin{lemma} \label{le_sel_sieve}
Let $c \in (0,1)$. Let $1 \leq Y \leq X$ and $1\leq q \leq X^{1-c}$, and let $X^{-c/2}\leq  \delta \leq 1$. Then for any reduced residue class $a$ modulo $q$,
$$\sum_{\substack{(1-\delta)X < m \leq X\\ P^-(m) > Y \\ m \equiv a \pmod{q}}} 1 \ll c^{-1} \frac{\delta X}{\phi(q)\log Y}.$$
\end{lemma}
\begin{proof}
This follows immediately from Selberg's sieve.
\end{proof}

Given $1 \leq q,Y \leq X$, define the counting function of $Y$-smooth numbers up to $X$ that are coprime to $q$ as
\begin{align}\label{eqq117}
\Psi_q(X,Y) \coloneqq |\{n \leq X : P^+(n) \leq Y, (n,q) = 1\}|.
\end{align}
We have the following estimate for $\Psi_q(X,Y)$ in short intervals. 

\begin{lemma} \label{le_qsmooth_short}
Let $10 \leq Y \leq X$ and set $u \coloneqq \log X/\log Y$. Assume that $Y \geq \exp((\log X)^{0.99})$ and $\exp(-(\log X)^{0.01}) \leq \delta \leq 1$. Finally, let $1 \leq q \leq e^{\sqrt{Y}}$. Then
\begin{align*}
\Psi_q((1+\delta) X,Y) -\Psi_q(X,Y) \ll  \rho(u)\frac{\phi(q)}{q}\delta X.
\end{align*}
\end{lemma}

\begin{proof}
By the sieve of Eratosthenes, we have 
\begin{align*}
\Psi_q((1+\delta) X,Y) -\Psi_q(X,Y)=\sum_{\substack{d\mid q\\P^{+}(d)\leq Y}}\mu(d)\sum_{\substack{\frac{X}{d}\leq m\leq (1+\delta)\frac{X}{d}\\P^{+}(m)\leq Y}}1    
\end{align*}
Let $S_1$ and $S_2$ be parts of the sum with $d\leq \exp(10(\log X)^{1/2})$ and $d>\exp(10(\log X)^{1/2})$, respectively. For estimating $S_2$, we crudely remove the smoothness condition from the $m$ and $d$ sums, and estimate the remaining sum using $1/d\leq \exp(-5(\log X)^{1/2})/\sqrt{d}$ to obtain
\begin{align*}
S_2&\ll \sum_{\substack{d\mid q\\d>\exp(10(\log X)^{1/2})}} \frac{\delta X}{d} \ll \delta X\exp(-5(\log X)^{1/2})\prod_{p\mid q}\Big(1-\frac{1}{\sqrt{p}}\Big)^{-1}\\
&\ll \delta X\exp(-5(\log X)^{1/2})\exp(3\sqrt{\omega(q)})
\end{align*}
and using $\omega(q)=o(\log q)$ this is certainly $\ll \delta X \rho(u)\frac{\phi(q)}{q}\exp(-\frac{1}{2}(\log X)^{1/2})$ by $u\leq (\log X)^{0.01}$ and the well-known estimate $\rho(u)=u^{-(1+o(1))u}$.

For the $S_1$ sum, we instead apply~\cite[Theorem 5.1]{hildebrand-tenenbaum} (noting that its hypothesis $\delta X/d\geq XY^{-5/12}$ is satisfied) so that we obtain
\begin{align} \label{eq:errorS1}
S_1&=\sum_{\substack{d\mid q\\d\leq \exp(10(\log X)^{1/2})\\P^{+}(d)\leq Y}}\Big(\mu(d)\frac{\delta X}{d}\rho\left(u-\frac{\log d}{\log Y}\right)\Big(1+O\Big(\frac{\log(u+1)}{\log Y}\Big)\Big)\Big) \nonumber\\
&=\sum_{\substack{d\mid q\\P^{+}(d)\leq Y}}\Big(\mu(d)\frac{\delta X}{d}\rho(u)\Big(1+O\Big(\frac{\log(u+1)}{\log Y}\Big)\Big)\Big) \nonumber\\
&+ O\Big(\delta X\exp(3\sqrt{\omega(q)}-5(\log X)^{1/2}) + \delta X\sum_{\substack{d\mid q\\d\leq \exp(10(\log X)^{1/2})\\P^{+}(d)\leq Y}} \frac{|\rho(u)-\rho(u-\frac{\log d}{\log Y})|}{d}\Big),
\end{align}
where we used the same bound as in the $S_2$ case to extend the $d$ sum to all $d\mid q, P^{+}(d)\leq Y$. 

As with $S_2$, the first error term in \eqref{eq:errorS1} is $\ll \delta X \rho(u)\frac{\phi(q)}{q}\exp(-\frac{1}{2}(\log X)^{1/2})$.
To treat the second, we apply the mean value theorem and the identity $u\rho'(u) = -\rho(u-1)$, for $d\leq \exp(10(\log X)^{1/2})$ to obtain 
\begin{align*}
|\rho(u-\frac{\log d}{\log Y}) - \rho(u)|\leq \frac{\log d}{\log Y} \max_{u-10(\log X)^{1/2}/\log Y \leq v \leq u} \frac{\rho(v-1)}{v} \ll \rho(u-2)\frac{(\log X)^{1/2}}{\log Y},
\end{align*}
and therefore the expression for $S_1$ simplifies to
\begin{align*}
&S_1=\delta \rho(u)X\Big(\prod_{\substack{p\mid q\\p\leq Y}}\Big(1-\frac{1}{p}\Big) +O\Big(\sum_{d\mid q}\frac{1}{d}\left(\frac{\log(u+1)}{\log Y}+\frac{\rho(u-2)}{\rho(u)}(\log X)^{-0.3}\right)\Big)\Big).
\end{align*}

 Now the result follows by recalling that $u\leq (\log X)^{0.01}$ and noting that the product over $p\mid q$ is $\asymp \frac{\phi(q)}{q}$ as $Y\geq \log^2 q$ and that $\rho(u-2)\ll u^{3}\rho(u)$ by~\cite[Formulas (2.8) and (2.4)]{hildebrand-tenenbaum}.
\end{proof} 

\begin{cor}\label{cor_rankin}
Let $1 \leq Y \leq X_1 < X_2 \leq e^{\sqrt{Y}}$ and $1\leq q\leq X_2$, with $Y \geq \exp((\log X_1)^{0.99})$. Then
$$\sum_{\substack{X_1 < n \leq X_2\\ P^+(n) \leq Y\\(n,q)=1}} \frac{1}{n} \ll\frac{\phi(q)}{q} \rho(u_1) \log \frac{2X_2}{X_1},$$
where $u_1 \coloneqq (\log X_1)/\log Y$. 
\end{cor}

\begin{proof}
Decompose the interval $(X_1,X_2]$ dyadically. Making use of Lemma~\ref{le_qsmooth_short}, we find
\begin{align*}
\sum_{\substack{X_1 < n \leq 2X_2\\ P^+(n) \leq Y\\(n,q)=1}} \frac{1}{n} \ll \sum_{X_1 < 2^j \leq 4X_2} 2^{-j} \sum_{\substack{2^{j-1}<n\leq 2^j\\ P^+(n) \leq Y\\(n,q)=1}} 1 \ll \frac{\phi(q)}{q}\rho(u_1) \sum_{X_1 < 2^j \leq 4X_2} 1 \ll \frac{\phi(q)}{q}\rho(u_1) \log\frac{2X_2}{X_1},
\end{align*}
as claimed.
\end{proof}

\subsection{Decoupling of variables}

The proof of Theorem~\ref{BVTHM} is based on obtaining bilinear structure in the sum, coming from the fact that the summation may be restricted to smooth numbers. Certainly any $x^{\eta}$-smooth number $n\in [x^{1-\eta},x]$ can be written as $n=dm$ with $d,m\in [x^{1/2-\eta},x^{1/2+\eta}]$, but a typical smooth number has a lot of representations of the above form, and therefore it appears nontrivial to decouple the $d$ and $m$ variables just from this. The following simple lemma however provides a more specific factorization that does allow the decoupling of our variables. 

\begin{lemma}\label{le_comb}
Let $x\geq 4$, and let $n\in [x^{1/2},x]$ be an integer. Then $n$ can be written uniquely as $dm$ with $d\in [x^{1/2}/P^-(m),x^{1/2})$ and $P^{+}(d)\leq P^{-}(m)$. 
\end{lemma}

\begin{proof}
Let $n=p_1p_2\cdots p_k$, where $p_1\leq p_2\leq \cdots \leq p_k$ are primes. Let $r\geq 1$ be the smallest index for which $p_1\cdots p_r\geq x^{1/2}$. Then $d=p_1\cdots p_{r-1}$, $m=p_{r}\cdots p_k$ works. We still need to show that this is the only possible choice of $d$ and $m$.

Let $d$ and $m$ be as in the lemma. Since $dm=p_1\cdots p_k$ and $P^{+}(d)\leq P^{-}(m)$, there exists $r\geq 1$ such that $d=p_1\cdots p_{r-1}$, $m=p_{r}\cdots p_k$, and by the condition on the size of $d$ we must have $p_1\cdots p_{r-1}<x^{1/2}$, $p_1\cdots p_{r-1}\geq x^{1/2}/p_r$. There is exactly one suitable $r$, namely the smallest $r$ with $p_1\cdots p_r\geq x^{1/2}$.
\end{proof}

We need to be able to control the size of the $P^{-}(m)$ variable, since if it is very small then so is $P^{+}(d)$, leading to character sums over very sparse sets. The next lemma says that for typical $n\leq x$ the corresponding $P^{-}(m)$ is reasonably large, even if $n$ is restricted to an arithmetic progression.

In what follows, set
\begin{align}\label{eqq105}
\theta_j\coloneqq\eta(1-\varepsilon^2)^j\quad \textnormal{for all}\quad j\geq 0,
\end{align}
and let 
\begin{align}\label{eqq116}
J\coloneqq\lceil \varepsilon^{-2}\log \log (1/\varepsilon)\rceil    
\end{align}
so that for small $\varepsilon>0$ we have
\begin{align*}
\theta_J \asymp_{\eta}1/\log\frac{1}{\varepsilon}\quad \textnormal{and}\quad \rho(1/(3\theta_{J}))\ll (3\theta_J)^{1/(6\theta_J)}\ll \varepsilon^{100}. 
\end{align*}
We have $J\leq 2\varepsilon^{-2}\log\log (1/\varepsilon)$ as long as $\varepsilon>0$ is small enough in terms of $\eta$.

\begin{lemma}[Restricting to numbers with specific factorizations]\label{le_decouple}
Let $x\geq 10$, $\eta\in (0,1/10)$ and $(\log x)^{-1/100}\leq \varepsilon\leq 1$. Let $\theta_j$ be given by~\eqref{eqq105} and $J$ given by~\eqref{eqq116}, and define
$$\mathcal{S}_{J}\coloneqq\bigcup_{0\leq j\leq J}\{n\leq x:\,\, n=dm,\,\, d\in (x^{1/2-\theta_{j+1}},x^{1/2}),\,\,P^{+}(d)\leq x^{\theta_{j+1}},\,\, P^{-}(m)\in (x^{\theta_{j+1}},x^{\theta_j}]\}.$$ 
Let $q\leq x^{1/2-100\eta}$. Then for $(a,q)=1$ we have
\begin{align*}
\sum_{\substack{n\leq x\\n\equiv a\pmod q\\P^{+}(n)\leq x^{\eta}}}(1-1_{\mathcal{S}_{J}}(n))\ll_{\eta} \varepsilon\frac{x}{q}.   
\end{align*}
\end{lemma}

\begin{proof}
We may assume that $\varepsilon$ is smaller than any fixed function of $\eta$. In what follows, let $n\leq x$, $P^{+}(n)\leq x^{\eta}$ and $n\equiv a \pmod q$ with $(a,q)=1$.

Owing to Lemma~\ref{le_comb}, we may write any $n$ as above uniquely in the form $n=dm$ with $P^{+}(d)\leq P^{-}(m)$ and $d\in [x^{1/2}/P^{-}(m),x^{1/2})$. Let us further denote by $\mathcal{T}_j$ the set of $n$ as above for which $P^{-}(m)\in (x^{\theta_{j+1}},x^{\theta_j}]$, so that every $n$ belongs to a unique set $\mathcal{T}_j$ with $j\geq 0$. We claim that $n\in \mc{S}_J$ unless one of the following holds:\\
(i) $n$ has a divisor $d\geq x^{1/2-\eta}$ with $P^{+}(d)\leq x^{\theta_{J+1}}$ and $P^{-}(n/d)\geq P^{+}(d)$;\\
(ii) For some $0 \leq j \leq J$ there exist two (not necessarily distinct) primes $p_1,p_2>x^{\theta_{j+1}}$ with $p_1p_2\mid n$ and $1\leq p_1/p_2\leq x^{\varepsilon^2}$;\\
(iii) For some $0\leq j\leq J$, we can write $n=md$ with $d\in [x^{1/2-\theta_{j}},x^{1/2-\theta_{j+1}}]$, $P^{+}(d)\leq x^{\theta_{j}}$,  $P^{-}(m)\in (x^{\theta_{j+1}},x^{\theta_j}]$.\\
Indeed, if $n\leq x$, $P^{+}(n)\leq x^{\eta}$ and none of (i), (ii), (iii) holds, then letting $j$ be the index for which $n\in \mathcal{T}_j$, we have $j\leq J$ (by negation of (i)) and in the factorization $n=dm$ of $n$ we have the conditions $P^{-}(m)\in (x^{\theta_{j+1}},x^{\theta_j}]$, $P^{+}(d)\leq x^{\theta_{j+1}}$ (by negation of (ii) and the fact that $\theta_{j}-\theta_{j+1}\leq \varepsilon^2$), and $d\in (x^{1/2-\theta_{j+1}},x^{1/2}]$ (by negation of (iii)), so that $n\in \mathcal{S}_J$.

Applying Lemma~\ref{le_sel_sieve}, the contribution of (i) is 
\begin{align*}
\ll \sum_{\substack{x^{1/2-\eta}\leq d\leq x^{1/2}\\P^{+}(d)\leq x^{\theta_{J+1}}\\(d,q)=1}}\sum_{\substack{m\leq x/d\\P^{-}(m)\geq P^{+}(d)\\m\equiv ad^{-1} \pmod q}}1&\ll \eta^{-1} \sum_{\substack{x^{1/2-\eta}\leq d\leq x^{1/2}\\P^{+}(d)\leq x^{\theta_{J+1}}\\(d,q)=1}}\frac{x/d}{\phi(q)(\log P^{+}(d))}\\
&\ll_{\eta}\sum_{k\geq \log(1/\theta_{J+1})-1}e^{-k}\sum_{\substack{x^{1/2-\eta}\leq d\leq x^{1/2}\\P^{+}(d)\in [x^{e^{-k-1}},\,x^{e^{-k}}]\\(d,q)=1}}\frac{1}{d\log x}\cdot\frac{x}{\phi(q)}.
\end{align*}
Set $u_0\coloneqq(\log x)^{0.01}$. The contribution of the terms with $e^{k}\leq u_0$ can be bounded using Lemma~\ref{le_qsmooth_short}, and $\rho(u)\ll u^{-u}$ (see~\cite[(2.6)]{hildebrand-tenenbaum}), yielding a contribution of
\begin{align*}
\ll\sum_{k\geq \log(1/\theta_{J+1})-1}e^{-k}\rho(e^k/3)\frac{x}{q}\ll \sum_{k\geq \log(1/\theta_{J+1})-1}e^{-(k-\log 3)e^{-k}/3}\ll \varepsilon^{100}\frac{x}{q}, 
\end{align*}
since $\theta_{J+1}\gg_{\eta} 1/\log(1/\varepsilon)$. The remaining terms with $e^{k}>u_0$ can be estimated trivially using Corollary~\ref{cor_rankin}, giving 
\begin{align*}
\ll \eta^{-1} \sum_{k\geq 0.01\log \log x} e^{-k}\rho(u_0/3)\frac{x}{q}\ll_{\eta} \varepsilon\frac{x}{q}.  
\end{align*}

Denoting $M=\theta_{J+1}\varepsilon^{-2}$ and applying the prime number theorem, the contribution of (ii) in turn is bounded by
\begin{align*}
\sum_{M\leq k\leq \varepsilon^{-2}}\sum_{p_1,p_2\in [x^{(k-1)\varepsilon^2},x^{(k+1)\varepsilon^2}]}\sum_{\substack{n\leq x\\n\equiv a \pmod q\\p_1p_2\mid n}}1&\ll \frac{x}{q}\sum_{M\leq k\leq \varepsilon^{-2}}\Big(\sum_{p\in [x^{(k-1)\varepsilon^2},x^{(k+1)\varepsilon^2}]}\frac{1}{p}\Big)^2\\
&\ll \frac{x}{q}\sum_{M\leq k\leq \varepsilon^{-2}}\Big(\log\Big(\frac{k+1}{k-1}\Big)+(\log x)^{-100}\Big)^2\\
&\ll \frac{x}{q}\sum_{M\leq k\leq \varepsilon^{-2}}\Big(\frac{1}{k^2}+(\log x)^{-100}\Big)\\
&\ll \frac{x}{qM},
\end{align*}
and by the definition of $M$ and the fact that $\theta_{J+1}\ll_{\eta} 1/\log(1/\varepsilon)$, this is $\ll_{\eta} \varepsilon\frac{x}{q}$. 

Lastly, by Lemma~\ref{le_sel_sieve} and Corollary~\ref{cor_rankin}, for any fixed $0\leq j\leq J$, the contribution of (iii) is
\begin{align*}
&\sum_{\substack{x^{1/2-\theta_j}\leq r\leq x^{1/2-\theta_{j+1}}\\P^{+}(r)\leq x^{\theta_j}\\(r,q)=1}}\sum_{\substack{s\leq x/r\\P^{-}(s)\in [x^{\theta_{j+1}},x^{\theta_j}]\\s\equiv ar^{-1}\pmod q}}1 
\leq \sum_{\substack{x^{1/2-\theta_j}\leq r\leq x^{1/2-\theta_{j+1}}\\P^{+}(r)\leq x^{\theta_{j}}\\(r,q)=1}}\sum_{\substack{p\in [x^{\theta_{j+1}},x^{\theta_j}]\\p\nmid q}}\sum_{\substack{s'\leq x/(pr)\\P^{-}(s')\geq x^{\theta_{j+1}}\\s'\equiv a(pr)^{-1}\pmod q}}1\\
&\ll \eta^{-1}\sum_{\substack{x^{1/2-\theta_j}\leq r\leq x^{1/2-\theta_{j+1}}\\P^{+}(r)\leq x^{\theta_j}\\(r,q)=1}}\sum_{p\in [x^{\theta_{j+1}},x^{\theta_j}]}\frac{x/(pr)}{\phi(q)\theta_{j+1}(\log x)}\\
&\ll_{\eta}\sum_{\substack{x^{1/2-\theta_j}\leq r\leq x^{1/2-\theta_{j+1}}\\P^{+}(r)\leq x^{\theta_j}\\(r,q)=1}}\frac{x}{\phi(q)r\theta_{j+1}(\log x)}\Big(\log\frac{\theta_{j}}{\theta_{j+1}}+(\log x)^{-100}\Big)\\
&\ll_{\eta}\frac{\theta_{j}-\theta_{j+1}}{\theta_{j+1}}\log \frac{\theta_j}{\theta_{j+1}}\rho(1/(3\theta_{j}))\frac{x}{q}+\frac{x}{q(\log x)^{99}}.
\end{align*}
Here the second term is certainly small enough. Using $\rho(u)\ll u^{-2}$, $\log(1+v)\ll v$ and formulas~\eqref{eqq105} and~\eqref{eqq116}, the first term summed over $0\leq j\leq J$ is crudely bounded by
\begin{align*}
\ll_{\eta} \sum_{0\leq j\leq J}(\theta_j-\theta_{j+1})^2\frac{x}{q}\ll_{\eta}J\varepsilon^4 \frac{x}{q} \ll_{\eta}\varepsilon^{1.9}\frac{x}{q}.   
\end{align*}
Therefore we have proved the assertion of the lemma.
\end{proof}

We further wish to split the $d$ and $m$ variables into short intervals to dispose of the cross-condition $dm\leq x$ on their product. This is achieved in the following lemma.

\begin{lemma}[Separating variables]\label{le_decouple2}
Let $x\geq 10$, $\eta\in (0,1/10)$ and $(\log x)^{-1/100} \leq  \e \leq 1$. Let $H\coloneqq\lfloor \varepsilon^{-1.1}\rfloor$. For each $0 \leq j \leq  J$ (with $J$ given by~\eqref{eqq116}) let $\theta_j$ be given by~\eqref{eqq105}, and write 
\begin{align}\label{eqq112}\begin{split}
\mc{I}_j :&= \{u\in \mathbb{Z}:\,\, H \theta_{j+1} \log x  \leq u \leq H \theta_j \log x-1\},\\
\mc{K}_j :&= \{v\in \mathbb{Z}:\,\, (1/2-\theta_{j+1})H\log x \leq v \leq \frac{1}{2}H \log x-1\}.
\end{split}
\end{align}
Define the set
\begin{align*}
\mc{S}'_{J}\coloneqq\bigcup_{0\leq j\leq J}\bigcup_{u\in \mc{I}_{j},v\in \mathcal{K}_j}&\{n=pdm',\,\, p\in (e^{u/H},e^{(u+1)/H}],\,\, d\in (e^{v/H},e^{(v+1)/H}],\,\, m'\leq xe^{-(u+v+2)/H},\\
&P^{+}(d)\leq x^{\theta_{j+1}},\,\, P^{-}(m')> x^{\theta_{j}}\}.    
\end{align*}
Let $q\leq x^{1/2-100\eta}$. Then for $(a,q)=1$ we have
\begin{align*}
\sum_{\substack{n\leq x\\n\equiv a\pmod q\\P^{+}(n)\leq x^{\eta}}}(1-1_{\mathcal{S}_{J}'}(n))\ll_{\eta} \varepsilon\frac{x}{q}.   
\end{align*}    
\end{lemma}

\begin{proof}
By Lemma~\ref{le_decouple}, it suffices to prove the claim with $1_{\mc{S}_J}(n)-1_{\mc{S}_J'}(n)$ in place of $1-1_{\mc{S}_J'}(n)$. We have $\mc{S}_J'\subset \mc{S}_J$, since for $n\in \mc{S}_J$ we have a unique way to write it, for some $0\leq j\leq J$, as $n=dm$ with $P^{+}(d)\leq x^{\theta_{j+1}}$, $P^{-}(m)\in (x^{\theta_{j+1}},x^{\theta_j}]$, and we may further write $m=pm'$, so that $p\in (x^{\theta_{j+1}},x^{\theta_j}]$ and $P^{-}(m')>p$.

Now, if we define $u_j^{(1)},u_j^{(2)}$ as the endpoints of the discrete interval $\mc{I}_j$, and similarly $v_j^{(1)}, v_j^{(2)}$ as the endpoints of $\mc{K}_j$, we see that $n\in \mc{S}_J$ belongs for unique $0\leq j\leq J$, $u\in \mc{I}_j$, $v\in \mc{K}_j$ to the set in the definition of $\mc{S}_J'$, unless one of the following holds for the factorization $n=pdm'$ of $n$:\\
(i) We have $p\in [e^{(u_j^{(i)}-1)/H},e^{(u_j^{(i)}+1)/H}]$ or $d\in [e^{(v_j^{(i)}-1)/H},e^{(v_j^{(i)}+1)/H}]$ for some $i\in \{1,2\}$ and $0\leq j\leq J$;\\
(ii) We have $p\in [e^{u/H},e^{(u+1)/H}]$, $d\in [e^{v/H},e^{(v+1)/H}]$, $m'\in [xe^{-(u+v+2)/H},xe^{-(u+v)/H}]$ for some $u\in \mc{I}_j$, $v\in \mc{K}_j$ and $0\leq j\leq J$.\\
(iii) We have $P^{-}(m')\in (x^{\theta_{j+1}},x^{\theta_j}]$.

Condition (iii) clearly leads to condition (ii) in the proof of Lemma~\ref{le_decouple} holding, so its contribution is $\ll_{\eta}\varepsilon x/q$.

We are left with the contributions of (i) and (ii). They are bounded similarly, so we only consider (ii).

For given $j,u,v$, Lemmas~\ref{le_sel_sieve} and~\ref{le_qsmooth_short} tell us that the contribution of (ii) is
\begin{align*}
&\sum_{\substack{e^{u/H}\leq p\leq e^{(u+1)/H}\\p\nmid q}}\sum_{\substack{e^{v/H}\leq d\leq e^{(v+1)/H}\\P^{+}(d)\leq x^{\theta_{j+1}}\\(d,q)=1}}\sum_{\substack{xe^{-(u+v+2)/H}\leq m'\leq xe^{-(u+v)/H}\\P^{-}(m')\geq x^{\theta_{j+1}}\\m'\equiv a(pd)^{-1}\pmod q}}1\\
&\ll \frac{\eta^{-1}}{H}\sum_{e^{u/H}< p\leq e^{(u+1)/H}}\sum_{\substack{e^{v/H}< d\leq e^{(v+1)/H}\\P^{+}(d)\leq x^{\theta_{j+1}}\\(d,q)=1}}\frac{xe^{-(u+v)/H}}{\phi(q)\theta_{j+1}\log x}\\
&\ll_{\eta} \frac{1}{H^2\theta_{j+1}}\rho\Big(1/(3\theta_{j+1})\Big)\frac{x}{uq\log x},
\end{align*}
where the second $1/H$ factor arose from summation over $d$ and the $1/u$ factor arose from the summation over $p$. Summing this over $u\in \mc{I}_j$, $v\in \mc{K}_j$ and $0\leq j\leq J$ and recalling that $|\mc{I}_j|\ll (\theta_j-\theta_{j+1})H(\log x)$, $|\mc{K}_j|\ll \theta_{j+1}H\log x$ and $\rho(y)\ll y^{-2}$ yields a bound of
\begin{align*}
\ll_{\eta} \sum_{0\leq j\leq J}(\theta_j-\theta_{j+1})\theta_{j+1}(H\log x)^2\cdot \frac{1}{H^2}\frac{1}{H\log^2 x}\cdot\frac{x}{q}\ll_{\eta}\frac{\varepsilon^2 J}{H}\cdot\frac{x}{q}\ll_{\eta}\varepsilon\frac{x}{q} \end{align*}
by the definitions of $H$ and $J$.
\end{proof}

Now that we have decoupled the variables, we may introduce Dirichlet characters and obtain a trilinear sum. For $u\in \mc{I}_j,v\in \mc{K}_j$ and $H=\lfloor \varepsilon^{-1.1}\rfloor$, write
\begin{align}\begin{split}\label{eqq113}
P_u(\chi)&=\sum_{e^{u/H}< p\leq e^{(u+1)/H}}f(p)\chi(p)\\
D_v(\chi)&=\sum_{\substack{e^{v/H}< d\leq e^{(v+1)/H}\\P^{+}(d) \leq x^{\theta_{j+1}}}} f(d)\chi(d),\\
M_{u,v}(\chi)&=\sum_{\substack{m\leq x/e^{(u+v+2)/H}\\P^{-}(m)  >x^{\theta_{j}}}} f(m)\chi(m).
\end{split}
\end{align}
Then we have the following.

\begin{lemma} \label{le_reduce2}
Let $x\geq 10$, $\eta\in (0,1/10)$, $\varepsilon\in ((\log x)^{-1/200},1)$, $q\leq x^{1/2-100\eta}$, and let $f\colon \mathbb{N}\to \mathbb{U}$ be a multiplicative function supported on $x^{\eta}$-smooth numbers. Letting $\chi_1$ be as in Theorem~\ref{BVTHM}, and recall the definitions in~\eqref{eqq112}. Then for $(a,q)=1$ we have
\begin{align*}
&\Big|\sum_{\substack{n \leq x\\ n \equiv a \pmod{q}}} f(n) - \frac{\chi_1(a)}{\phi(q)} \sum_{n \leq x} f(n)\overline{\chi_1}(n)\Big|\\
&\leq \frac{1}{\phi(q)}\sum_{\chi \neq \chi_1\pmod q} \sum_{0 \leq j \leq J}\sum_{u \in \mc{I}_j}\sum_{v \in \mc{K}_j} |P_u(\bar{\chi})||D_v(\bar{\chi})||M_{u,v}(\bar{\chi})|+O_{\eta}\Big(\frac{\varepsilon x}{q}\Big).
\end{align*}
\end{lemma}

\begin{proof}
Applying Lemma~\ref{le_decouple2} to both $f$ and $f\overline{\chi_1}$ and observing that the union of sets in the definition of $\mc{S}_J'$ is disjoint, we see that the left-hand side in the statement is
\begin{align*}
&\Big|\sum_{0 \leq j \leq J} \sum_{\substack{u \in \mc{I}_j\\ v \in \mc{K}_j}} \sum_{e^{u/H} < p \leq e^{(u+1)/H}} \sum_{\substack{e^{v/H} < d \leq e^{(v+1)/H}\\ P^+(d) \leq x^{\theta_{j+1}}}} \sum_{\substack{m \leq xe^{-(u+v+2)/H} \\ P^-(m)> x^{\theta_{j}}}} f(p)f(d)f(m)\xi_q(mdp)\Big|\\
&+ O_{\eta}\Big(\frac{\varepsilon x}{q}\Big),
\end{align*}
where 
\begin{align*}
\xi_q(n)\coloneqq1_{n \equiv a \pmod{q}} - \frac{\chi_1(a)}{\phi(q)}\overline{\chi_1}(n).  
\end{align*}
Making use of the orthogonality of characters and then applying the triangle inequality, the main term here is (omitting the summation ranges for brevity)
\begin{align*}
&\Big|\sum_{0 \leq j \leq  J} \sum_{\substack{u \in \mc{I}_j\\ v \in \mc{K}_j}}  \sum_{\chi  \neq \chi_1\pmod q}\frac{\chi(a)}{\phi(q)} \Big(\sum_{p} f(p)\overline{\chi}(p)\Big) \Big(\sum_{\substack{d\\ P^+(d) \leq x^{\theta_j+1}}} f(d)\overline{\chi}(d)\Big) \Big(\sum_{\substack{m\\ P^-(m) > x^{\theta_{j}}}} f(m)\overline{\chi}(m)\Big)\Big| \\
&\leq \frac{1}{\phi(q)}\sum_{0 \leq j \leq J}  \sum_{u \in \mc{I}_j} \sum_{v \in \mc{K}_j}\sum_{\chi \neq \chi_1\pmod q} |P_u(\bar{\chi})||D_v(\bar{\chi})||M_{u,v}(\bar{\chi})|,
\end{align*}
and the claim follows.
\end{proof}

\subsection{The main proof}

Let $\eta > 0$. Suppose henceforth that the multiplicative function $f\colon \mb{N} \ra \mb{U}$ is supported on $x^{\eta}$-smooth integers. Our task is to prove Theorem~\ref{BVTHM}, i.e., to obtain cancellation in the deviation
$$\max_{a \in \mb{Z}_q^{\times}} \Big|\sum_{\substack{n \leq x\\ n \equiv a \pmod{q}}} f(n) - \frac{\chi_1(a)}{\phi(q)} \sum_{n \leq x} f(n)\overline{\chi_1}(n)\Big|.$$
In what follows, let $(\log x)^{-1/200} \leq \e \leq 1$, let $\theta_j$ and $J$ be given by~\eqref{eqq105} and~\eqref{eqq116}, and recall the notation of~\eqref{eqq112} and~\eqref{eqq113}. 

According to Lemma~\ref{le_reduce2}, we can restrict ourselves to bounding the product of character sums present in that lemma. Taking the maximum over $(u,v)\in \mathcal{I}_j\times \mathcal{K}_j$ there, it suffices to prove that
\begin{align} \label{eq:trilinUVJ}
\sum_{0 \leq j\leq J}\frac{(\theta_{j}-\theta_{j+1})\theta_{j+1} H^2(\log x)^2}{\phi(q)}\sum_{\chi\neq \chi_1\pmod q}|P_{u_j}(\overline{\chi})||D_{v_j}(\overline{\chi})||M_{u_j,v_j}(\overline{\chi})|\ll_{\eta} \varepsilon\frac{x}{q},    
\end{align}
where for each $0\leq j\leq J$ the numbers $u_j\in \mc{I}_j$, $v_j\in \mc{K}_j$ are chosen so that they give maximal contribution.

In analogy with the proofs of Theorems~\ref{MRAPTHM} and~\ref{MRAPHybrid}, for each $j\leq J$ we define\footnote{We only need to split the $\chi$ spectrum into two sets here, as opposed to many sets in the proof of Theorem~\ref{MRAPTHM}. This is owing to the fact that $P_{u_j}(\chi)$ already has length $\gg q^{\varepsilon}$, and thus our large values estimates for it are effective. The reason we are allowed to take $P_{u_j}(\chi)$ so long here (unlike in our previous proofs) is that we are assuming $q\leq x^{1/2-100\eta}$. If we only assumed that $q=o(x^{1/2})$, we would have to perform an iterative decomposition as in the preceding sections.} the sets $\mc{X}^{(j)}$ and $\mc{U}^{(j)}$ by
\begin{align*}
\mc{X}^{(j)} &\coloneqq \{\chi \neq \chi_1\pmod q: |P_{u_j}(\overline{\chi})| \leq \varepsilon^{3}e^{u_j/H}/u_j\} \\
\mc{U}^{(j)} &\coloneqq \{\chi\neq \chi_1\pmod q\} \bk \mc{X}^{(j)}.
\end{align*}
\subsubsection{Case of $\mc{X}^{(j)}$} For a given $0 \leq j \leq J$, consider the contribution from $\mc{X}^{(j)}$. Applying Cauchy--Schwarz, we have
\begin{align*}
&\frac{1}{\phi(q)} \sum_{\chi \in \mc{X}^{(j)}}  |P_{u_j}(\overline{\chi})||D_{v_j}(\overline{\chi})||M_{u_j,v_j}(\overline{\chi})| \\
&\leq \Big(\frac{1}{\phi(q)} \sum_{\chi \in \mc{X}^{(j)}}  |M_{u_j,v_j}(\chi)|^2\Big)^{1/2} \Big(\frac{1}{\phi(q)} \sum_{\chi \in \mc{X}^{(j)}}  |P_{u_j}(\overline{\chi})|^2|D_{v_j}(\overline{\chi})|^2\Big)^{1/2}.
\end{align*}

We begin by bounding the first bracketed sum. We do not use Lemma~\ref{SIMPLEORTHO} directly for this, since that would lose one factor of $\log x$ that comes from the sparsity of the $m$ variable in the definition of $M_{u_j,v_j}(\chi)$. Instead, we expand the square and apply orthogonality, which shows that the first bracketed sum is bounded by
\begin{align*}
\Big(\sum_{\substack{m_1 \leq xe^{-(u_j+v_j)/H} \\ P^-(m_1) \geq x^{\theta_j}}} \, \sum_{\substack{m_2\leq xe^{-(u_j+v_j)/H}\\ P^{-}(m_2)\geq x^{\theta_{j}}\\m_2\equiv m_1\pmod q}}1\Big)^{1/2}.
\end{align*}
Taking the maximum over $m_1$, summing over $m_2$ conditioned to $m_2 \equiv m_1 \pmod{q}$, and applying Lemma~\ref{le_sel_sieve} (recalling that $xe^{-(u_j+v_j)/H}/q\gg x^{\eta}$), this is
\begin{align*}
\ll \frac{\eta^{-1}}{\theta_{j}\phi(q)^{1/2}\log x}xe^{-(u_j+v_j)/H}.  
\end{align*}

To treat the remaining bracketed expression, we use the pointwise bound from the definition of $\mc{X}^{(j)}$, and then use Lemma~\ref{SIMPLEORTHO} to bound $D_{v_j}(\overline{\chi})$, giving
\begin{align}
&\Big(\frac{1}{\phi(q)} \sum_{\chi \in \mc{X}^{(j)}} |P_{u_j}(\overline{\chi})|^2|D_{v_j}(\overline{\chi})|^2 \Big)^{1/2} \ll \Big(\frac{\varepsilon^{6}}{\phi(q)} e^{2u_j/H}/u_j^2\sum_{\chi \in \mc{X}^{(j)}} |D_{v_j}(\overline{\chi})|^2\Big)^{1/2} \nonumber\\
&\ll  \Big(\frac{\varepsilon^{6}}{\phi(q)}e^{2u_j/H}/u_j^2 \Big(\phi(q)+ \frac{\phi(q)}{q}e^{v_j/H}\Big) \Big(\Psi_q(e^{(v_j+1)/H},x^{\theta_{j+1}}) - \Psi_q(e^{v_j/H},x^{\theta_{j+1}})\Big)\Big)^{1/2}. \label{psi_diff}
\end{align}

By Lemma~\ref{le_qsmooth_short},
$$\Psi_q(e^{(v_j+1)/H},x^{\theta_{j+1}}) - \Psi_q(e^{v_j/H},x^{\theta_{j+1}}) \ll \rho(1/(3\theta_{j+1}))\frac{\phi(q)}{q}e^{v_j/H}/H.$$
Inserting this into~\eqref{psi_diff}, and using $e^{v_j/H}/q \geq 1$ for any $v_j \in \mc{K}_j$, results in the bound
\begin{align*}
\ll  \varepsilon^{3}\Big(\frac{\phi(q)}{q^2}\rho(1/(3\theta_{j+1}))H^{-1}\Big)^{1/2}e^{(u_j+v_j)/H}/u_j.
\end{align*}
Combining this with the contribution from $M_{u_j,v_j}(\chi)$ yields the upper bound
\begin{align*}
&\frac{1}{\phi(q)}\sum_{\chi \in \mc{X}^{(j)}} |P_{u_j}(\overline{\chi})||D_{v_j}(\overline{\chi})||M_{u_j,v_j}(\overline{\chi})| \\
&\ll_{\eta} \varepsilon^{3}H^{-1/2}\frac{\rho(1/(3\theta_{j+1}))^{1/2}}{\theta_{j+1}}\frac{x}{u_j(\log x)q}.
\end{align*}
Recalling $H= \lfloor \e^{-1.1}\rfloor$, $\rho(u)\ll u^{-2}$ and $\theta_{j}-\theta_{j+1}\ll \varepsilon^2$, when inserted into \eqref{eq:trilinUVJ} this expression yields
\begin{align*}\ll_{\eta} \varepsilon^{3}(\theta_{j}-\theta_{j+1})\theta_{j+1} H^2(\log x)^2\cdot H^{-1/2}\rho(1/(3\theta_{j+1}))^{1/2}\frac{x}{H\theta_{j+1}(\log x)^2 q} \ll_{\eta} \e^{4.1}\frac{x}{q}.
\end{align*}
Finally summing this over $0\leq j\leq J$, the bound we obtain is
\begin{align*}
\ll_{\eta} J\varepsilon^{4.1} \frac{x}{q}\ll_{\eta}\varepsilon^{2}\frac{x}{q},    
\end{align*}
which is good enough.

\subsubsection{Case of $\mc{U}^{(j)}$} 
It remains to consider the contributions from $\mc{U}^{(j)}$. We restrict to $q \in \mc{Q}_{x,\varepsilon^{9.5},\varepsilon^{-100}}$ with the notation of Lemma~\ref{le_Qset}. As in the proof of Theorem~\ref{MRAPTHM}, that set satisfies the desired size bound $|[1,Q]\setminus \mathcal{Q}_{x,\varepsilon^{9.5}, \varepsilon^{-100}}|\ll Qx^{-\varepsilon^{200}}$ (since $9.5\cdot 20<200$), and for any set $\mathcal{Q}'\subset [1,x]$ of coprime integers the set $\mathcal{Q}_{x,\varepsilon^{9.5}, \varepsilon^{-100}}$ intersects it in $\ll (\log x)^{\varepsilon^{-200}}$ points (and under GRH we have $\mathcal{Q}_{x,\varepsilon^{9.5},\varepsilon^{-100}}=[1,x]\cap \mathbb{Z}$). We also recall that in Theorem~\ref{BVTHM} the character $\chi_1\pmod q$ is such that $\inf_{|t|\leq \log x}\mathbb{D}_q(f,\chi_j(n)n^{it};x)$ is minimal.

By Proposition~\ref{prop_largevalues_hyb} (with $\delta\coloneqq e^{1/H}-1\asymp1/H$), for $q$ as above we have $|\mathcal{U}^{(j)}|\ll \varepsilon^{-6}H\ll \varepsilon^{-7.1}$, since  $P_{u}(\chi)$ has length $\gg x^{\theta_{J}}$ and $\theta_J\gg_{\eta} 1/(\log \frac{1}{\varepsilon})$. 

Furthermore, applying Proposition~\ref{prop_sup_hyb} (and Remark~\ref{rem_sup}) to $f(n)1_{P^{+}(n)\leq x^{\theta_j}}$ (and recalling $q \in \mc{Q}_{x,\varepsilon^{9.5},\varepsilon^{-100}}$), we see that\footnote{Note that the saving of $\varepsilon^{9.5}$ is much better than the trivial saving (which we do not need to exploit here) that comes from the fact that $d$ is supported on $x^{\theta_j}$-smooth numbers. The trivial saving would only be better if $\theta_j$ is roughly of size $1/\log(1/\varepsilon)$ or smaller, but as we shall see the contribution of these large values of the index $j$ is small in any case by trivial estimation.} 
\begin{align}\label{eqq115}
|D_{v_j}(\overline{\chi})|=\Big|\sum_{e^{v_j/H}\leq d\leq e^{(v_j+1)/H}}f(d)\overline{\chi}(d)1_{P^{+}(d)\leq x^{\theta_{j+1}}}\Big|\ll \varepsilon^{9.5}\frac{\phi(q)}{q}e^{v_j/H}   \end{align}
for all $\chi\in \mathcal{U}^{(j)}$, except possibly for the $\chi=\chi^{(j)}$ that minimizes the pretentious distance $\inf_{|t|\leq \log x}\mathbb{D}_q(f,\chi(n)1_{P^{+}(n)\leq x^{\theta_{j+1}}}n^{it};x)$. We argue that $\chi^{(j)}$ must be the character $\chi_1$ of Theorem~\ref{BVTHM}, in which case $\chi^{(j)}\not \in \mathcal{U}^{(j)}$ and we can ignore this character.

By applying Lemma~\ref{le_hal_hyb}, we see that either
\begin{align*}
\inf_{|t|\leq \log x}\mathbb{D}_q^2(f,\chi^{(j)}(n)1_{P^{+}(n)\leq x^{\theta_{j+1}}}n^{it};x)\leq 1.01\log\frac{1}{\varepsilon^{9.5}}+O(1)
\end{align*}
or else~\eqref{eqq115} holds without any exceptional characters. We may assume we are in the former case, and then by  $\theta_{j+1}\geq \theta_{J+1}\gg_{\eta} 1/(\log(1/\varepsilon))$ and trivial estimation we obtain
\begin{align*}
\inf_{|t|\leq \log x}\mathbb{D}_q^2(f,\chi^{(j)}(n)n^{it};x)\leq 1.1\log\frac{1}{\varepsilon^{9.5}}+O_{\eta}(1).    
\end{align*}
But we have the same for $\chi_1$ in place of $\chi^{(j)}$ by the minimality of $\chi_1$. Thus, assuming that $\chi^{(j)}\neq \chi_1$ and applying the pretentious triangle inequality as in the proof of Proposition~\ref{prop_sup_hyb} (using also that $\varepsilon > (\log x)^{-1/200}$), we obtain a contradiction. This means that we may assume from now on that~\eqref{eqq115} holds for all $\chi\in \mathcal{U}^{(j)}$ and $0\leq j\leq J$.

Now we take the maximum over $\chi\in \mc{U}^{(j)}$ in the sum that we are considering and apply the Brun--Titchmarsh inequality to $P_{u_j}(\overline{\chi})$ and Lemma~\ref{le_sel_sieve}  to $M_{u_j,v_j}(\overline{\chi})$ to bound 
\begin{align*}
&\frac{1}{\phi(q)} \sum_{\chi \in \mc{U}^{(j)}}  |P_{u_j}(\overline{\chi})||D_{v_j}(\overline{\chi})||M_{u_j,v_j}(\overline{\chi})| \\
&\ll e^{u_j/H}/(u_j/H) \cdot |\mc{U}^{(j)}|\max_{\chi \in \mc{U}^{(j)}} |D_{v_j}(\bar{\chi})|\cdot |M_{u_j,v_j}(\bar{\chi})| \\
&\ll e^{u_j/H}/(u_j/H)\cdot \varepsilon^{-7.1} \cdot \frac{\varepsilon^{9.5}e^{v_j/H}}{q} \cdot xe^{-(u_j+v_j)/H}\frac{\eta^{-1}}{H\theta_{j+1}\log x}\\
&\ll_{\eta}\varepsilon^{2.4}\frac{x}{q\theta_{j+1}^2H(\log x)^2},
\end{align*}
and this multiplied by $(\theta_j-\theta_{j+1})H^2(\log x)^2$ and summed over $0\leq j\leq J$ (recalling that $\theta_J\gg_{\eta}1/\log (1/\varepsilon)$) produces the bound
\begin{align*}
\ll_{\eta} \varepsilon^{2.4}\left(\log \frac{1}{\varepsilon}\right)^2H \sum_{0\leq j\leq J}(\theta_j-\theta_{j+1})\frac{x}{q}\ll_{\eta}\varepsilon^{1.2} \frac{x}{q}.   
\end{align*}
This completes the proof of Theorem~\ref{BVTHM}.

\section{A Linnik-type result}

\label{sec: linnik}

In this section, we prove our Linnik-type theorems stated in Section~\ref{section_apps}. As in the proof of Theorem~\ref{MRAPTHM}, we employ the Matom\"aki--Radziwi\l{}\l{} method in arithmetic progressions. 

Our main propositions in this section concern products of exactly three primes of the form
\begin{align}\label{eqq49a}
E_3^{*}\coloneqq  \{n=p_1p_2p_3:\,\, P_j^{1-\varepsilon}\leq p_j\leq P_j,\,\, j\in \{1,2,3\}\},\quad P_1=q^{1000\varepsilon},  P_2=P_3=q.    
\end{align}

\begin{prop}[$E_3^{*}$ numbers in progressions to smooth moduli] \label{prop_linnik1} For every small enough $\varepsilon>0$ there exists $\eta(\varepsilon)>0$ such that the following holds. 

Let $q\geq 2$ with $P^{+}(q)\leq q^{\eta(\varepsilon)}$. There exists a real character $\xi\pmod q$ such that for all $a$ coprime to $q$ we have
\begin{align}\label{eqq50}\begin{split}
\sum_{\substack{n\in E_3^{*}\\n\equiv a \pmod q}}\frac{1}{n}&=\frac{1+O(\varepsilon)}{\phi(q)}\sum_{P_1^{1-\varepsilon}\leq p_1\leq P_1}\sum_{P_2^{1-\varepsilon}\leq p_2\leq P_2}\sum_{P_3^{1-\varepsilon}\leq p_3\leq P_3}\frac{1}{p_1p_2p_3}\\
&+\frac{\xi(a)}{\varphi(q)}\sum_{P_1^{1-\varepsilon}\leq p_1\leq P_1}\sum_{P_2^{1-\varepsilon}\leq p_2\leq P_2}\sum_{P_3^{1-\varepsilon}\leq p_3\leq P_3}\frac{\xi(p_1p_2p_3)}{p_1p_2p_3}  
\end{split}
\end{align}
with $E_3^{*}$, $P_1, P_2, P_3$ as in~\eqref{eqq49a}. 
\end{prop}

\begin{prop}[$E_3^{*}$ numbers in progressions to prime moduli] \label{prop_linnik2} For every small enough $\varepsilon>0$ there exists $M(\varepsilon)\geq 1$ such that the following holds. 

Let $q\geq 2$. Suppose that the product $\prod_{\chi\pmod q}L(s,\chi)$ has the zero-free region $\Re(s)\geq 1-\frac{M(\varepsilon)}{\log q}$, $|\Im(s)|\leq (\log q)^3$. Then for all $a$ coprime to $q$ we have
\begin{align}\label{eqq61}
\sum_{\substack{n\in E_3^{*}\\n\equiv a \pmod q}}\frac{1}{n}&=\frac{1+O(\varepsilon)}{\phi(q)}\sum_{P_1^{1-\varepsilon}\leq p_1\leq P_1}\sum_{P_2^{1-\varepsilon}\leq p_2\leq P_2}\sum_{P_3^{1-\varepsilon}\leq p_3\leq P_3}\frac{1}{p_1p_2p_3}
\end{align}
with $E_3^{*}$, $P_1, P_2, P_3$ as in~\eqref{eqq49a}.
\end{prop}

We shall deduce Theorem~\ref{theo_linnik1}(i)--(ii)  from these two propositions at the end of the section. 

\subsection{Auxiliary lemmas}
In order to prove these propositions, we shall need a result of Chang~\cite[Theorem 10]{chang}, giving an improved zero-free region for $L(s,\chi)$ when the conductor of $\chi$ is smooth.

\begin{lemma}[Zero-free region for $L$-functions to smooth moduli]\label{Chang_zero1}
Suppose $q\geq 2$ and $P^+(q)\le q^{\kappa}$ with $C/(\log \log(10q))<\kappa<0.001$ for large enough $C>0$. Then the product $\prod_{\chi\pmod q}L(s,\chi)$ obeys the zero-free region
\begin{align*}
 \Re(s)\geq 1-\frac{c\kappa^{-1}}{\log q},\quad |\Im(s)|\leq q 
\end{align*}
for some constant $c>0$, apart from possibly a single zero $\beta$. If $\beta$ exists, then it is real and simple and corresponds to a unique real character $\hspace{-0.1cm}\pmod q$.
\end{lemma}

\begin{proof}
This follows from work of Chang~\cite[Theorem 10]{chang} (improving on work of Iwaniec~\cite{iwaniec-smooth}). Indeed, that theorem shows that, apart from possibly one real, simple zero corresponding to a real, non-principal character, there are constants $c,c' > 0$ such that $L(s,\chi)$ has the zero-free region 
\begin{align*}
\Re(s)>1-c\min\left\{\frac{1}{\log P^{+}(q)},\frac{\log \log d'}{(\log d')\log(2\frac{\log d}{\log d'})},\frac{1}{(\log (dT))^{1-c'}}\right\},\quad |\Im(s)|\leq T,    
\end{align*}
where $d$ is the conductor of $\chi$
and $d'=\prod_{p\mid d} p$. We take $T=q$ and note that the middle term in the minimum is $\gg \frac{\log \log d}{\log d}\geq \frac{\log \log q}{\log q}$, and this produces the zero-free region of the lemma.
\end{proof}

We will also need the following mean value estimate for sums over small sets of characters.
\begin{lemma}[Hal\'asz--Montgomery type estimate over primes]\label{le_sp} Let $q\geq 1$ be an integer, and let $\Xi$ be a set of characters $\hspace{-0.1cm}\pmod q$. Then for $k\in \{2,3\}$, $\eta>0$, $2\leq R<\sqrt{N}$, and for any complex numbers $a_p$, we have the estimate
\begin{align*}
\sum_{\chi\in \Xi}\Big|\sum_{p\leq N}a_p\chi(p)\Big|^2\ll_{k,\eta}  \Big(\frac{N}{\log R}+N^{1-1/k}q^{(k+1)/4k^2+\eta}|\Xi|R^{2/k}\Big)\sum_{p\leq N}|a_p|^2.  
\end{align*}

\end{lemma}

\begin{proof}
This is a result of Puchta~\cite[Theorem 3]{sch-puc}.
\end{proof}

In the proof of Theorem~\ref{theo_linnik1}(i)--(ii), we will need pointwise estimates for logarithmically weighted character sums assuming only a narrow zero-free region. By a simple Perron's formula argument, we can obtain cancellation in 
\begin{align*}
\sum_{P\leq p\leq P^{1+\kappa}}\frac{\chi(p)}{p}    
\end{align*}
for $\chi\neq \chi_0\pmod q$, $\kappa>0$ fixed, and $P\in [ q^{\varepsilon},q]$ if we assume a zero-free region of the form $\Re(s)>1-3\frac{\log \log q}{\log P}$, $|\Im(s)|\leq q$ for $L(s,\chi)$; the need for this zero-free region comes from pointwise estimation of $|\frac{L'}{L}(s,\chi)|\ll \log^2(q(|t|+2))$ which costs us two logarithms (in the region where we are $\gg \frac{1}{\log(q(|t|+1))}$ away from any zeros). However, we must argue differently, since we are only willing to assume a zero-free region of the form $\Re(s)>1-\frac{M(\varepsilon)}{\log P}$, $|\Im(s)|\leq q$ (which we know for smooth moduli apart from Siegel zeros). To do so, we exploit the logarithmic weight $1/p$  in the sum over $P\leq p\leq P^{1+\kappa}$, which allows us to insert a carefully chosen smoothing. A variant of such an argument is known as a Rodosskii bound in the literature.

\begin{lemma}[A Rodosskii-type bound]\label{le_rodosskii} Let $q\geq 2$, $\varepsilon>0$, $\kappa>0$, and let $\chi\pmod q$ be a non-principal character. Suppose that $L(s,\chi)\neq 0$ for $\Re(s)>1-\frac{\kappa^{-2}}{\log q}$, $|\Im(s)|\leq (\log q)^3$. Then, provided that $P\geq q^{\kappa}\gg_{\kappa} 1$, we have
\begin{align}\label{form15}
\sup_{|t|\leq (\log q)^3/2}\Big|\sum_{P\leq p\leq P^{1+\varepsilon}}\frac{\chi(p)}{p^{1+it}}\Big|\leq C_0 \kappa    
\end{align}
with $C_0>0$ an absolute constant.
\end{lemma}

\begin{proof}

This is a slight modification of results proved by Soundararajan~\cite[Lemma 4.2]{sound} and by Harper~\cite[Rodosskii Bound 1]{harper-smooth}; in those bounds there is the nonnegative function $(1-\Re(\chi(p)p^{-it}))/p$ in place of $\chi(p)/p^{1+it}$ in~\eqref{form15}, and consequently only lower bounds of the correct order of magnitude are needed in those results. We will choose a more elaborate smoothing to obtain asymptotics (up to $O(\kappa)$) for~\eqref{form15}. Also note that our range of $|t|$ is smaller than in the works mentioned above, but correspondingly the zero-free region is assumed to a lower height.

We may assume without loss of generality that $\kappa< \varepsilon/10<1/10$, since otherwise the trivial Mertens bound for~\eqref{form15} is good enough. By splitting the interval $[P,P^{1+\varepsilon}]$ into $\ll \varepsilon/\kappa$ intervals of the form $[y,y^{1+\kappa}]$ (and possibly one additional interval), it suffices to show that
\begin{align*}
\sup_{|t|\leq (\log q)^3/2}\Big|\sum_{y\leq p\leq y^{1+\kappa}}\frac{\chi(p)\log p}{p^{1+it}}\Big|\ll \kappa^2 \log y 
\end{align*}
uniformly for $y\in [P,P^2]$.

We introduce the continuous, nonnegative weight function
\begin{align*}
g(u)=\begin{cases}\kappa^{-2}u,&u\in [0,\kappa^2]\\
1,&u\in [\kappa^2, \kappa-\kappa^2]\\
\kappa^{-2}(\kappa-u), &u\in [\kappa-\kappa^2,\kappa],\\
0,&u\not \in [0,\kappa];
\end{cases}    
\end{align*}
in other words, $g$ is a trapezoid function. We further define the weight function
\begin{align*}
W(p)=W_{y,\kappa}(p)=g\Big(\frac{\log\frac{p}{y}}{\log y}\Big)\log y.
\end{align*}
Since $W(p)=\log y$ for $p\in [y^{1+\kappa^2},y^{1+\kappa-\kappa^2}]$, and $0\leq W(p)\leq \log y$ everywhere, by estimating the contribution of $p\in [y,y^{1+\kappa^2}]\cup [y^{1+\kappa-\kappa^2},y^{1+\kappa}]$ trivially, it suffices to show that
\begin{align}\label{eqq49}
\sup_{|t|\leq (\log q)^3/2}\Big|\sum_{p}\frac{\chi(p)W(p)\log p}{p^{1+it}}\Big|\ll \kappa^2 \log^2 y.    
\end{align}

Let $\chi^{*}$ be the primitive character that induces $\chi$. Since the contribution of $p\mid q$ to the sum in~\eqref{eqq49} is negligible, and since we can replace $\log p$ with the von Mangoldt function, from Perron's formula we see that 
\begin{align*}
\sum_{p}\frac{\chi(p)W(p)\log p}{p^{1+it}}=-\frac{1}{2\pi i}
\int_{-i\infty}^{i\infty}\frac{L'}{L}(1+it+s,\chi^{*})\widetilde{W}(s)\, ds+O(\kappa^2\log y),
\end{align*}
where
\begin{align*}
\widetilde{W}(s)\coloneqq  \int_{0}^{\infty}W(x)x^{s-1}\, dx=\kappa^{-2}\frac{y^{(1+\kappa)s}-y^{(1+\kappa-\kappa^2)s}-y^{(1+\kappa^2)s}+y^s}{s^2}    
\end{align*}
is the Mellin transform of $W$.

Shifting the contours to the left, and noting that $\widetilde{W}(s)$ is entire and $|\widetilde{W}(s)|\ll \frac{\kappa^{-2}}{|s|^2}$ for $\Re(s)\leq 0$, we reach
\begin{align}\label{eqq51}
\sum_{p}\frac{\chi(p)W(p)\log p}{p^{1+it}}=-\sum_{\rho}
\widetilde{W}(\rho-1-it)+O(\kappa^2\log y),
\end{align}
where the sum is taken over all nontrivial zeros of $L(s,\chi^*)$. Since $|t|\leq \frac{(\log q)^3}{2}$, we can truncate the $\rho$ sum to end up with 
\begin{align*}
\sum_{p}\frac{\chi(p)W(p)\log p}{p^{1+it}}=-\sum_{|\Im(\rho)|\le (\log q)^3}
\widetilde{W}(\rho-1-it)+O(\kappa^2\log y).
\end{align*}

Let $A\coloneqq  \kappa^{-2}$. Thanks to our assumption on zero-free regions, we clearly have
\begin{align*}|\widetilde{W}(\rho-1-it)|\ll \frac{ \kappa^{-2}y^{-\frac{A}{\log q}}}{|\rho-1-it|^2},
\end{align*}
and consequently
\begin{align*}
\Big|\sum_{p}\frac{\chi(p)W(p)\log p}{p^{1+it}}\Big|\ll \kappa^{-2} y^{-\frac{A}{\log q}}\sum_{|\Im(\rho)|\le (\log q)^3}\frac{1}{|1+it-\rho|^2}+\kappa^2 \log y.
\end{align*}

We now note that for any zero $\rho=\beta+i\gamma,$ with $|\gamma|\le (\log q)^3$ we must have $\beta\le 1-\frac{A}{\log q}$, and so
\[\frac{1}{|1+it-\rho|^2}\ll \frac{1}{|1+1/\log q+it-\rho|^2}\ll \frac{\log q}{A}\Re \Big(\frac{1}{1+1/\log q+it-\rho}\Big).\]
Thus we can estimate 
\begin{align*}\Big|\sum_{p}\frac{\chi(p)W(p)\log p}{p^{1+it}}\Big|\ll \kappa^{-2}y^{-\frac{A}{\log q}}\cdot\frac{\log q}{A}\sum_{\rho}\Re \Big(\frac{1}{1+1/\log q+it-\rho}\Big)+\kappa^2 \log y.
\end{align*}

Recall that $y\ge P\ge q^{1/\sqrt{A}}$. We can use the Hadamard factorization theorem in the form given in~\cite[Chapter 12]{Davenport} on the right-hand side of the above formula, and estimate $|\frac{L'}{L}(1+1/\log q + it,\chi)|\ll \log q$, to see that
\begin{align*}
\Big|\sum_{p}\frac{\chi(p)W(p)\log p}{p^{1+it}}\Big|\ll \kappa^{-2} e^{-\sqrt{A}}A(\log y)^2+\kappa^2\log y\ll \kappa^2 \log^2 y\end{align*}
by our choice of $A$. This finishes the proof of the lemma. 
\end{proof}

\subsection{Proof of Propositions~\ref{prop_linnik1} and~\ref{prop_linnik2}}
\begin{proof}[Proof of Proposition~\ref{prop_linnik1}.] We may assume that $\varepsilon>0$ is small enough and that $q$ is large enough in terms of $\varepsilon$, since we must have $q^{\varepsilon'}\geq 2$, and we are free to choose the dependence of $\varepsilon'$ on $\varepsilon$. We shall show that if $q$ is such that we have the zero-free region $L(s,\chi)\neq 0$ for $\Re(s)\geq 1-\varepsilon^{-100}/\log q$, $|\Im(s)|\leq (\log q)^3$ for all $\chi\pmod q$ apart from possibly one real character $\xi$, then~\eqref{eqq50} holds\footnote{If this bad $\xi$ does not exist, let $\xi$ be any non-principal real character in what follows.}. This zero-free region is in particular satisfied for those $q$ that satisfy $P^{+}(q)\leq q^{\eta(\varepsilon)}$ with small enough $\eta(\varepsilon)>0$.

By the orthogonality of characters, we have
\begin{align*}
\sum_{\substack{n\in E_3^{*}\\n\equiv a \pmod q}}\frac{1}{n}=\sum_{\chi\in \{\chi_0,\xi\}\pmod q}\frac{\chi(a)}{\phi(q)}P_1(\overline{\chi})P_2(\overline{\chi})P_3(\overline{\chi})+\sum_{\substack{\chi\pmod q\\\chi\neq \chi_0,\xi}}\frac{\chi(a)}{\phi(q)}P_1(\overline{\chi})P_2(\overline{\chi})P_3(\overline{\chi}), \end{align*}
where we have defined
\begin{align*}
P_j(\chi)\coloneqq  \sum_{P_j^{1-\varepsilon}\leq p\leq P_j}\frac{\chi(p)}{p},\quad  j\in \{1,2,3\}.
\end{align*}
In the above expression, in the term corresponding to $\chi_0$ we can replace $\chi_0$ by $1$ at the cost of $O((\log q)/q^{\varepsilon})$.

We employ the Matom\"aki--Radziwi\l{}\l{} method as in our other proofs. Let
\begin{align*}
\mathcal{X}:&=\{\chi\pmod q:\,\,\chi\not \in \{\chi_0,\xi\},\,\, |P_1(\overline{\chi})|\leq P_1^{-0.01}\},\\
\mathcal{U}_S:&=\{\chi\pmod q:\,\, \chi\not \in \{ \chi_0,\xi\}\}\setminus \mathcal{X}.
\end{align*}   
Unlike in the earlier sections, there is no $\mathcal{U}_L$ case to analyze, owing to the fact that for $\chi\in \mathcal{U}_S$ we already have some cancellation in $|P_1(\chi)|$ by Lemma~\ref{le_rodosskii} and our assumption on $q$.

The case of $\mc{X}$ is handled similarly to our other proofs. Indeed, by Cauchy--Schwarz, we have
\begin{align*}
\sum_{\chi\in \mc{X}}|P_1(\overline{\chi})||P_2(\overline{\chi})||P_3(\overline{\chi})|\ll P_1^{-0.01}\Big(\sum_{\chi\in \mc{X}}|P_2(\overline{\chi})|^2\Big)^{1/2}  \Big(\sum_{\chi\in \mc{X}}|P_3(\overline{\chi})|^2\Big)^{1/2}.   
\end{align*}
By the mean value theorem for character sums (Lemma~\ref{SIMPLEORTHO}) and the fact that $P_1=q^{1000\varepsilon}$, $P_2=P_3=q$, this is
\begin{align*}
\ll q^{-10\varepsilon}\phi(q)\Big(\sum_{P_2^{1-\varepsilon}\leq p_2\leq P_2}\frac{1}{p_2^2}\Big)^{1/2}\Big(\sum_{P_3^{1-\varepsilon}\leq p_3\leq P_3}\frac{1}{p_3^2}\Big)^{1/2}\ll q^{-\varepsilon},    
\end{align*}
say, since $\phi(q)/(P_2P_3)^{\frac{1}{2}(1-\varepsilon)}\ll q^{\varepsilon}.$ 

The remaining case to consider is that of $\mc{U}_S$. Note that, combining the assumed zero-free region for $L(s,\chi)$, $\chi\neq \xi\pmod q$ with Lemma~\ref{le_rodosskii}, we see that $|P_1(\chi)|\ll \varepsilon^2$ for all $\chi\in \mathcal{U}_S$.

We first estimate $|\mc{U}_S|$. For each $\chi \in \mc{U}_S$ we decompose $P_1(\bar{\chi})$ into dyadic segments $[y,2y]$ with $y = 2^j \in [P_1^{1-\varepsilon}/2,P_1]$ and use partial summation to obtain
$$
|P_1(\bar{\chi})| \ll (\varepsilon \log P_1) \max_{\substack{y \in [P_1^{1-\varepsilon}/2, P_1] \\ y = 2^j}} \frac{1}{y} \Bigg|\sum_{y \leq p \leq 2y} \chi(p)\Bigg|.
$$
From Lemma~\ref{PRIMMVTwithT}, which bounds the number of large values taken by a prime-supported character sum, we have the size bound
\begin{align*}
|\mc{U}_S| &\leq |\{\chi\pmod q:\,\, |P_1(\overline{\chi})|>P_1^{-0.01}\}| \\
&\leq \Bigg|\Bigg\{\chi \pmod{q} : \, \, \max_{\substack{y \in [P_1^{1-\varepsilon}, P_1] \\ y = 2^j}} \frac{1}{y} \Bigg|\sum_{y \leq p \leq 2y} \chi(p)\Bigg| \geq P_1^{-0.01}(\e \log P_1)^{-1}\Bigg\} \Bigg| \\ 
&\ll \sum_{\substack{P_1^{1-\varepsilon}/2 \leq y \leq P_1 \\ y = 2^j}} |\{\chi\pmod{q} : \, \, |\sum_{y \leq p \leq 2y} \chi(p)| \geq y^{1-0.02}\}| \ll q^{0.05},  
\end{align*}
recalling that $P_1 = q^{1000\varepsilon}$, and $q$ is sufficiently large in terms of $\varepsilon$.

Introducing the dyadic sums 
\begin{align*}
P_{j,v}(\chi)\coloneqq  \sum_{\substack{e^v\leq p\leq e^{v+1}\\P_j^{1-\varepsilon}\leq p\leq P_j}}\frac{\chi(p)}{p}, \quad v\in I_j\coloneqq  [(1-\varepsilon)\log P_j,\log P_j],   
\end{align*}
the upper bound on $|P_1(\chi)|$ above and Cauchy--Schwarz give
\begin{align*}
\sum_{\chi\in \mc{U}_S}|P_1(\overline{\chi})||P_2(\overline{\chi})||P_3(\overline{\chi})|&\ll \varepsilon^2\sum_{v_1,v_2\in I_2}\sum_{\chi\in \mc{U}_S}|P_{2,v_1}(\overline{\chi})||P_{3,v_2}(\overline{\chi})|\\
&\ll  \varepsilon^2 (\varepsilon\log q)^2 \Big(\sum_{\chi\in \mc{U}_S}|P_{2,v'_1}(\overline{\chi})|^2\Big)^{1/2}  \Big(\sum_{\chi\in \mc{U}_S}|P_{3,v'_2}(\overline{\chi})|^2\Big)^{1/2}   
\end{align*}
for some $v'_1,v'_2\in I_2$ (since as $P_2 = P_3$ we have $I_2 = I_3$). It remains to be shown that for any $v \in I_j$,
\begin{align*}
\sum_{\chi\in \mc{U}_S}|P_{j,v}(\overline{\chi})|^2\ll \frac{1}{\log^2 q}   
\end{align*}
for $j\in \{2,3\}$, since then we get a bound of $\ll \varepsilon^4 $ for the sum over $\chi\in \mc{U}_S$, and this (multiplied by the $1/\phi(q)$ factor) can be included in the error term in~\eqref{eqq50}.

For this purpose, we apply Lemma~\ref{le_sp}, which is a sharp inequality of Hal\'asz--Montgomery-type for character sums over primes\footnote{For this estimate to work, it is crucial that the character sums $P_{j,v}(\chi)$ are long enough in terms of $q$; in particular, we need them to have length $\gg q^{1/3+\varepsilon}$.}. We take $N=e^{v+1}$, $|\Xi|=|\mathcal{U}_S|\ll q^{0.05}$, $k=3$, $R=N^{0.0001}$, $a_p=\frac{1}{p}1_{p\in [e^v, e^{v+1}]\cap [P_j^{1-\varepsilon},P_j]}$ in that lemma. Since the term $N^{2/3}q^{1/9}|\Xi|R^{2/3}$ appearing in Lemma~\ref{le_sp} is smaller than the other term $\frac{N}{\log R}$ for our choice of parameters, we get a bound of $\ll e^v/v\cdot \frac{1}{ve^v}\ll \frac{1}{\log^2 q}$, as desired. This completes the analysis of the $\mc{U}_S$ case, so  Proposition~\ref{prop_linnik1} follows.
\end{proof}

\begin{proof}[Proof of Proposition~\ref{prop_linnik2}] The proof of Proposition~\ref{prop_linnik2} is similar to that of Proposition~\ref{prop_linnik1}, except that there are no exceptional characters arising. The proof of~\eqref{eqq50} goes through for any $q$ for which $L(s,\chi)\neq 0$ whenever $\Re(s)>1-\varepsilon^{-100}/\log q$, $|\Im(s)|\leq (\log q)^3$ and $\chi\neq \xi\pmod q$. Moreover, since under the assumption of Proposition~\ref{prop_linnik2} the exceptional character $\xi$ does not exist (that is, the above holds for all $\chi\pmod q$), we can delete the term involving $\xi$ from~\eqref{eqq50}, giving~\eqref{eqq61}. This gives Proposition~\ref{prop_linnik2}.
\end{proof}

\subsection{Deductions of Linnik-type theorems}

Corollary~\ref{theo_schnirelmann} is a direct consequence of Theorem~\ref{theo_linnik1}(i) (by fixing $\varepsilon>0$ in its statement). Hence, it suffices to prove Theorem~\ref{theo_linnik1}(i)--(ii).

\begin{proof}[Proof of Theorem~\ref{theo_linnik1}(ii)] It suffices to show that for all but $\ll_{\varepsilon}1$ primes $q\in [Q^{1/2},Q]$ the right-hand side of~\eqref{eqq61} is $>0$; indeed, then the smallest $q$-smooth $E_3$ number in the progression $a\pmod q$ is $\leq q^{2+1000\varepsilon}$ (and since $\varepsilon>0$ is arbitrarily small, this is good enough). 

In view of Proposition~\ref{prop_linnik2}, it suffices to show that for all but $\ll_{\varepsilon} 1$ primes $q \in [Q^{1/2},Q]$, $\prod_{\chi\pmod q}L(s,\chi)$ obeys the zero-free region $\Re(s)\geq 1-\frac{M(\varepsilon)}{\log q}$, $|\Im(s)|\leq (\log q)^3$ required by that proposition.

Since $q$ is a prime, all the characters modulo $q$ apart from the principal one are primitive. Moreover, the zeros of the $L$-function corresponding to the principal character are the same as the zeros of the Riemann zeta function, so we have the Vinogradov--Korobov zero-free region for this $L$-function. It therefore suffices to consider the $L$-functions corresponding to primitive characters. By the log-free zero density estimate (Lemma~\ref{le_zerodensity}), we immediately see that $\prod_{\chi\pmod q}^{*} L(s,\chi)$ has the required zero-free region  for all but
$\ll \exp(100M(\varepsilon))$ prime
moduli $q\in [Q^{1/2},Q]$, so we have the claimed result. 
\end{proof}

\begin{proof}[Proof of Theorem~\ref{theo_linnik1}(i)] Fixing $\delta>0$, we will show that if $P^{+}(q)\leq q^{\varepsilon'}$ with $\varepsilon'$ very small in terms of $\delta$, then the least product of exactly three primes in every reduced residue class $a\pmod q$ is $\ll q^{2+\delta}$.

Let $\varepsilon>0$ be very small in terms of $\delta$. By Lemma~\ref{Chang_zero1}, we have the zero-free region required by Proposition~\ref{prop_linnik2} whenever $P^{+}(q)\leq q^{\eta(\varepsilon)}$ with $\eta(\varepsilon)>0$ small enough, apart from possibly a single zero $\beta$, which is real and simple and corresponds to a single real character $\hspace{-0.1cm}\pmod q$.

If this exceptional zero $\beta$ does not exist, then from Proposition~\ref{prop_linnik2} we obtain a positive lower bound for the left-hand side of~\eqref{eqq50}. Therefore, we can assume that $\beta$ exists. This is a  real zero of an $L$-function $\hspace{-0.1cm}\pmod q$, and we write the zero as $\beta=1-\frac{c}{\log q}$ with $c>0$. By a result of Heath-Brown~\cite[Corollary 2]{hb-siegel} on Linnik's theorem and Siegel zeros, if $c\leq c_0(\delta)$ for a suitably small function $c_0(\delta)$, then the least prime in any arithmetic progression $a\pmod q$ with $(a,q)=1$ is $\ll q^{2+\delta/2}$, and thus also the least $n\equiv a\pmod q$ with exactly three prime factors obeys the same bound (indeed, if $p_1,p_2\ll \log^2 q$ are chosen to be primes not dividing $q$ and $p\ll q^{2+\delta/2}$ is a prime $\equiv a(p_1p_2)^{-1}\pmod q$, then $p_1p_2p\ll q^{2+\delta}$ and $p_1p_2p\equiv a \pmod q$). Thus we have proved the theorem if $\beta\geq 1-\frac{c_0(\delta)}{\log q}$, so henceforth we will assume we are in the opposite case.

According to Proposition~\ref{prop_linnik1}, it suffices to show that 
\begin{align*}
\Big|\sum_{P_3^{1-\varepsilon}\leq p\leq P_3}\frac{\xi(p)}{p}\Big|\leq (1-\sqrt{\varepsilon})\sum_{P_3^{1-\varepsilon}\leq p\leq P_3}\frac{1}{p},    
\end{align*}
since then the left-hand side of~\eqref{eqq50} is $>0$ for $\varepsilon>0$ small enough.  

Following the exact same argument as in the proof of Lemma~\ref{le_rodosskii}, and introducing the same weight function $W=W_{y,\kappa}$ with $y\in [P_3^{1-\varepsilon},P_3]$ and $\kappa=\varepsilon^{10}$ (and using~\eqref{eqq51}), it is enough to show that 
\begin{align*}
\Big|\sum_{\rho}\widetilde{W}(\rho-1)\Big|\leq (1-10\sqrt{\varepsilon})(\log^2 y)\sum_{y^{1-\kappa}\leq p\leq y}\frac{1}{p},    
\end{align*}
where the sum is over the nontrivial zeros $\rho$ of $L(s,\xi)$. Just as in the proof of Lemma~\ref{le_rodosskii}, the contribution of all the zeros $\rho\neq \beta$ is $\ll \varepsilon(\log^2 y) \sum_{y^{1-\kappa}\leq p\leq y}\frac{1}{p}$ as long as $P^{+}(q)\leq q^{\eta_1(\varepsilon)}$ with $\eta_1(\varepsilon)$ small enough. It suffices to show, then, that
\begin{align}\label{eqq53}
|\widetilde{W}(\beta-1)|\leq  (1-11\sqrt{\varepsilon})(\log^2 y)\sum_{y^{1-\kappa}\leq p\leq y}\frac{1}{p}.   
\end{align}

We recall that $\beta\leq 1-\frac{c_0(\delta)}{2\log y}$, and denote \begin{align*}
F(u)\coloneqq  \widetilde{W}\Big(-\frac{u}{\log y}\Big)=\kappa^{-2}\frac{e^{-au}-e^{-bu}-e^{-cu}+e^{-u}}{u^2}\log^2 y,    
\end{align*} 
where $a=1+\kappa$, $b=1+\kappa-\kappa^2$, $c=1+\kappa^2$ and the value at $u=0$ is interpreted as the limit as $u\to 0$. We compute using L'H\^opital's rule that $\widetilde{W}(0)=F(0)=\kappa(1-\kappa)\log^2 y$, and differentiation shows that $F$ is decreasing, so $\widetilde{W}$ is increasing. Moreover, $F'$ is increasing and $F'(u)=(\kappa/2\cdot (-2+\kappa+\kappa^2)+O(\kappa u))\log^2 y$ for $|u|\leq 1$. Thus, by the mean value theorem applied to $F$ we have
\begin{align*}
\widetilde{W}(\beta-1)\leq \widetilde{W}\Big(-\frac{c_0(\delta)}{2\log y}\Big)=F\Big(\frac{c_0(\delta)}{2}\Big)\leq F(0)+\frac{c_0(\delta)}{2}F'\Big(\frac{c_0(\delta)}{2}\Big)\leq  \kappa\Big(1-\kappa-\frac{c_0(\delta)}{4}\Big)\log^2 y,    
\end{align*}
since $\delta>0$ is small. We further have $1-\kappa-c_0(\delta)/4\leq 1-100\sqrt{\varepsilon}$ if $\varepsilon>0$ (and hence $\kappa$) is small enough in terms of $\delta$, so that~\eqref{eqq53} holds by Mertens' bound. This completes the proof. 
\end{proof}

\begin{proof}[Proof of Proposition~\ref{theo_linnik_mobius}] The proof of Proposition~\ref{theo_linnik_mobius} follows along similar lines as those above, so we merely sketch it, indicating the required modifications. We outline the lower bound for $n$ with $\mu(n) = -1$; the corresponding estimate for $\mu(n) = +1$ is proved in the analogous way. 

When considering numbers $n$ with $\mu(n)=-1$, we restrict to those $n$ that belong to the set
\begin{align*}
\mc{S}\coloneqq  \{n\in \mathbb{N}:\,\, \Omega_{[P_j,Q_j]}(n)=1,\,\, j\in \{1,2\}\} \end{align*}
with $P_1=x^{\varepsilon/10}$, $Q_1=x^{\varepsilon/5}$, $P_2=x^{1/2-\varepsilon}$, $Q_2=x^{1/2-\varepsilon/2}$;
this introduces essentially the same factorization patterns for our $n$ as in the case of products of exactly three primes. By writing $1_{\mu(n)=-1}=\frac{1}{2}(\mu^2(n)-\mu(n))$, it suffices to bound
\begin{align*}
\sum_{\substack{n\leq x\\n\equiv a\pmod q}}\mu^2(n)1_{\mc{S}}(n)\gg \varepsilon \frac{x}{q},\quad \quad  \Big|\sum_{\substack{n\leq x\\n\equiv a\pmod q}}\mu(n)1_{\mc{S}}(n)\Big|\ll \varepsilon^2 \frac{x}{q}.   
\end{align*}

We concentrate on the latter bound (the former is similar but easier). Write $n=p_1p_2m$ with $p_j\in [P_j,Q_j]$ for $j = 1,2$, $m\leq \frac{x}{p_1p_2}$. As in the previous sections, we can easily get rid of the cross condition on the variables by splitting into short intervals, so applying orthogonality of characters it suffices to show that
\begin{align*}
\frac{1}{\varphi(q)}\sum_{\substack{\chi\pmod q\\\chi\neq \chi_0}}\Big|Q_{v_1,H}(\chi)Q_{v_2,H}(\chi)R_{v_1+v_2,H}(\chi)\Big|\ll \frac{\varepsilon^2 x}{H^3(\log Q_1)(\log Q_2)q},    
\end{align*}
uniformly for $v_i\in I_i$,
where we have defined
\begin{align*}
&Q_{v,H}(\chi)\coloneqq  \sum_{e^{v/H}\leq p<e^{(v+1)/H}}\chi(p),\quad R_{v,H}(\chi)\coloneqq  \sum_{m\leq x/e^{v/H}}\mu(m)\chi(m)1_{\mathcal{T}}(m),\\
&I_i=[H\log P_i,H\log Q_i],\quad H=\lfloor \varepsilon^{-3}\rfloor,    
\end{align*}
and $\mc{T}$ is the set of numbers coprime to all the primes in $[P_j,Q_j]$ for $j\in \{1,2\}$. We consider the cases 
\begin{align*}
\mc{X}:&=\{\chi\pmod q:\,\, |Q_{v_1,H}(\overline{\chi})|\leq e^{0.99v_1/H}\}\setminus \{\chi_0\}\\
\mc{U}_S:&=\{\chi\pmod q:\,\, |Q_{v_1,H}(\chi)|\leq \varepsilon^{20}e^{v_1/H}/v_1\}\setminus (\mc{X}\cup \{\chi_0\})\\
\mc{U}_L:&=\{\chi\pmod q\}\setminus (\mathcal{X}\cup \mathcal{U}_S\cup\{\chi_0\}).
\end{align*}

The case of $\mc{X}$ is easy and is handled just as in the proof of Proposition~\ref{prop_linnik1}. The case of $\mc{U}_S$ is also handled similarly as in that proposition, except that we also need a Hal\'asz--Montgomery estimate for $\sum_{\chi\in \mc{U}_S}|R_{v_1+v_2,H}(\chi)|^2$. This bound takes the same form as Lemma~\ref{le_sp}, but is proved simply by applying duality and the Burgess bound (since $R_{v_1+v_2,H}(\chi)$ is a sum over the integers rather than over the primes). 

Finally, the $\mc{U}_L$ set is small in the sense that $|\mc{U}_L|\ll \varepsilon^{-43}$ by Proposition~\ref{prop_largevalues_hyb} whenever we have the zero-free region
\begin{align}\label{eqq166}
\prod_{\chi\pmod q}L(s,\chi)\neq 0 \quad \text{ for }\quad \Re(s)>1-\frac{M(\varepsilon)}{\log q},\, |\Im(s)|\leq 3q   \end{align} 
with $M(\varepsilon)$ large enough. It thus suffices to prove that \begin{align*}
\sup_{\chi\neq \chi_0\pmod q}|R_{v_1+v_2,H}(\chi)|\ll \varepsilon^{60}\frac{\phi(q)}{q}xe^{-(v_1+v_2)/H},
\end{align*} 
and by Lemma~\ref{le_hal_hyb} this reduces to the bound 
\begin{align}\label{eqq60}
\sup_{\substack{\chi\pmod q\\\chi\neq \chi_0}}\inf_{|t|\leq (\log q)^3/2}\sum_{\substack{p\leq x\\p\nmid q}}\frac{1+\Re(\chi(p)p^{-it})}{p}\geq 61\log\frac{1}{\varepsilon}+O(1).    
\end{align}
At first, a direct application of Lemma~\ref{le_hal_hyb} reduces to proving~\eqref{eqq60} with $\chi(p)p^{-it}1_{\mathcal{S}}(p)$ in place of $\chi(p)p^{-it}$, but since $\log Q_j/\log P_j\ll 1$ by our choices, the contribution of those $p$ with $1_{\mathcal{S}}(p)\neq 1$ is negligible in~\eqref{eqq60}.

Restricting the sum in~\eqref{eqq60} to $p\in [x^{\kappa},x]$ with $\kappa=\varepsilon^{61}$, we indeed obtain~\eqref{eqq60} from Lemma~\ref{le_rodosskii} upon splitting $[x^{\kappa},x]$ into segments $$
[x^{\kappa_j},x^{\kappa_{j+1}}], 
 \text{ where } \kappa_j = \kappa(1+\varepsilon)^j \text{ and } 0 \leq j \leq \lceil \log(1/\kappa)/\log(1/\varepsilon)\rceil \ll 1,
 $$ 
 as long as we have the zero-free region~\eqref{eqq166}. This zero-free region is indeed available by Lemma~\ref{le_zerodensity} for all but $\ll_{\varepsilon}1$ primes $q\in [Q^{1/2},Q]$, as in the proof of Theorem \ref{theo_linnik1}(ii).
\end{proof}

\section*{Acknowledgments} We are grateful to Kannan Soundararajan for inspiring discussions and remarks, and in particular for his insistence that there should be a way to improve our main result. This indeed led to a stronger version of our main theorem. We also thank Claus Bauer, Andrew Granville, Kaisa Matom\"aki and Maksym Radziwi\l{}\l{} for useful comments and discussions.  

We thank the anonymous referees for their very helpful and valuable comments that improved the exposition of this paper and led to strengthened formulations of some of the results. 

The third author was supported by a Titchmarsh Fellowship of the University of Oxford, 
Academy of Finland grant no. 340098, and funding from European Union's Horizon
Europe research and innovation programme under Marie Sk\l{}odowska-Curie grant agreement No
101058904.

This project was initiated while the authors were visiting CRM in Montreal in spring 2018, and they would like to thank CRM for excellent working conditions during their visit there.

\bibliography{main1.bib}
\bibliographystyle{plain}

\end{document}